\documentclass[numbers=enddot,12pt,final,onecolumn,notitlepage]{scrartcl}%
\usepackage[headsepline,footsepline,manualmark]{scrlayer-scrpage}
\usepackage[all,cmtip]{xy}
\usepackage{amssymb}
\usepackage{amsmath}
\usepackage{amsthm}
\usepackage{framed}
\usepackage{comment}
\usepackage{color}
\usepackage{hyperref}
\usepackage[sc]{mathpazo}
\usepackage[T1]{fontenc}
\usepackage{tikz}
\usepackage{needspace}
\usepackage{tabls}
\providecommand{\U}[1]{\protect\rule{.1in}{.1in}}
\usetikzlibrary{arrows}
\newcounter{exer}
\newcounter{exera}
\theoremstyle{definition}
\newtheorem{theo}{Theorem}[section]
\newenvironment{theorem}[1][]
{\begin{theo}[#1]\begin{leftbar}}
{\end{leftbar}\end{theo}}
\newtheorem{lem}[theo]{Lemma}
\newenvironment{lemma}[1][]
{\begin{lem}[#1]\begin{leftbar}}
{\end{leftbar}\end{lem}}
\newtheorem{prop}[theo]{Proposition}

\newtheorem{defi}[theo]{Definition}
\newenvironment{definition}[1][]
{\begin{defi}[#1]\begin{leftbar}}
{\end{leftbar}\end{defi}}
\newtheorem{remk}[theo]{Remark}

\newtheorem{coro}[theo]{Corollary}
\newenvironment{corollary}[1][]
{\begin{coro}[#1]\begin{leftbar}}
{\end{leftbar}\end{coro}}
\newtheorem{conv}[theo]{Convention}
\newenvironment{convention}[1][]
{\begin{conv}[#1]\begin{leftbar}}
{\end{leftbar}\end{conv}}
\newtheorem{quest}[theo]{Question}
\newenvironment{question}[1][]
{\begin{quest}[#1]\begin{leftbar}}
{\end{leftbar}\end{quest}}
\newtheorem{warn}[theo]{Warning}
\newenvironment{warning}[1][]
{\begin{warn}[#1]\begin{leftbar}}
{\end{leftbar}\end{warn}}
\newtheorem{conj}[theo]{Conjecture}

\newtheorem{exam}[theo]{Example}
\newenvironment{example}[1][]
{\begin{exam}[#1]\begin{leftbar}}
{\end{leftbar}\end{exam}}
\newtheorem{exmp}[exer]{Exercise}

\newtheorem{exetwo}[exera]{Additional exercise}

\newenvironment{statement}{\begin{quote}}{\end{quote}}
\let\sumnonlimits\sum
\let\prodnonlimits\prod
\let\cupnonlimits\bigcup
\let\capnonlimits\bigcap
\renewcommand{\sum}{\sumnonlimits\limits}
\renewcommand{\prod}{\prodnonlimits\limits}
\renewcommand{\bigcup}{\cupnonlimits\limits}
\renewcommand{\bigcap}{\capnonlimits\limits}
\newenvironment{verlong}{}{}
\newenvironment{vershort}{}{}

\excludecomment{verlong}
\includecomment{vershort}
\excludecomment{noncompile}

\ihead{Noncommutative Abel-like identities}
\ohead{page \thepage}
\begin{document}

\title{Noncommutative Abel-like identities}
\author{Darij Grinberg}
\date{2017 (corrected April 14, 2026) }
\maketitle

\begin{abstract}
\textbf{Abstract.} We generalize the Abel--Hurwitz identities to an almost
entirely noncommutative setting. Namely, let $V$ be a finite set of size $n$,
and let $\mathbb{L}$ be any noncommutative ring. For each $s\in V$, let
$x_{s}\in\mathbb{L}$. Set $x\left(  S\right)  :=\sum_{s\in S}x_{s}$ for any
$S\subseteq V$. Let $X$ and $Y$ be two elements of $\mathbb{L}$ such that
$X+Y$ lies in the center of $\mathbb{L}$. Then, we show that%
\begin{align*}
&  \sum_{S\subseteq V}\left(  X+x\left(  S\right)  \right)  ^{\left\vert
S\right\vert }\left(  Y-x\left(  S\right)  \right)  ^{n-\left\vert
S\right\vert }=\sum_{\substack{i_{1},i_{2},\ldots,i_{k}\in V\text{ distinct}%
}}\left(  X+Y\right)  ^{n-k}x_{i_{1}}x_{i_{2}}\cdots x_{i_{k}};\\
&  \sum_{S\subseteq V}X\left(  X+x\left(  S\right)  \right)  ^{\left\vert
S\right\vert -1}\left(  Y-x\left(  S\right)  \right)  ^{n-\left\vert
S\right\vert }=\left(  X+Y\right)  ^{n};\\
&  \sum_{S\subseteq V}X\left(  X+x\left(  S\right)  \right)  ^{\left\vert
S\right\vert -1}\left(  Y-x\left(  S\right)  \right)  ^{n-\left\vert
S\right\vert -1}\left(  Y-x\left(  V\right)  \right)  =\left(  X+Y-x\left(
V\right)  \right)  \left(  X+Y\right)  ^{n-1}.
\end{align*}
(Negative powers are understood to be cancelled by other factors.)

\end{abstract}
\tableofcontents

\section{Introduction}

In this (self-contained) note, we are going to prove three identities that
hold in arbitrary unital noncommutative rings, and generalize some well-known
combinatorial identities (known as the \textit{Abel--Hurwitz identities}).

In their simplest and least general versions, the identities we are
generalizing are equalities between polynomials in $\mathbb{Z}\left[
X,Y,Z\right]  $; namely, they state that
\begin{align}
\sum_{k=0}^{n}\dbinom{n}{k}\left(  X+kZ\right)  ^{k}\left(  Y-kZ\right)
^{n-k}  &  =\sum_{k=0}^{n}\dfrac{n!}{k!}\left(  X+Y\right)  ^{k}%
Z^{n-k};\label{eq.abel.1}\\
\sum_{k=0}^{n}\dbinom{n}{k}X\left(  X+kZ\right)  ^{k-1}\left(  Y-kZ\right)
^{n-k}  &  =\left(  X+Y\right)  ^{n};\label{eq.abel.2}\\
\sum_{k=0}^{n}\dbinom{n}{k}X\left(  X+kZ\right)  ^{k-1}Y\left(  Y+\left(
n-k\right)  Z\right)  ^{n-k-1}  &  =\left(  X+Y\right)  \left(  X+Y+nZ\right)
^{n-1} \label{eq.abel.3}%
\end{align}
for every nonnegative integer $n$.\ \ \ \ \footnote{The pedantic reader will
have observed that two of these identities contain \textquotedblleft
fractional\textquotedblright\ terms like $X^{-1}$ and $Y^{-1}$ and thus should
be regarded as identities in the function field $\mathbb{Q}\left(
X,Y,Z\right)  $ rather than in the polynomial ring $\mathbb{Z}\left[
X,Y,Z\right]  $. However, this is a false alarm, because all these
\textquotedblleft fractional\textquotedblright\ terms are cancelled. For
example, the addend for $k=0$ in the sum on the left hand side of
(\ref{eq.abel.2}) contains the \textquotedblleft fractional\textquotedblright%
\ term $\left(  X+0Z\right)  ^{0-1}=X^{-1}$, but this term is cancelled by the
factor $X$ directly to its left. Similarly, all the other \textquotedblleft
fractional\textquotedblright\ terms disappear. Thus, all three identities are
actually identities in $\mathbb{Z}\left[  X,Y,Z\right]  $.} These identities
have a long history; for example, (\ref{eq.abel.2}) goes back to Abel
\cite{Abel26}, who observed that it is a generalization of the binomial
formula (obtained by specializing $Z$ to $0$). The equality (\ref{eq.abel.1})
is ascribed to Cauchy in Riordan's text \cite[\S 1.5, Cauchy's identity]%
{Riorda68} (at least in the specialization $Z=1$; but the general version can
be recovered from this specialization by dehomogenization). The equality
(\ref{eq.abel.3}) is also well-known in combinatorics, and tends to appear in
the context of tree enumeration (see, e.g., \cite[Theorem 2]{Grinbe17}) and of
umbral calculus (see, e.g., \cite[Section 2.6, Example 3]{Roman84}).

The identities (\ref{eq.abel.1}), (\ref{eq.abel.2}) and (\ref{eq.abel.3}) have
been generalized by various authors in different directions. The most famous
generalization is due to Hurwitz \cite{Hurwit02}, who replaced $Z$ by $n$
commuting indeterminates $Z_{1},Z_{2},\ldots,Z_{n}$. More precisely, the
equalities (IV), (II) and (III) in \cite{Hurwit02} say (in a more modern
language) that if $n$ is a nonnegative integer and $V$ denotes the set
$\left\{  1,2,\ldots,n\right\}  $, then%
\begin{equation}
\sum_{S\subseteq V}\left(  X+\sum_{s\in S}Z_{s}\right)  ^{\left\vert
S\right\vert }\left(  Y-\sum_{s\in S}Z_{s}\right)  ^{n-\left\vert S\right\vert
}=\sum_{\substack{i_{1},i_{2},\ldots,i_{k}\text{ are}\\\text{distinct}%
\\\text{elements of }V}}\left(  X+Y\right)  ^{n-k}Z_{i_{1}}Z_{i_{2}}\cdots
Z_{i_{k}};\ \ \ \ \ \label{eq.hurwitz.1}%
\end{equation}%
\begin{equation}
\sum_{S\subseteq V}X\left(  X+\sum_{s\in S}Z_{s}\right)  ^{\left\vert
S\right\vert -1}\left(  Y-\sum_{s\in S}Z_{s}\right)  ^{n-\left\vert
S\right\vert }=\left(  X+Y\right)  ^{n}; \label{eq.hurwitz.2}%
\end{equation}%
\begin{equation}
\sum_{S\subseteq V}X\left(  X+\sum_{s\in S}Z_{s}\right)  ^{\left\vert
S\right\vert -1}Y\left(  Y+\sum_{s\in V\setminus S}Z_{s}\right)
^{n-\left\vert S\right\vert -1}=\left(  X+Y\right)  \left(  X+Y+\sum_{s\in
V}Z_{s}\right)  ^{n-1} \label{eq.hurwitz.3}%
\end{equation}
in the polynomial ring $\mathbb{Z}\left[  X,Y,Z_{1},Z_{2},\ldots,Z_{n}\right]
$.\ \ \ \ \footnote{The sum on the right hand side of (\ref{eq.hurwitz.1})
ranges over all nonnegative integers $k$ and all $k$-tuples $\left(
i_{1},i_{2},\ldots,i_{k}\right)  $ of distinct elements of $V$. This includes
the case of $k=0$ and the empty $0$-tuple (which contributes the addend
$\left(  X+Y\right)  ^{n-0}\left(  \text{empty product}\right)  =\left(
X+Y\right)  ^{n}$). Notice that many of the addends in this sum will be equal
(indeed, if two $k$-tuples $\left(  i_{1},i_{2},\ldots,i_{k}\right)  $ and
$\left(  j_{1},j_{2},\ldots,j_{k}\right)  $ are permutations of each other,
then they produce equal addends).
\par
Once again, \textquotedblleft fractional\textquotedblright\ terms appear in
two of these identities, but are all cancelled.} It is easy to see that
setting all indeterminates $Z_{1},Z_{2},\ldots,Z_{n}$ equal to a single
indeterminate $Z$ transforms these three identities (\ref{eq.hurwitz.1}),
(\ref{eq.hurwitz.2}) and (\ref{eq.hurwitz.3}) into the original three
identities (\ref{eq.abel.1}), (\ref{eq.abel.2}) and (\ref{eq.abel.3}). We note
that combinatorial proofs of (\ref{eq.hurwitz.3}) were given by Francon in
\cite[Proposition 3.5]{Francon74} and by Shapiro in \cite{Shapir91}.

In this note, we shall show that the three identities (\ref{eq.hurwitz.1}),
(\ref{eq.hurwitz.2}) and (\ref{eq.hurwitz.3}) can be further generalized to a
noncommutative setting: Namely, the commuting indeterminates $X,Y,Z_{1}%
,Z_{2},\ldots,Z_{n}$ can be replaced by arbitrary elements $X,Y,x_{1}%
,x_{2},\ldots,x_{n}$ of any noncommutative ring $\mathbb{L}$, provided that a
centrality assumption holds (for the identities (\ref{eq.hurwitz.1}) and
(\ref{eq.hurwitz.2}), the sum $X+Y$ needs to lie in the center of $\mathbb{L}%
$, whereas for (\ref{eq.hurwitz.3}), the sum $X+Y+\sum_{s\in V}x_{s}$ needs to
lie in the center of $\mathbb{L}$), and provided that the product $Y\left(
Y+\sum_{s\in V\setminus S}Z_{s}\right)  ^{n-\left\vert S\right\vert -1}$ in
(\ref{eq.hurwitz.3}) is replaced by $\left(  Y+\sum_{s\in V\setminus S}%
x_{s}\right)  ^{n-\left\vert S\right\vert -1}Y$. These generalized versions of
(\ref{eq.hurwitz.1}), (\ref{eq.hurwitz.2}) and (\ref{eq.hurwitz.3}) are
Theorem \ref{thm.1}, Theorem \ref{thm.2} and Theorem \ref{thm.5} below, and
will be proven by a not-too-complicated induction on $n$.

\subsection*{Acknowledgments}

This note was prompted by an enumerative result of Gjergji Zaimi
\cite{Zaimi17}. The computer algebra SageMath \cite{SageMath} (specifically,
its \texttt{FreeAlgebra} class) was used to make conjectures. Thanks to Dennis
Stanton for making me aware of \cite{Johns96}.

\section{The identities}

Let us now state our results.

\begin{convention}
Let $\mathbb{L}$ be a noncommutative ring with unity.
\end{convention}

We claim that the following four theorems hold:\footnote{We promised three
identities, but we are stating four theorems. This is not a mistake, since
Theorem \ref{thm.5} is just an equivalent version of Theorem \ref{thm.4} (more
precisely, it is obtained from Theorem \ref{thm.4} by replacing $Y$ with
$Y+\sum_{s\in V}x_{s}$) and so should not be considered a separate identity.
We are stating these two theorems on an equal footing since we have no opinion
on which of them is the \textquotedblleft better\textquotedblright\ one.}

\begin{theorem}
\label{thm.1}Let $V$ be a finite set. Let $n=\left\vert V\right\vert $. For
each $s\in V$, let $x_{s}$ be an element of $\mathbb{L}$. Let $X$ and $Y$ be
two elements of $\mathbb{L}$ such that $X+Y$ lies in the center of
$\mathbb{L}$. Then,%
\[
\sum_{S\subseteq V}\left(  X+\sum_{s\in S}x_{s}\right)  ^{\left\vert
S\right\vert }\left(  Y-\sum_{s\in S}x_{s}\right)  ^{n-\left\vert S\right\vert
}=\sum_{\substack{i_{1},i_{2},\ldots,i_{k}\text{ are}\\\text{distinct}%
\\\text{elements of }V}}\left(  X+Y\right)  ^{n-k}x_{i_{1}}x_{i_{2}}\cdots
x_{i_{k}}.
\]
(Here, the sum on the right hand side ranges over all nonnegative integers $k$
and all $k$-tuples $\left(  i_{1},i_{2},\ldots,i_{k}\right)  $ of distinct
elements of $V$. In particular, it has an addend corresponding to $k=0$ and
$\left(  i_{1},i_{2},\ldots,i_{k}\right)  =\left(  {}\right)  $ (the empty
$0$-tuple); this addend is $\underbrace{\left(  X+Y\right)  ^{n-0}}_{=\left(
X+Y\right)  ^{n}}\cdot\underbrace{\left(  \text{empty product}\right)  }%
_{=1}=\left(  X+Y\right)  ^{n}$.)
\end{theorem}

\begin{example}
In the case when $V=\left\{  1,2\right\}  $, the claim of Theorem \ref{thm.1}
takes the following form (for any two elements $x_{1}$ and $x_{2}$ of
$\mathbb{L}$, and any two elements $X$ and $Y$ of $\mathbb{L}$ such that $X+Y$
lies in the center of $\mathbb{L}$):%
\begin{align*}
&  X^{0}Y^{2}+\left(  X+x_{1}\right)  ^{1}\left(  Y-x_{1}\right)  ^{1}+\left(
X+x_{2}\right)  ^{1}\left(  Y-x_{2}\right)  ^{1}\\
&  \ \ \ \ \ \ \ \ \ \ +\left(  X+x_{1}+x_{2}\right)  ^{2}\left(  Y-\left(
x_{1}+x_{2}\right)  \right)  ^{0}\\
&  =\left(  X+Y\right)  ^{2}+\left(  X+Y\right)  ^{1}x_{1}+\left(  X+Y\right)
^{1}x_{2}+\left(  X+Y\right)  ^{0}x_{1}x_{2}+\left(  X+Y\right)  ^{0}%
x_{2}x_{1}.
\end{align*}
If we try to verify this identity by subtracting the right hand side from the
left hand side and expanding, we can quickly realize that it boils down to%
\[
\left[  x_{1}+x_{2}+X,X+Y\right]  =0,
\]
where $\left[  a,b\right]  $ denotes the commutator of two elements $a$ and
$b$ of $\mathbb{L}$ (that is, $\left[  a,b\right]  =ab-ba$). Since $X+Y$ is
assumed to lie in the center of $\mathbb{L}$, this equality is correct. This
example shows that the requirement that $X+Y$ should lie in the center of
$\mathbb{L}$ cannot be lifted from Theorem \ref{thm.1}.

This example might suggest that we can replace this requirement by the weaker
condition that $\left[  \sum_{s\in V}x_{s}+X,X+Y\right]  =0$; but this would
not suffice for $n=3$.
\end{example}

\begin{theorem}
\label{thm.2}Let $V$ be a finite set. Let $n=\left\vert V\right\vert $. For
each $s\in V$, let $x_{s}$ be an element of $\mathbb{L}$. Let $X$ and $Y$ be
two elements of $\mathbb{L}$ such that $X+Y$ lies in the center of
$\mathbb{L}$. Then,%
\[
\sum_{S\subseteq V}X\left(  X+\sum_{s\in S}x_{s}\right)  ^{\left\vert
S\right\vert -1}\left(  Y-\sum_{s\in S}x_{s}\right)  ^{n-\left\vert
S\right\vert }=\left(  X+Y\right)  ^{n}.
\]
(Here, the product $X\left(  X+\sum_{s\in S}x_{s}\right)  ^{\left\vert
S\right\vert -1}$ has to be interpreted as $1$ when $S=\varnothing$.)
\end{theorem}

\begin{example}
In the case when $V=\left\{  1,2\right\}  $, the claim of Theorem \ref{thm.2}
takes the following form (for any two elements $x_{1}$ and $x_{2}$ of
$\mathbb{L}$, and any two elements $X$ and $Y$ of $\mathbb{L}$ such that $X+Y$
lies in the center of $\mathbb{L}$):%
\begin{align*}
&  XX^{-1}Y^{2}+X\left(  X+x_{1}\right)  ^{0}\left(  Y-x_{1}\right)
^{1}+X\left(  X+x_{2}\right)  ^{0}\left(  Y-x_{2}\right)  ^{1}\\
&  \ \ \ \ \ \ \ \ \ \ +X\left(  X+x_{1}+x_{2}\right)  ^{1}\left(  Y-\left(
x_{1}+x_{2}\right)  \right)  ^{0}\\
&  =\left(  X+Y\right)  ^{2}.
\end{align*}
(As explained in Theorem \ref{thm.2}, we should interpret the product
$XX^{-1}$ as $1$, so we don't need $X$ to be invertible.) This identity boils
down to $XY=YX$, which is a consequence of $X+Y$ lying in the center of
$\mathbb{L}$. Computations with $n\geq3$ show that merely assuming $XY=YX$
(without requiring that $X+Y$ lie in the center of $\mathbb{L}$) is not sufficient.
\end{example}

\begin{theorem}
\label{thm.4}Let $V$ be a finite set. Let $n=\left\vert V\right\vert $. For
each $s\in V$, let $x_{s}$ be an element of $\mathbb{L}$. Let $X$ and $Y$ be
two elements of $\mathbb{L}$ such that $X+Y$ lies in the center of
$\mathbb{L}$. Then,%
\begin{align*}
&  \sum_{S\subseteq V}X\left(  X+\sum_{s\in S}x_{s}\right)  ^{\left\vert
S\right\vert -1}\left(  Y-\sum_{s\in S}x_{s}\right)  ^{n-\left\vert
S\right\vert -1}\left(  Y-\sum_{s\in V}x_{s}\right) \\
&  =\left(  X+Y-\sum_{s\in V}x_{s}\right)  \left(  X+Y\right)  ^{n-1}.
\end{align*}
(Here,

\begin{itemize}
\item the product $X\left(  X+\sum_{s\in S}x_{s}\right)  ^{\left\vert
S\right\vert -1}$ has to be interpreted as $1$ when $S=\varnothing$;

\item the product $\left(  Y-\sum_{s\in S}x_{s}\right)  ^{n-\left\vert
S\right\vert -1}\left(  Y-\sum_{s\in V}x_{s}\right)  $ has to be interpreted
as $1$ when $\left\vert S\right\vert =n$;

\item the product $\left(  X+Y-\sum_{s\in V}x_{s}\right)  \left(  X+Y\right)
^{n-1}$ has to be interpreted as $1$ when $n=0$.)
\end{itemize}
\end{theorem}

\begin{theorem}
\label{thm.5}Let $V$ be a finite set. Let $n=\left\vert V\right\vert $. For
each $s\in V$, let $x_{s}$ be an element of $\mathbb{L}$. Let $X$ and $Y$ be
two elements of $\mathbb{L}$ such that $X+Y+\sum_{s\in V}x_{s}$ lies in the
center of $\mathbb{L}$. Then,%
\begin{align*}
&  \sum_{S\subseteq V}X\left(  X+\sum_{s\in S}x_{s}\right)  ^{\left\vert
S\right\vert -1}\left(  Y+\sum_{s\in V\setminus S}x_{s}\right)  ^{n-\left\vert
S\right\vert -1}Y\\
&  =\left(  X+Y\right)  \left(  X+Y+\sum_{s\in V}x_{s}\right)  ^{n-1}.
\end{align*}
(Here,

\begin{itemize}
\item the product $X\left(  X+\sum_{s\in S}x_{s}\right)  ^{\left\vert
S\right\vert -1}$ has to be interpreted as $1$ when $S=\varnothing$;

\item the product $\left(  Y+\sum_{s\in V\setminus S}x_{s}\right)
^{n-\left\vert S\right\vert -1}Y$ has to be interpreted as $1$ when
$\left\vert S\right\vert =n$;

\item the product $\left(  X+Y\right)  \left(  X+Y+\sum_{s\in V}x_{s}\right)
^{n-1}$ has to be interpreted as $1$ when $n=0$.)
\end{itemize}
\end{theorem}

Before we prove these theorems, let us cite some appearances of their
particular cases in the literature:

\begin{itemize}
\item Theorem \ref{thm.1} generalizes \cite[Problem 4]{qedmo09} (which is
obtained by setting $\mathbb{L}=\mathbb{Z}\left[  X,Y\right]  $ and $x_{s}=1$)
and \cite[\S 1.5, Cauchy's identity]{Riorda68} (which is obtained by setting
$\mathbb{L}=\mathbb{Z}\left[  X,Y\right]  $ and $X=x$ and $Y=y+n$ and
$x_{s}=1$).

\item Theorem \ref{thm.2} generalizes \cite[Chapter III, Theorem B]{Comtet74}
(which is obtained by setting $\mathbb{L}=\mathbb{Z}\left[  X,Y\right]  $ and
$x_{s}=z$) and \cite[Theorem 4]{qedmo09} (which is obtained by setting
$\mathbb{L}=\mathbb{Z}\left[  X,Y\right]  $ and $x_{s}=1$) and \cite[(11)]%
{Kalai79} (which is obtained by setting $\mathbb{L}=\mathbb{Z}\left[
x,y\right]  $ and $X=x$ and $Y=n+y$) and \cite[1.3]{KelPos} (which is obtained
by setting $\mathbb{L}=\mathbb{Z}\left[  z,y,x\left(  a\right)  \ \mid\ a\in
V\right]  $ and $X=y$ and $Y=z+x\left(  V\right)  $ and $x_{s}=x\left(
s\right)  $) and \textquotedblleft Hurwitz's formula\textquotedblright\ in
\cite[solution to Section 1.2.6, Exercise 51]{Knuth97} (which is obtained by
setting $V=\left\{  1,2,\ldots,n\right\}  $ and $X=x$ and $Y=y$ and
$x_{s}=z_{s}$) and \cite[\S 1.5, (13)]{Riorda68} (which is obtained by setting
$\mathbb{L}=\mathbb{Z}\left[  X,Y,a\right]  $ and $X=x$ and $Y=y+na$ and
$x_{s}=a$) and \cite[Exercise 5.31 \textbf{b}]{Stanley-EC2} (which is obtained
by setting $\mathbb{L}=\mathbb{Z}\left[  x_{1},x_{2},\ldots,x_{n+2}\right]  $
and $X=x_{n+1}$ and $Y=\sum_{i=1}^{n}x_{i}+x_{n+2}$).

\item Theorem \ref{thm.5} generalizes \cite[Chapter III, Exercise
20]{Comtet74} (which is obtained when $\mathbb{L}$ is commutative) and
\cite[1.2]{KelPos} (which is obtained by setting $\mathbb{L}=\mathbb{Z}\left[
z,y,x\left(  a\right)  \ \mid\ a\in V\right]  $ and $X=y$ and $Y=z$ and
$x_{s}=x\left(  s\right)  $) and \cite[Section 2.3.4.4, Exercise 30]{Knuth97}
(which is obtained by setting $V=\left\{  1,2,\ldots,n\right\}  $ and $X=x$
and $Y=y$ and $x_{s}=z_{s}$).
\end{itemize}

\section{The proofs}

We now come to the proofs of the identities stated above.

\begin{convention}
We shall use the notation $\mathbb{N}$ for the set $\left\{  0,1,2,\ldots
\right\}  $.
\end{convention}

\subsection{Proofs of Theorems \ref{thm.1} and \ref{thm.2}}

\begin{vershort}
\begin{proof}
[Proof of Theorem \ref{thm.1} and Theorem \ref{thm.2}.]We shall prove Theorem
\ref{thm.1} and Theorem \ref{thm.2} together, by a simultaneous induction. The
induction base (the case $n=0$) is left to the reader.

For the \textit{induction step}, we fix a positive integer $n$, and we assume
(as the induction hypothesis) that both Theorem \ref{thm.1} and Theorem
\ref{thm.2} are proven for $n-1$ instead of $n$. We shall now prove Theorem
\ref{thm.1} and Theorem \ref{thm.2} for our number $n$. So let $V$, $x_{s}$,
$X$ and $Y$ be as in Theorem \ref{thm.1} and Theorem \ref{thm.2}.

Fix $t\in V$.

The induction hypothesis shows that Theorem \ref{thm.2} is proven for $n-1$
instead of $n$. We can thus apply Theorem \ref{thm.2} to $V\setminus\left\{
t\right\}  $ instead of $V$ (since the finite set $V\setminus\left\{
t\right\}  $ has size $\left\vert V\setminus\left\{  t\right\}  \right\vert
=n-1$). Thus, we obtain%
\begin{equation}
\sum_{S\subseteq V\setminus\left\{  t\right\}  }X\left(  X+\sum_{s\in S}%
x_{s}\right)  ^{\left\vert S\right\vert -1}\left(  Y-\sum_{s\in S}%
x_{s}\right)  ^{n-1-\left\vert S\right\vert }=\left(  X+Y\right)  ^{n-1}.
\label{pf.thm.2.1}%
\end{equation}

But the induction hypothesis also shows that Theorem \ref{thm.1} is proven for
$n-1$ instead of $n$. Thus, we can apply Theorem \ref{thm.1} to $V\setminus
\left\{  t\right\}  $ instead of $V$ (since the finite set $V\setminus\left\{
t\right\}  $ has size $\left\vert V\setminus\left\{  t\right\}  \right\vert
=n-1$). Thus, we obtain%
\begin{align}
&  \sum_{S\subseteq V\setminus\left\{  t\right\}  }\left(  X+\sum_{s\in
S}x_{s}\right)  ^{\left\vert S\right\vert }\left(  Y-\sum_{s\in S}%
x_{s}\right)  ^{n-1-\left\vert S\right\vert }\nonumber\\
&  =\sum_{\substack{i_{1},i_{2},\ldots,i_{k}\text{ are}\\\text{distinct}%
\\\text{elements of }V\setminus\left\{  t\right\}  }}\left(  X+Y\right)
^{n-1-k}x_{i_{1}}x_{i_{2}}\cdots x_{i_{k}}. \label{pf.thm.2.4}%
\end{align}
Likewise, we can apply Theorem \ref{thm.1} to $V\setminus\left\{  t\right\}
$, $X+x_{t}$ and $Y-x_{t}$ instead of $V$, $X$ and $Y$ (because the finite set
$V\setminus\left\{  t\right\}  $ has size $\left\vert V\setminus\left\{
t\right\}  \right\vert =n-1$, and because the sum $\left(  X+x_{t}\right)
+\left(  Y-x_{t}\right)  =X+Y$ lies in the center of $\mathbb{L}$). We thus
obtain
\begin{align}
&  \sum_{S\subseteq V\setminus\left\{  t\right\}  }\left(  X+x_{t}+\sum_{s\in
S}x_{s}\right)  ^{\left\vert S\right\vert }\left(  Y-x_{t}-\sum_{s\in S}%
x_{s}\right)  ^{n-1-\left\vert S\right\vert }\nonumber\\
&  =\sum_{\substack{i_{1},i_{2},\ldots,i_{k}\text{ are}\\\text{distinct}%
\\\text{elements of }V\setminus\left\{  t\right\}  }}\left(
\underbrace{\left(  X+x_{t}\right)  +\left(  Y-x_{t}\right)  }_{=X+Y}\right)
^{n-1-k}x_{i_{1}}x_{i_{2}}\cdots x_{i_{k}}\nonumber\\
&  =\sum_{\substack{i_{1},i_{2},\ldots,i_{k}\text{ are}\\\text{distinct}%
\\\text{elements of }V\setminus\left\{  t\right\}  }}\left(  X+Y\right)
^{n-1-k}x_{i_{1}}x_{i_{2}}\cdots x_{i_{k}}. \label{pf.thm.2.5}%
\end{align}
Now,%
\begin{align}
&  \sum_{\substack{S\subseteq V;\\t\in S}}\left(  X+\sum_{s\in S}x_{s}\right)
^{\left\vert S\right\vert -1}\left(  Y-\sum_{s\in S}x_{s}\right)
^{n-\left\vert S\right\vert }\nonumber\\
&  =\underbrace{\sum_{\substack{S\subseteq V;\\t\notin S}}}_{=\sum_{S\subseteq
V\setminus\left\{  t\right\}  }}\ \ \underbrace{\left(  X+\sum_{s\in
S\cup\left\{  t\right\}  }x_{s}\right)  ^{\left\vert S\cup\left\{  t\right\}
\right\vert -1}}_{\substack{=\left(  X+x_{t}+\sum_{s\in S}x_{s}\right)
^{\left\vert S\right\vert }\\\text{(since }t\notin S\text{ yields }\sum_{s\in
S\cup\left\{  t\right\}  }x_{s}=x_{t}+\sum_{s\in S}x_{s}\\\text{and
}\left\vert S\cup\left\{  t\right\}  \right\vert =\left\vert S\right\vert
+1\text{)}}}\ \ \underbrace{\left(  Y-\sum_{s\in S\cup\left\{  t\right\}
}x_{s}\right)  ^{n-\left\vert S\cup\left\{  t\right\}  \right\vert }%
}_{\substack{=\left(  Y-x_{t}-\sum_{s\in S}x_{s}\right)  ^{n-1-\left\vert
S\right\vert }\\\text{(since }t\notin S\text{ yields }\sum_{s\in S\cup\left\{
t\right\}  }x_{s}=x_{t}+\sum_{s\in S}x_{s}\\\text{and }\left\vert
S\cup\left\{  t\right\}  \right\vert =\left\vert S\right\vert +1\text{)}%
}}\nonumber\\
&  \ \ \ \ \ \ \ \ \ \ \ \ \ \ \ \ \ \ \ \ \left(
\begin{array}
[c]{c}%
\text{here, we have substituted }S\cup\left\{  t\right\}  \text{ for }S\text{
in the sum, since}\\
\text{the map }\left\{  S\subseteq V\ \mid\ t\notin S\right\}  \rightarrow
\left\{  S\subseteq V\ \mid\ t\in S\right\}  ,\ S\mapsto S\cup\left\{
t\right\} \\
\text{is a bijection}%
\end{array}
\right) \nonumber\\
&  =\sum_{S\subseteq V\setminus\left\{  t\right\}  }\left(  X+x_{t}+\sum_{s\in
S}x_{s}\right)  ^{\left\vert S\right\vert }\left(  Y-x_{t}-\sum_{s\in S}%
x_{s}\right)  ^{n-1-\left\vert S\right\vert }\nonumber\\
&  =\sum_{\substack{i_{1},i_{2},\ldots,i_{k}\text{ are}\\\text{distinct}%
\\\text{elements of }V\setminus\left\{  t\right\}  }}\left(  X+Y\right)
^{n-1-k}x_{i_{1}}x_{i_{2}}\cdots x_{i_{k}}\ \ \ \ \ \ \ \ \ \ \left(  \text{by
(\ref{pf.thm.2.5})}\right) \label{pf.thm.1.3}\\
&  =\sum_{S\subseteq V\setminus\left\{  t\right\}  }\left(  X+\sum_{s\in
S}x_{s}\right)  ^{\left\vert S\right\vert }\left(  Y-\sum_{s\in S}%
x_{s}\right)  ^{n-1-\left\vert S\right\vert } \label{pf.thm.1.4}%
\end{align}
(by (\ref{pf.thm.2.4})). Multiplying both sides of this equality by $X$, we
obtain%
\begin{align}
&  \sum_{\substack{S\subseteq V;\\t\in S}}X\left(  X+\sum_{s\in S}%
x_{s}\right)  ^{\left\vert S\right\vert -1}\left(  Y-\sum_{s\in S}%
x_{s}\right)  ^{n-\left\vert S\right\vert }\nonumber\\
&  =\sum_{S\subseteq V\setminus\left\{  t\right\}  }X\left(  X+\sum_{s\in
S}x_{s}\right)  ^{\left\vert S\right\vert }\left(  Y-\sum_{s\in S}%
x_{s}\right)  ^{n-1-\left\vert S\right\vert }. \label{pf.thm.2.compareX}%
\end{align}

Now,%
\begin{align*}
&  \sum_{S\subseteq V}X\left(  X+\sum_{s\in S}x_{s}\right)  ^{\left\vert
S\right\vert -1}\left(  Y-\sum_{s\in S}x_{s}\right)  ^{n-\left\vert
S\right\vert }\\
&  =\underbrace{\sum_{\substack{S\subseteq V;\\t\in S}}X\left(  X+\sum_{s\in
S}x_{s}\right)  ^{\left\vert S\right\vert -1}\left(  Y-\sum_{s\in S}%
x_{s}\right)  ^{n-\left\vert S\right\vert }}_{\substack{_{\substack{=\sum
_{S\subseteq V\setminus\left\{  t\right\}  }X\left(  X+\sum_{s\in S}%
x_{s}\right)  ^{\left\vert S\right\vert }\left(  Y-\sum_{s\in S}x_{s}\right)
^{n-1-\left\vert S\right\vert }}}\\\text{(by (\ref{pf.thm.2.compareX}))}}}\\
&  \ \ \ \ \ \ \ \ \ \ +\underbrace{\sum_{\substack{S\subseteq V;\\t\notin
S}}}_{=\sum_{S\subseteq V\setminus\left\{  t\right\}  }}X\left(  X+\sum_{s\in
S}x_{s}\right)  ^{\left\vert S\right\vert -1}\underbrace{\left(  Y-\sum_{s\in
S}x_{s}\right)  ^{n-\left\vert S\right\vert }}_{=\left(  Y-\sum_{s\in S}%
x_{s}\right)  \left(  Y-\sum_{s\in S}x_{s}\right)  ^{n-1-\left\vert
S\right\vert }}\\
&  =\sum_{S\subseteq V\setminus\left\{  t\right\}  }X\underbrace{\left(
X+\sum_{s\in S}x_{s}\right)  ^{\left\vert S\right\vert }}_{\substack{=\left(
X+\sum_{s\in S}x_{s}\right)  ^{\left\vert S\right\vert -1}\left(  X+\sum_{s\in
S}x_{s}\right)  \\\text{(it is easy to check that this makes sense}%
\\\text{even for }S=\varnothing\text{, due to the }X\text{ factor on our
left)}}}\left(  Y-\sum_{s\in S}x_{s}\right)  ^{n-1-\left\vert S\right\vert }\\
&  \ \ \ \ \ \ \ \ \ \ +\sum_{S\subseteq V\setminus\left\{  t\right\}
}X\left(  X+\sum_{s\in S}x_{s}\right)  ^{\left\vert S\right\vert -1}\left(
Y-\sum_{s\in S}x_{s}\right)  \left(  Y-\sum_{s\in S}x_{s}\right)
^{n-1-\left\vert S\right\vert }\\
&  =\sum_{S\subseteq V\setminus\left\{  t\right\}  }X\left(  X+\sum_{s\in
S}x_{s}\right)  ^{\left\vert S\right\vert -1}\left(  X+\sum_{s\in S}%
x_{s}\right)  \left(  Y-\sum_{s\in S}x_{s}\right)  ^{n-1-\left\vert
S\right\vert }\\
&  \ \ \ \ \ \ \ \ \ \ +\sum_{S\subseteq V\setminus\left\{  t\right\}
}X\left(  X+\sum_{s\in S}x_{s}\right)  ^{\left\vert S\right\vert -1}\left(
Y-\sum_{s\in S}x_{s}\right)  \left(  Y-\sum_{s\in S}x_{s}\right)
^{n-1-\left\vert S\right\vert }%
\end{align*}%
\begin{align}
&  =\sum_{S\subseteq V\setminus\left\{  t\right\}  }X\left(  X+\sum_{s\in
S}x_{s}\right)  ^{\left\vert S\right\vert -1}\underbrace{\left(  \left(
X+\sum_{s\in S}x_{s}\right)  +\left(  Y-\sum_{s\in S}x_{s}\right)  \right)
}_{=X+Y}\left(  Y-\sum_{s\in S}x_{s}\right)  ^{n-1-\left\vert S\right\vert
}\nonumber\\
&  =\sum_{S\subseteq V\setminus\left\{  t\right\}  }\underbrace{X\left(
X+\sum_{s\in S}x_{s}\right)  ^{\left\vert S\right\vert -1}\left(  X+Y\right)
}_{\substack{=\left(  X+Y\right)  X\left(  X+\sum_{s\in S}x_{s}\right)
^{\left\vert S\right\vert -1}\\\text{(since }X+Y\text{ lies in the center of
}\mathbb{L}\text{)}}}\left(  Y-\sum_{s\in S}x_{s}\right)  ^{n-1-\left\vert
S\right\vert }\nonumber\\
&  =\left(  X+Y\right)  \underbrace{\sum_{S\subseteq V\setminus\left\{
t\right\}  }X\left(  X+\sum_{s\in S}x_{s}\right)  ^{\left\vert S\right\vert
-1}\left(  Y-\sum_{s\in S}x_{s}\right)  ^{n-1-\left\vert S\right\vert }%
}_{\substack{=\left(  X+Y\right)  ^{n-1}\\\text{(by (\ref{pf.thm.2.1}))}%
}}\nonumber\\
&  =\left(  X+Y\right)  \left(  X+Y\right)  ^{n-1}=\left(  X+Y\right)  ^{n}.
\label{pf.thm.2.firstgoal}%
\end{align}

Now, forget that we fixed $t$. We thus have proven the equalities
(\ref{pf.thm.1.3}), (\ref{pf.thm.2.5}) and (\ref{pf.thm.2.firstgoal}) for each
$t\in V$.

The set $V$ is nonempty (since $\left\vert V\right\vert =n$ is positive).
Hence, there exists some $q\in V$. Consider this $q$. Applying
(\ref{pf.thm.2.firstgoal}) to $t=q$, we obtain%
\begin{equation}
\sum_{S\subseteq V}X\left(  X+\sum_{s\in S}x_{s}\right)  ^{\left\vert
S\right\vert -1}\left(  Y-\sum_{s\in S}x_{s}\right)  ^{n-\left\vert
S\right\vert }=\left(  X+Y\right)  ^{n}. \label{pf.thm.2.firstgoal'}%
\end{equation}
Therefore, Theorem \ref{thm.2} is proven for our $n$. It remains to prove
Theorem \ref{thm.1} for our $n$.

From (\ref{pf.thm.2.firstgoal'}), we obtain%
\begin{align}
\left(  X+Y\right)  ^{n}  &  =\sum_{S\subseteq V}X\left(  X+\sum_{s\in S}%
x_{s}\right)  ^{\left\vert S\right\vert -1}\left(  Y-\sum_{s\in S}%
x_{s}\right)  ^{n-\left\vert S\right\vert }\nonumber\\
&  =Y^{n}+\sum_{\substack{S\subseteq V;\\S\neq\varnothing}}X\left(
X+\sum_{s\in S}x_{s}\right)  ^{\left\vert S\right\vert -1}\left(  Y-\sum_{s\in
S}x_{s}\right)  ^{n-\left\vert S\right\vert } \label{pf.thm.2.firstgoal''}%
\end{align}
(here, we have split off the addend for $S=\varnothing$ from the sum).

We have%
\begin{align*}
&  \sum_{S\subseteq V}\left(  X+\sum_{s\in S}x_{s}\right)  ^{\left\vert
S\right\vert }\left(  Y-\sum_{s\in S}x_{s}\right)  ^{n-\left\vert S\right\vert
}\\
&  =Y^{n}+\sum_{\substack{S\subseteq V;\\S\neq\varnothing}}\underbrace{\left(
X+\sum_{s\in S}x_{s}\right)  ^{\left\vert S\right\vert }}_{=\left(
X+\sum_{s\in S}x_{s}\right)  \left(  X+\sum_{s\in S}x_{s}\right)  ^{\left\vert
S\right\vert -1}}\left(  Y-\sum_{s\in S}x_{s}\right)  ^{n-\left\vert
S\right\vert }\\
&  \ \ \ \ \ \ \ \ \ \ \ \ \ \ \ \ \ \ \ \ \left(  \text{here, we have split
off the addend for }S=\varnothing\text{ from the sum}\right) \\
&  =Y^{n}+\sum_{\substack{S\subseteq V;\\S\neq\varnothing}}\left(
X+\underbrace{\sum_{s\in S}x_{s}}_{=\sum_{t\in S}x_{t}}\right)  \left(
X+\sum_{s\in S}x_{s}\right)  ^{\left\vert S\right\vert -1}\left(  Y-\sum_{s\in
S}x_{s}\right)  ^{n-\left\vert S\right\vert }\\
&  =Y^{n}+\sum_{\substack{S\subseteq V;\\S\neq\varnothing}}\underbrace{\left(
X+\sum_{t\in S}x_{t}\right)  \left(  X+\sum_{s\in S}x_{s}\right)  ^{\left\vert
S\right\vert -1}\left(  Y-\sum_{s\in S}x_{s}\right)  ^{n-\left\vert
S\right\vert }}_{=X\left(  X+\sum_{s\in S}x_{s}\right)  ^{\left\vert
S\right\vert -1}\left(  Y-\sum_{s\in S}x_{s}\right)  ^{n-\left\vert
S\right\vert }+\sum_{t\in S}x_{t}\left(  X+\sum_{s\in S}x_{s}\right)
^{\left\vert S\right\vert -1}\left(  Y-\sum_{s\in S}x_{s}\right)
^{n-\left\vert S\right\vert }}\\
&  =\underbrace{Y^{n}+\sum_{\substack{S\subseteq V;\\S\neq\varnothing
}}X\left(  X+\sum_{s\in S}x_{s}\right)  ^{\left\vert S\right\vert -1}\left(
Y-\sum_{s\in S}x_{s}\right)  ^{n-\left\vert S\right\vert }}%
_{\substack{=\left(  X+Y\right)  ^{n}\\\text{(by (\ref{pf.thm.2.firstgoal''}%
))}}}\\
&  \ \ \ \ \ \ \ \ \ \ +\underbrace{\sum_{\substack{S\subseteq V;\\S\neq
\varnothing}}\sum_{t\in S}}_{=\sum_{t\in V}\sum_{\substack{S\subseteq V;\\t\in
S}}}x_{t}\left(  X+\sum_{s\in S}x_{s}\right)  ^{\left\vert S\right\vert
-1}\left(  Y-\sum_{s\in S}x_{s}\right)  ^{n-\left\vert S\right\vert }\\
&  =\left(  X+Y\right)  ^{n}+\sum_{t\in V}x_{t}\underbrace{\sum
_{\substack{S\subseteq V;\\t\in S}}\left(  X+\sum_{s\in S}x_{s}\right)
^{\left\vert S\right\vert -1}\left(  Y-\sum_{s\in S}x_{s}\right)
^{n-\left\vert S\right\vert }}_{\substack{=\sum_{\substack{i_{1},i_{2}%
,\ldots,i_{k}\text{ are}\\\text{distinct}\\\text{elements of }V\setminus
\left\{  t\right\}  }}\left(  X+Y\right)  ^{n-1-k}x_{i_{1}}x_{i_{2}}\cdots
x_{i_{k}}\\\text{(by (\ref{pf.thm.1.3}))}}}\\
&  =\left(  X+Y\right)  ^{n}+\sum_{t\in V}x_{t}\sum_{\substack{i_{1}%
,i_{2},\ldots,i_{k}\text{ are}\\\text{distinct}\\\text{elements of }%
V\setminus\left\{  t\right\}  }}\left(  X+Y\right)  ^{n-1-k}x_{i_{1}}x_{i_{2}%
}\cdots x_{i_{k}}.
\end{align*}
Compared with%
\begin{align*}
&  \sum_{\substack{i_{1},i_{2},\ldots,i_{k}\text{ are}\\\text{distinct}%
\\\text{elements of }V}}\left(  X+Y\right)  ^{n-k}x_{i_{1}}x_{i_{2}}\cdots
x_{i_{k}}\\
&  =\underbrace{\left(  X+Y\right)  ^{n-0}}_{=\left(  X+Y\right)  ^{n}%
}\underbrace{\left(  \text{empty product}\right)  }_{=1}+\sum_{\substack{i_{1}%
,i_{2},\ldots,i_{k}\text{ are}\\\text{distinct}\\\text{elements of }%
V;\\k>0}}\left(  X+Y\right)  ^{n-k}x_{i_{1}}x_{i_{2}}\cdots x_{i_{k}}\\
&  \ \ \ \ \ \ \ \ \ \ \ \ \ \ \ \ \ \ \ \ \left(
\begin{array}
[c]{c}%
\text{here, we have split off the addend}\\
\text{for }k=0\text{ and }\left(  i_{1},i_{2},\ldots,i_{k}\right)  =\left(
{}\right) \\
\text{from the sum}%
\end{array}
\right) \\
&  =\left(  X+Y\right)  ^{n}+\sum_{\substack{i_{1},i_{2},\ldots,i_{k}\text{
are}\\\text{distinct}\\\text{elements of }V;\\k>0}}\left(  X+Y\right)
^{n-k}x_{i_{1}}x_{i_{2}}\cdots x_{i_{k}}\\
&  =\left(  X+Y\right)  ^{n}+\sum_{\substack{i_{1},i_{2},\ldots,i_{k+1}\text{
are}\\\text{distinct}\\\text{elements of }V}}\underbrace{\left(  X+Y\right)
^{n-\left(  k+1\right)  }}_{=\left(  X+Y\right)  ^{n-1-k}}x_{i_{1}}x_{i_{2}%
}\cdots x_{i_{k+1}}\\
&  \ \ \ \ \ \ \ \ \ \ \ \ \ \ \ \ \ \ \ \ \left(  \text{here, we have
substituted }k+1\text{ for }k\text{ in the sum}\right) \\
&  =\left(  X+Y\right)  ^{n}+\sum_{\substack{i_{1},i_{2},\ldots,i_{k+1}\text{
are}\\\text{distinct}\\\text{elements of }V}}\left(  X+Y\right)
^{n-1-k}x_{i_{1}}x_{i_{2}}\cdots x_{i_{k+1}}\\
&  =\left(  X+Y\right)  ^{n}+\underbrace{\sum_{\substack{t,i_{1},i_{2}%
,\ldots,i_{k}\text{ are}\\\text{distinct}\\\text{elements of }V}}}%
_{=\sum_{t\in V}\sum_{\substack{i_{1},i_{2},\ldots,i_{k}\text{ are}%
\\\text{distinct}\\\text{elements of }V\setminus\left\{  t\right\}  }%
}}\underbrace{\left(  X+Y\right)  ^{n-1-k}x_{t}}_{\substack{=x_{t}\left(
X+Y\right)  ^{n-1-k}\\\text{(since }X+Y\text{ lies in the center of
}\mathbb{L}\text{,}\\\text{and therefore so does }\left(  X+Y\right)
^{n-1-k}\text{)}}}x_{i_{1}}x_{i_{2}}\cdots x_{i_{k}}\\
&  \ \ \ \ \ \ \ \ \ \ \ \ \ \ \ \ \ \ \ \ \left(
\begin{array}
[c]{c}%
\text{here, we have renamed the}\\
\text{summation index }\left(  i_{1},i_{2},\ldots,i_{k+1}\right) \\
\text{as }\left(  t,i_{1},i_{2},\ldots,i_{k}\right)
\end{array}
\right) \\
&  =\left(  X+Y\right)  ^{n}+\sum_{t\in V}\sum_{\substack{i_{1},i_{2}%
,\ldots,i_{k}\text{ are}\\\text{distinct}\\\text{elements of }V\setminus
\left\{  t\right\}  }}x_{t}\left(  X+Y\right)  ^{n-1-k}x_{i_{1}}x_{i_{2}%
}\cdots x_{i_{k}}\\
&  =\left(  X+Y\right)  ^{n}+\sum_{t\in V}x_{t}\sum_{\substack{i_{1}%
,i_{2},\ldots,i_{k}\text{ are}\\\text{distinct}\\\text{elements of }%
V\setminus\left\{  t\right\}  }}\left(  X+Y\right)  ^{n-1-k}x_{i_{1}}x_{i_{2}%
}\cdots x_{i_{k}},
\end{align*}
this yields%
\[
\sum_{S\subseteq V}\left(  X+\sum_{s\in S}x_{s}\right)  ^{\left\vert
S\right\vert }\left(  Y-\sum_{s\in S}x_{s}\right)  ^{n-\left\vert S\right\vert
}=\sum_{\substack{i_{1},i_{2},\ldots,i_{k}\text{ are}\\\text{distinct}%
\\\text{elements of }V}}\left(  X+Y\right)  ^{n-k}x_{i_{1}}x_{i_{2}}\cdots
x_{i_{k}}.
\]
Therefore, Theorem \ref{thm.1} is proven for our $n$. This completes the
induction step. Hence, both Theorem \ref{thm.1} and Theorem \ref{thm.2} are proven.
\end{proof}
\end{vershort}

\begin{verlong}
We shall prove Theorem \ref{thm.1} and Theorem \ref{thm.2} together, by a
simultaneous induction. In order to make the computations more palatable, let
us first introduce a convenient notation:

\begin{definition}
\label{def.x(S)}Assume that $V$ is a finite set. Assume that $x_{s}$ is an
element of $\mathbb{L}$ for each $s\in V$. Assume that $S$ is a subset of $V$.
Then, $x\left(  S\right)  $ shall denote the element $\sum_{s\in S}x_{s}$ of
$\mathbb{L}$.
\end{definition}

We make a trivial observation:

\begin{lemma}
\label{lem.x(S).triv}Let $V$ be a finite set. For each $s\in V$, let $x_{s}$
be an element of $\mathbb{L}$.

\begin{enumerate}
\item[\textbf{(a)}] We have $x\left(  \varnothing\right)  =0$.

\item[\textbf{(b)}] Let $t\in V$. Let $S$ be a subset of $V\setminus\left\{
t\right\}  $. Then, $x\left(  S\cup\left\{  t\right\}  \right)  =x_{t}%
+x\left(  S\right)  $.
\end{enumerate}
\end{lemma}

\begin{proof}
[Proof of Lemma \ref{lem.x(S).triv}.]\textbf{(a)} The definition of $x\left(
\varnothing\right)  $ yields $x\left(  \varnothing\right)  =\sum
_{s\in\varnothing}x_{s}=\left(  \text{empty sum}\right)  =0$. This proves
Lemma \ref{lem.x(S).triv} \textbf{(a)}. \medskip

\textbf{(b)} We have $t\notin S$\ \ \ \ \footnote{\textit{Proof.} Assume the
contrary. Thus, $t\in S$. Hence, $t\in S\subseteq V\setminus\left\{
t\right\}  $. In other words, $t\in V$ and $t\notin\left\{  t\right\}  $. But
$t\notin\left\{  t\right\}  $ contradicts $t\in\left\{  t\right\}  $. This
contradiction shows that our assumption was false. Qed.}. Hence, $\left(
S\cup\left\{  t\right\}  \right)  \setminus\left\{  t\right\}  =S$. Also,
$t\in\left\{  t\right\}  \subseteq S\cup\left\{  t\right\}  $.

The definition of $x\left(  S\right)  $ yields $x\left(  S\right)  =\sum_{s\in
S}x_{s}$.

Now, the definition of $x\left(  S\cup\left\{  t\right\}  \right)  $ yields%
\begin{align*}
x\left(  S\cup\left\{  t\right\}  \right)   &  =\sum_{s\in S\cup\left\{
t\right\}  }x_{s}=x_{t}+\underbrace{\sum_{\substack{s\in S\cup\left\{
t\right\}  ;\\s\neq t}}}_{\substack{=\sum_{s\in\left(  S\cup\left\{
t\right\}  \right)  \setminus\left\{  t\right\}  }=\sum_{s\in S}\\\text{(since
}\left(  S\cup\left\{  t\right\}  \right)  \setminus\left\{  t\right\}
=S\text{)}}}x_{s}\\
&  \ \ \ \ \ \ \ \ \ \ \ \ \ \ \ \ \ \ \ \ \left(
\begin{array}
[c]{c}%
\text{here, we have split off the addend for }s=t\\
\text{from the sum (since }t\in S\cup\left\{  t\right\}  \text{)}%
\end{array}
\right) \\
&  =x_{t}+\underbrace{\sum_{s\in S}x_{s}}_{=x\left(  S\right)  }%
=x_{t}+x\left(  S\right)  .
\end{align*}
This proves Lemma \ref{lem.x(S).triv} \textbf{(b)}.
\end{proof}

Now, let us restate Theorem \ref{thm.1} and Theorem \ref{thm.2} together in a
form that will be convenient for us to prove:

\begin{lemma}
\label{lem.12}Let $V$ be a finite set. Let $n=\left\vert V\right\vert $. For
each $s\in V$, let $x_{s}$ be an element of $\mathbb{L}$. Let $X$ and $Y$ be
two elements of $\mathbb{L}$ such that $X+Y$ lies in the center of
$\mathbb{L}$. Then,%
\begin{align}
&  \sum_{S\subseteq V}\left(  X+x\left(  S\right)  \right)  ^{\left\vert
S\right\vert }\left(  Y-x\left(  S\right)  \right)  ^{n-\left\vert
S\right\vert }\nonumber\\
&  =\sum_{\substack{i_{1},i_{2},\ldots,i_{k}\text{ are}\\\text{distinct}%
\\\text{elements of }V}}\left(  X+Y\right)  ^{n-k}x_{i_{1}}x_{i_{2}}\cdots
x_{i_{k}} \label{eq.lem12.1}%
\end{align}
and%
\begin{equation}
Y^{n}+\sum_{\substack{S\subseteq V;\\S\neq\varnothing}}X\left(  X+x\left(
S\right)  \right)  ^{\left\vert S\right\vert -1}\left(  Y-x\left(  S\right)
\right)  ^{n-\left\vert S\right\vert }=\left(  X+Y\right)  ^{n}.
\label{eq.lem12.2}%
\end{equation}
(Here, the sum on the right hand side of (\ref{eq.lem12.1}) has to be
interpreted as in Theorem \ref{thm.1}.)
\end{lemma}

Note that the equalities (\ref{eq.lem12.1}) and (\ref{eq.lem12.2}) are
restatements of the claims of Theorem \ref{thm.1} and of Theorem \ref{thm.2},
respectively. Unlike the claim of Theorem \ref{thm.2}, however, the equality
(\ref{eq.lem12.2}) does not require any convention about how to interpret the
term $X\left(  X+\sum_{s\in S}x_{s}\right)  ^{\left\vert S\right\vert -1}$
when $S=\varnothing$, because the sum on the left hand side (\ref{eq.lem12.2})
has no addend for $S=\varnothing$.

\begin{proof}
[Proof of Lemma \ref{lem.12}.]Let us first notice that all terms appearing in
Lemma \ref{lem.12} are well-defined\footnote{\textit{Proof.} This follows from
the following four observations:
\par
\begin{itemize}
\item For every subset $S$ of $V$, the term $\left\vert S\right\vert $ is
well-defined. (\textit{Proof:} Let $S$ be a subset of $V$. Then, $S$ is a
subset of the finite set $V$, and therefore itself is finite. Hence, the term
$\left\vert S\right\vert $ is well-defined. Qed.)
\par
\item For every subset $S$ of $V$, the term $\left(  Y-x\left(  S\right)
\right)  ^{n-\left\vert S\right\vert }$ is well-defined. (\textit{Proof:} Let
$S$ be a subset of $V$. Then, $\left\vert S\right\vert \leq\left\vert
V\right\vert =n$, so that $n-\left\vert S\right\vert \geq0$. In other words,
$n-\left\vert S\right\vert \in\mathbb{N}$. Hence, the term $\left(  Y-x\left(
S\right)  \right)  ^{n-\left\vert S\right\vert }$ is well-defined. Qed.)
\par
\item Whenever $i_{1},i_{2},\ldots,i_{k}$ are distinct elements of $V$, the
term $\left(  X+Y\right)  ^{n-k}$ is well-defined. (\textit{Proof:} Let
$i_{1},i_{2},\ldots,i_{k}$ be distinct elements of $V$. Then, $\left\vert
\left\{  i_{1},i_{2},\ldots,i_{k}\right\}  \right\vert =k$ (since $i_{1}%
,i_{2},\ldots,i_{k}$ are $k$ distinct elements). But $\left\{  i_{1}%
,i_{2},\ldots,i_{k}\right\}  \subseteq V$ (since $i_{1},i_{2},\ldots,i_{k}$
are elements of $V$) and therefore $\left\vert \left\{  i_{1},i_{2}%
,\ldots,i_{k}\right\}  \right\vert \leq\left\vert V\right\vert =n$. Hence,
$n\geq\left\vert \left\{  i_{1},i_{2},\ldots,i_{k}\right\}  \right\vert =k$,
so that $n-k\geq0$. In other words, $n-k\in\mathbb{N}$. Thus, the term
$\left(  X+Y\right)  ^{n-k}$ is well-defined. Qed.)
\par
\item For every subset $S$ of $V$ satisfying $S\neq\varnothing$, the term
$\left(  X+x\left(  S\right)  \right)  ^{\left\vert S\right\vert -1}$ is
well-defined. (\textit{Proof:} Let $S$ be a subset of $V$ satisfying
$S\neq\varnothing$. Then, $\left\vert S\right\vert >0$ (since $S\neq
\varnothing$) and thus $\left\vert S\right\vert \geq1$ (since $\left\vert
S\right\vert $ is an integer). In other words, $\left\vert S\right\vert
-1\in\mathbb{N}$. Hence, the term $\left(  X+x\left(  S\right)  \right)
^{\left\vert S\right\vert -1}$ is well-defined. Qed.)
\end{itemize}
}.

We shall prove Lemma \ref{lem.12} by induction over $n$.

Lemma \ref{lem.12} holds when $n=0$\ \ \ \ \footnote{\textit{Proof.} Assume
that $n=0$. We must prove that Lemma \ref{lem.12} holds.
\par
We have $\left\vert V\right\vert =n=0$. Hence, $V=\varnothing$. Thus, the only
subset of $V$ is the empty set $\varnothing$. Hence, the sum $\sum_{S\subseteq
V}\left(  X+x\left(  S\right)  \right)  ^{\left\vert S\right\vert }\left(
Y-x\left(  S\right)  \right)  ^{n-\left\vert S\right\vert }$ has only one
addend, namely the addend for $S=\varnothing$. Therefore, this sum simplifies
as follows:%
\begin{align}
&  \sum_{S\subseteq V}\left(  X+x\left(  S\right)  \right)  ^{\left\vert
S\right\vert }\left(  Y-x\left(  S\right)  \right)  ^{n-\left\vert
S\right\vert }\nonumber\\
&  =\left(  X+x\left(  \varnothing\right)  \right)  ^{\left\vert
\varnothing\right\vert }\left(  Y-x\left(  \varnothing\right)  \right)
^{n-\left\vert \varnothing\right\vert }=\underbrace{\left(  X+x\left(
\varnothing\right)  \right)  ^{0}}_{=1}\underbrace{\left(  Y-x\left(
\varnothing\right)  \right)  ^{0-0}}_{=\left(  Y-x\left(  \varnothing\right)
\right)  ^{0}=1}\nonumber\\
&  \ \ \ \ \ \ \ \ \ \ \left(  \text{since }\left\vert \varnothing\right\vert
=0\text{ and }n=0\right) \nonumber\\
&  =1. \label{pf.lem.12.indbase.fn1.1}%
\end{align}
\par
On the other hand, the set $V$ has no elements (since $V=\varnothing$). Thus,
the only $k$-tuple $\left(  i_{1},i_{2},\ldots,i_{k}\right)  $ of distinct
elements of $V$ is the empty $0$-tuple $\left(  {}\right)  $. Therefore, the
sum $\sum_{\substack{i_{1},i_{2},\ldots,i_{k}\text{ are}\\\text{distinct}%
\\\text{elements of }V}}\left(  X+Y\right)  ^{n-k}x_{i_{1}}x_{i_{2}}\cdots
x_{i_{k}}$ has only one addend, namely the addend for $k=0$ and $\left(
i_{1},i_{2},\ldots,i_{k}\right)  =\left(  {}\right)  $. Thus, this sum
simplifies as follows:%
\begin{align*}
&  \sum_{\substack{i_{1},i_{2},\ldots,i_{k}\text{ are}\\\text{distinct}%
\\\text{elements of }V}}\left(  X+Y\right)  ^{n-k}x_{i_{1}}x_{i_{2}}\cdots
x_{i_{k}}\\
&  =\left(  X+Y\right)  ^{n-0}\underbrace{\left(  \text{empty product}\right)
}_{=1}=\left(  X+Y\right)  ^{n-0}=\left(  X+Y\right)  ^{0-0}%
\ \ \ \ \ \ \ \ \ \ \left(  \text{since }n=0\right) \\
&  =\left(  X+Y\right)  ^{0}=1.
\end{align*}
Comparing this with (\ref{pf.lem.12.indbase.fn1.1}), we obtain
\begin{equation}
\sum_{S\subseteq V}\left(  X+x\left(  S\right)  \right)  ^{\left\vert
S\right\vert }\left(  Y-x\left(  S\right)  \right)  ^{n-\left\vert
S\right\vert }=\sum_{\substack{i_{1},i_{2},\ldots,i_{k}\text{ are}%
\\\text{distinct}\\\text{elements of }V}}\left(  X+Y\right)  ^{n-k}x_{i_{1}%
}x_{i_{2}}\cdots x_{i_{k}}.\nonumber
\end{equation}
In other words, (\ref{eq.lem12.1}) holds.
\par
Recall that the only subset of $V$ is the empty set $\varnothing$. Hence,
there exist no subset of $V$ distinct from $\varnothing$. In other words,
there exists no subset $S$ of $V$ satisfying $S\neq\varnothing$. Thus, the sum
$\sum_{\substack{S\subseteq V;\\S\neq\varnothing}}X\left(  X+x\left(
S\right)  \right)  ^{\left\vert S\right\vert -1}\left(  Y-x\left(  S\right)
\right)  ^{n-\left\vert S\right\vert }$ is empty. Hence,%
\[
\sum_{\substack{S\subseteq V;\\S\neq\varnothing}}X\left(  X+x\left(  S\right)
\right)  ^{\left\vert S\right\vert -1}\left(  Y-x\left(  S\right)  \right)
^{n-\left\vert S\right\vert }=\left(  \text{empty sum}\right)  =0.
\]
Thus,%
\begin{align*}
Y^{n}+\underbrace{\sum_{\substack{S\subseteq V;\\S\neq\varnothing}}X\left(
X+x\left(  S\right)  \right)  ^{\left\vert S\right\vert -1}\left(  Y-x\left(
S\right)  \right)  ^{n-\left\vert S\right\vert }}_{=0}  &  =Y^{n}%
=Y^{0}\ \ \ \ \ \ \ \ \ \ \left(  \text{since }n=0\right) \\
&  =1.
\end{align*}
Comparing this with%
\begin{align*}
\left(  X+Y\right)  ^{n}  &  =\left(  X+Y\right)  ^{0}%
\ \ \ \ \ \ \ \ \ \ \left(  \text{since }n=0\right) \\
&  =1,
\end{align*}
we obtain%
\[
Y^{n}+\sum_{\substack{S\subseteq V;\\S\neq\varnothing}}X\left(  X+x\left(
S\right)  \right)  ^{\left\vert S\right\vert -1}\left(  Y-x\left(  S\right)
\right)  ^{n-\left\vert S\right\vert }=\left(  X+Y\right)  ^{n}.
\]
In other words, (\ref{eq.lem12.2}) holds.
\par
Thus, we have shown that both (\ref{eq.lem12.1}) and (\ref{eq.lem12.2}) hold.
In other words, Lemma \ref{lem.12} holds. Qed.}. This completes the induction base.

\textit{Induction step:} Let $N$ be a positive integer. Assume (as the
induction hypothesis) that Lemma \ref{lem.12} holds when $n=N-1$. We must then
prove that Lemma \ref{lem.12} holds when $n=N$.

Let $V$, $n$, $x_{s}$, $X$ and $Y$ be as in Lemma \ref{lem.12}. Assume that
$n=N$. We are going to prove the equalities (\ref{eq.lem12.1}) and
(\ref{eq.lem12.2}).

We note that every subset $S$ of $V$ satisfying $S\neq\varnothing$ must
satisfy
\begin{equation}
\left\vert S\right\vert -1\geq0 \label{pf.lem.12.step.S-1>=0}%
\end{equation}
\footnote{\textit{Proof of (\ref{pf.lem.12.step.S-1>=0}):} Let $S$ be a subset
of $V$ satisfying $S\neq\varnothing$. Then, $\left\vert S\right\vert >0$
(since $S\neq\varnothing$). Hence, $\left\vert S\right\vert \geq1$ (since
$\left\vert S\right\vert $ is an integer), so that $\left\vert S\right\vert
-1\geq0$, qed.}. Hence, $\left(  X+x\left(  S\right)  \right)  ^{\left\vert
S\right\vert -1}$ is a well-defined element of $\mathbb{L}$ for every such $S$.

Let $t\in V$ be arbitrary. The set $V\setminus\left\{  t\right\}  $ is a
subset of the set $V$, and thus is finite (since the set $V$ is finite). From
$t\in V$, we obtain $\left\vert V\setminus\left\{  t\right\}  \right\vert
=\underbrace{\left\vert V\right\vert }_{=n=N}-\,1=N-1$. In other words,
$N-1=\left\vert V\setminus\left\{  t\right\}  \right\vert $. Also, recall that
$x_{s}$ is an element of $\mathbb{L}$ for each $s\in V$. Hence, $x_{s}$ is an
element of $\mathbb{L}$ for each $s\in V\setminus\left\{  t\right\}  $ (since
each $s\in V\setminus\left\{  t\right\}  $ is an element of $V$). Finally, of
course, we have $N-1=N-1$.

Thus, we can apply Lemma \ref{lem.12} to $V\setminus\left\{  t\right\}  $ and
$N-1$ instead of $V$ and $n$ (since we have assumed that Lemma \ref{lem.12}
holds when $n=N-1$). As the result, we conclude that%
\begin{align}
&  \sum_{S\subseteq V\setminus\left\{  t\right\}  }\left(  X+x\left(
S\right)  \right)  ^{\left\vert S\right\vert }\left(  Y-x\left(  S\right)
\right)  ^{N-1-\left\vert S\right\vert }\nonumber\\
&  =\sum_{\substack{i_{1},i_{2},\ldots,i_{k}\text{ are}\\\text{distinct}%
\\\text{elements of }V\setminus\left\{  t\right\}  }}\left(  X+Y\right)
^{N-1-k}x_{i_{1}}x_{i_{2}}\cdots x_{i_{k}} \label{pf.lem.12.step.ih1}%
\end{align}
and%
\begin{equation}
Y^{N-1}+\sum_{\substack{S\subseteq V\setminus\left\{  t\right\}
;\\S\neq\varnothing}}X\left(  X+x\left(  S\right)  \right)  ^{\left\vert
S\right\vert -1}\left(  Y-x\left(  S\right)  \right)  ^{N-1-\left\vert
S\right\vert }=\left(  X+Y\right)  ^{N-1}. \label{pf.lem.12.step.ih2}%
\end{equation}

Subtracting $Y^{N-1}$ from both sides of (\ref{pf.lem.12.step.ih2}), we obtain%
\begin{equation}
\sum_{\substack{S\subseteq V\setminus\left\{  t\right\}  ;\\S\neq\varnothing
}}X\left(  X+x\left(  S\right)  \right)  ^{\left\vert S\right\vert -1}\left(
Y-x\left(  S\right)  \right)  ^{N-1-\left\vert S\right\vert }=\left(
X+Y\right)  ^{N-1}-Y^{N-1}. \label{pf.lem.12.step.ih2b}%
\end{equation}

But $x_{t}$ is an element of $\mathbb{L}$ (since $x_{s}$ is an element of
$\mathbb{L}$ for each $s\in V$). Thus, $X+x_{t}$ and $Y-x_{t}$ are two
elements of $\mathbb{L}$. Their sum $\left(  X+x_{t}\right)  +\left(
Y-x_{t}\right)  $ lies in the center of $\mathbb{L}$ (because this sum equals
$\left(  X+x_{t}\right)  +\left(  Y-x_{t}\right)  =X+Y$, but we know that
$X+Y$ lies in the center of $\mathbb{L}$). Hence, we can apply Lemma
\ref{lem.12} to $V\setminus\left\{  t\right\}  $, $N-1$, $X+x_{t}$ and
$Y-x_{t}$ instead of $V$, $n$, $X$ and $Y$ (since we have assumed that Lemma
\ref{lem.12} holds when $n=N-1$). As a result, we conclude that%
\begin{align}
&  \sum_{S\subseteq V\setminus\left\{  t\right\}  }\left(  X+x_{t}+x\left(
S\right)  \right)  ^{\left\vert S\right\vert }\left(  Y-x_{t}-x\left(
S\right)  \right)  ^{N-1-\left\vert S\right\vert }\nonumber\\
&  =\sum_{\substack{i_{1},i_{2},\ldots,i_{k}\text{ are}\\\text{distinct}%
\\\text{elements of }V\setminus\left\{  t\right\}  }}\left(  X+x_{t}%
+Y-x_{t}\right)  ^{N-1-k}x_{i_{1}}x_{i_{2}}\cdots x_{i_{k}}
\label{pf.lem.12.step.ih1'}%
\end{align}
and%
\begin{align*}
&  \left(  Y-x_{t}\right)  ^{N-1} +\sum_{\substack{S\subseteq V\setminus
\left\{  t\right\}  ;\\S\neq\varnothing}}\left(  X+x_{t}\right)  \left(
X+x_{t}+x\left(  S\right)  \right)  ^{\left\vert S\right\vert -1}\left(
Y-x_{t}-x\left(  S\right)  \right)  ^{N-1-\left\vert S\right\vert }\\
&  =\left(  X+x_{t}+Y-x_{t}\right)  ^{N-1}.
\end{align*}

The equality (\ref{pf.lem.12.step.ih1'}) becomes%
\begin{align}
&  \sum_{S\subseteq V\setminus\left\{  t\right\}  }\left(  X+x_{t}+x\left(
S\right)  \right)  ^{\left\vert S\right\vert }\left(  Y-x_{t}-x\left(
S\right)  \right)  ^{N-1-\left\vert S\right\vert }\nonumber\\
&  =\sum_{\substack{i_{1},i_{2},\ldots,i_{k}\text{ are}\\\text{distinct}%
\\\text{elements of }V\setminus\left\{  t\right\}  }}\left(
\underbrace{X+x_{t}+Y-x_{t}}_{=X+Y}\right)  ^{N-1-k}x_{i_{1}}x_{i_{2}}\cdots
x_{i_{k}}\nonumber\\
&  =\sum_{\substack{i_{1},i_{2},\ldots,i_{k}\text{ are}\\\text{distinct}%
\\\text{elements of }V\setminus\left\{  t\right\}  }}\left(  X+Y\right)
^{N-1-k}x_{i_{1}}x_{i_{2}}\cdots x_{i_{k}}\label{pf.lem.12.step.ih1=1'o}\\
&  =\sum_{S\subseteq V\setminus\left\{  t\right\}  }\left(  X+x\left(
S\right)  \right)  ^{\left\vert S\right\vert }\left(  Y-x\left(  S\right)
\right)  ^{N-1-\left\vert S\right\vert }\ \ \ \ \ \ \ \ \ \ \left(  \text{by
(\ref{pf.lem.12.step.ih1})}\right) \label{pf.lem.12.step.ih1=1'}\\
&  =\underbrace{\left(  X+x\left(  \varnothing\right)  \right)  ^{\left\vert
\varnothing\right\vert }}_{\substack{=\left(  X+x\left(  \varnothing\right)
\right)  ^{0}\\\text{(since }\left\vert \varnothing\right\vert =0\text{)}%
}}\underbrace{\left(  Y-x\left(  \varnothing\right)  \right)  ^{N-1-\left\vert
\varnothing\right\vert }}_{\substack{=\left(  Y-0\right)  ^{N-1-0}%
\\\text{(since }x\left(  \varnothing\right)  =0\text{ (by Lemma
\ref{lem.x(S).triv} \textbf{(a)})}\\\text{and }\left\vert \varnothing
\right\vert =0\text{)}}}\nonumber\\
&  \ \ \ \ \ \ \ \ \ \ +\sum_{\substack{S\subseteq V\setminus\left\{
t\right\}  ;\\S\neq\varnothing}}\left(  X+x\left(  S\right)  \right)
^{\left\vert S\right\vert }\underbrace{\left(  Y-x\left(  S\right)  \right)
^{N-1-\left\vert S\right\vert }}_{\substack{=\left(  Y-x\left(  S\right)
\right)  ^{N-\left\vert S\right\vert -1}\\\text{(since }N-1-\left\vert
S\right\vert =N-\left\vert S\right\vert -1\text{)}}}\nonumber\\
&  \ \ \ \ \ \ \ \ \ \ \ \ \ \ \ \ \ \ \ \ \left(
\begin{array}
[c]{c}%
\text{here, we have split off the addend for }S=\varnothing\text{ from the
sum}\\
\text{(since }\varnothing\text{ is a subset of }V\setminus\left\{  t\right\}
\text{)}%
\end{array}
\right) \nonumber\\
&  =\underbrace{\left(  X+x\left(  \varnothing\right)  \right)  ^{0}}%
_{=1}\underbrace{\left(  Y-0\right)  ^{N-1-0}}_{\substack{=Y^{N-1}%
\\\text{(since }Y-0=Y\\\text{and }N-1-0=N-1\text{)}}}+\sum
_{\substack{S\subseteq V\setminus\left\{  t\right\}  ;\\S\neq\varnothing
}}\left(  X+x\left(  S\right)  \right)  ^{\left\vert S\right\vert }\left(
Y-x\left(  S\right)  \right)  ^{N-\left\vert S\right\vert -1}\nonumber\\
&  =Y^{N-1}+\sum_{\substack{S\subseteq V\setminus\left\{  t\right\}
;\\S\neq\varnothing}}\left(  X+x\left(  S\right)  \right)  ^{\left\vert
S\right\vert }\left(  Y-x\left(  S\right)  \right)  ^{N-\left\vert
S\right\vert -1}. \label{pf.lem.12.step.ih1=1'b}%
\end{align}

Let us record the following (well-known and simple) fact: The map%
\begin{align}
\left\{  S\subseteq V\ \mid\ t\notin S\right\}   &  \rightarrow\left\{
S\subseteq V\ \mid\ t\in S\right\}  ,\nonumber\\
S  &  \mapsto S\cup\left\{  t\right\}  \label{pf.lem.12.step.bij}%
\end{align}
is a bijection\footnote{Its inverse is the map
\begin{align*}
\left\{  S\subseteq V\ \mid\ t\in S\right\}   &  \rightarrow\left\{
S\subseteq V\ \mid\ t\notin S\right\}  ,\\
T  &  \mapsto T\setminus\left\{  t\right\}  .
\end{align*}
}.

Any subset $S$ of $V$ satisfying $t\notin S$ must satisfy%
\begin{equation}
N-\left\vert S\right\vert \geq1 \label{pf.lem.12.step.pos1}%
\end{equation}
\footnote{\textit{Proof of (\ref{pf.lem.12.step.pos1}):} Let $S$ be a subset
of $V$ satisfying $t\notin S$. We have $S\subseteq V\setminus\left\{
t\right\}  $ (since $S\subseteq V$ and $t\notin S$), and thus $\left\vert
S\right\vert \leq\left\vert V\setminus\left\{  t\right\}  \right\vert =N-1$.
In other words, $\left(  N-1\right)  -\left\vert S\right\vert \geq0$. Thus,
$N-\left\vert S\right\vert -1=\left(  N-1\right)  -\left\vert S\right\vert
\geq0$, so that $N-\left\vert S\right\vert \geq1$. This proves
(\ref{pf.lem.12.step.pos1}).}.

Any subset $S$ of $V$ satisfying $t\in S$ must automatically satisfy
$S\neq\varnothing$ (because $S$ has at least one element (namely, $t$)). Thus,
for a subset $S$ of $V$, the condition $\left(  t\in S\text{ and }%
S\neq\varnothing\right)  $ is equivalent to the condition $\left(  t\in
S\right)  $. Hence, we have the following equality of summation signs:%
\begin{equation}
\sum_{\substack{S\subseteq V;\\t\in S;\\S\neq\varnothing}}=\sum
_{\substack{S\subseteq V;\\t\in S}}. \label{pf.lem.12.step.sumeq1}%
\end{equation}
Now,%
\begin{align}
&  \underbrace{\sum_{\substack{S\subseteq V;\\S\neq\varnothing;\\t\in S}%
}}_{\substack{_{\substack{=\sum_{\substack{S\subseteq V;\\t\in S;\\S\neq
\varnothing}}=\sum_{\substack{S\subseteq V;\\t\in S}}}}\\\text{(by
(\ref{pf.lem.12.step.sumeq1}))}}}\left(  X+x\left(  S\right)  \right)
^{\left\vert S\right\vert -1}\left(  Y-x\left(  S\right)  \right)
^{N-\left\vert S\right\vert }\nonumber\\
&  =\sum_{\substack{S\subseteq V;\\t\in S}}\left(  X+x\left(  S\right)
\right)  ^{\left\vert S\right\vert -1}\left(  Y-x\left(  S\right)  \right)
^{N-\left\vert S\right\vert }\nonumber\\
&  =\underbrace{\sum_{\substack{S\subseteq V;\\t\notin S}}}_{=\sum_{S\subseteq
V\setminus\left\{  t\right\}  }}\underbrace{\left(  X+x\left(  S\cup\left\{
t\right\}  \right)  \right)  ^{\left\vert S\cup\left\{  t\right\}  \right\vert
-1}\left(  Y-x\left(  S\cup\left\{  t\right\}  \right)  \right)
^{N-\left\vert S\cup\left\{  t\right\}  \right\vert }}_{\substack{=\left(
X+x\left(  S\cup\left\{  t\right\}  \right)  \right)  ^{\left(  \left\vert
S\right\vert +1\right)  -1}\left(  Y-x\left(  S\cup\left\{  t\right\}
\right)  \right)  ^{N-\left(  \left\vert S\right\vert +1\right)
}\\\text{(since }\left\vert S\cup\left\{  t\right\}  \right\vert =\left\vert
S\right\vert +1\text{ (since }t\notin S\text{))}}}\nonumber\\
&  \ \ \ \ \ \ \ \ \ \ \ \ \ \ \ \ \ \ \ \ \left(
\begin{array}
[c]{c}%
\text{here, we have substituted }S\cup\left\{  t\right\}  \text{ for }S\text{
in the sum,}\\
\text{since the map (\ref{pf.lem.12.step.bij}) is a bijection}%
\end{array}
\right) \nonumber\\
&  =\sum_{S\subseteq V\setminus\left\{  t\right\}  }\left(
X+\underbrace{x\left(  S\cup\left\{  t\right\}  \right)  }_{\substack{=x_{t}%
+x\left(  S\right)  \\\text{(by Lemma \ref{lem.x(S).triv} \textbf{(b)})}%
}}\right)  ^{\left(  \left\vert S\right\vert +1\right)  -1}\left(
Y-\underbrace{x\left(  S\cup\left\{  t\right\}  \right)  }_{\substack{=x_{t}%
+x\left(  S\right)  \\\text{(by Lemma \ref{lem.x(S).triv} \textbf{(b)})}%
}}\right)  ^{N-\left(  \left\vert S\right\vert +1\right)  }\nonumber\\
&  =\sum_{S\subseteq V\setminus\left\{  t\right\}  }\underbrace{\left(
X+x_{t}+x\left(  S\right)  \right)  ^{\left(  \left\vert S\right\vert
+1\right)  -1}}_{\substack{=\left(  X+x_{t}+x\left(  S\right)  \right)
^{\left\vert S\right\vert }\\\text{(since }\left(  \left\vert S\right\vert
+1\right)  -1=\left\vert S\right\vert \text{)}}}\underbrace{\left(  Y-\left(
x_{t}+x\left(  S\right)  \right)  \right)  ^{N-\left(  \left\vert S\right\vert
+1\right)  }}_{\substack{=\left(  Y-x_{t}-x\left(  S\right)  \right)
^{N-1-\left\vert S\right\vert }\\\text{(since }Y-\left(  x_{t}+x\left(
S\right)  \right)  =Y-x_{t}-x\left(  S\right)  \\\text{and }N-\left(
\left\vert S\right\vert +1\right)  =N-1-\left\vert S\right\vert \text{)}%
}}\nonumber\\
&  =\sum_{S\subseteq V\setminus\left\{  t\right\}  }\left(  X+x_{t}+x\left(
S\right)  \right)  ^{\left\vert S\right\vert }\left(  Y-x_{t}-x\left(
S\right)  \right)  ^{N-1-\left\vert S\right\vert } \label{pf.lem.12.step.sum0}%
\\
&  =Y^{N-1}+\sum_{\substack{S\subseteq V\setminus\left\{  t\right\}
;\\S\neq\varnothing}}\left(  X+x\left(  S\right)  \right)  ^{\left\vert
S\right\vert }\left(  Y-x\left(  S\right)  \right)  ^{N-\left\vert
S\right\vert -1}\ \ \ \ \ \ \ \ \ \ \left(  \text{by
(\ref{pf.lem.12.step.ih1=1'b})}\right)  .\nonumber
\end{align}
Multiplying both sides of this equality by $X$, we obtain%
\begin{align}
&  X\sum_{\substack{S\subseteq V;\\S\neq\varnothing;\\t\in S}}\left(
X+x\left(  S\right)  \right)  ^{\left\vert S\right\vert -1}\left(  Y-x\left(
S\right)  \right)  ^{N-\left\vert S\right\vert }\nonumber\\
&  =X\left(  Y^{N-1}+\sum_{\substack{S\subseteq V\setminus\left\{  t\right\}
;\\S\neq\varnothing}}\left(  X+x\left(  S\right)  \right)  ^{\left\vert
S\right\vert }\left(  Y-x\left(  S\right)  \right)  ^{N-\left\vert
S\right\vert -1}\right) \nonumber\\
&  =XY^{N-1}+\underbrace{X\sum_{\substack{S\subseteq V\setminus\left\{
t\right\}  ;\\S\neq\varnothing}}\left(  X+x\left(  S\right)  \right)
^{\left\vert S\right\vert }\left(  Y-x\left(  S\right)  \right)
^{N-\left\vert S\right\vert -1}}_{=\sum_{\substack{S\subseteq V\setminus
\left\{  t\right\}  ;\\S\neq\varnothing}}X\left(  X+x\left(  S\right)
\right)  ^{\left\vert S\right\vert }\left(  Y-x\left(  S\right)  \right)
^{N-\left\vert S\right\vert -1}}\nonumber\\
&  =XY^{N-1}+\sum_{\substack{S\subseteq V\setminus\left\{  t\right\}
;\\S\neq\varnothing}}X\underbrace{\left(  X+x\left(  S\right)  \right)
^{\left\vert S\right\vert }}_{=\left(  X+x\left(  S\right)  \right)
^{\left\vert S\right\vert -1}\left(  X+x\left(  S\right)  \right)  }\left(
Y-x\left(  S\right)  \right)  ^{N-\left\vert S\right\vert -1}\nonumber\\
&  =XY^{N-1}+\sum_{\substack{S\subseteq V\setminus\left\{  t\right\}
;\\S\neq\varnothing}}X\left(  X+x\left(  S\right)  \right)  ^{\left\vert
S\right\vert -1}\left(  X+x\left(  S\right)  \right)  \left(  Y-x\left(
S\right)  \right)  ^{N-\left\vert S\right\vert -1}.
\label{pf.lem.12.step.sum1}%
\end{align}

Each subset $S$ of $V$ satisfies either $t\in S$ or $t\notin S$ (but not
both). Hence,
\begin{align*}
&  \sum_{\substack{S\subseteq V;\\S\neq\varnothing}}X\left(  X+x\left(
S\right)  \right)  ^{\left\vert S\right\vert -1}\left(  Y-x\left(  S\right)
\right)  ^{N-\left\vert S\right\vert }\\
&  =\underbrace{\sum_{\substack{S\subseteq V;\\S\neq\varnothing;\\t\in
S}}X\left(  X+x\left(  S\right)  \right)  ^{\left\vert S\right\vert -1}\left(
Y-x\left(  S\right)  \right)  ^{N-\left\vert S\right\vert }}_{\substack{=X\sum
_{\substack{S\subseteq V;\\S\neq\varnothing;\\t\in S}}\left(  X+x\left(
S\right)  \right)  ^{\left\vert S\right\vert -1}\left(  Y-x\left(  S\right)
\right)  ^{N-\left\vert S\right\vert }\\=XY^{N-1}+\sum_{\substack{S\subseteq
V\setminus\left\{  t\right\}  ;\\S\neq\varnothing}}X\left(  X+x\left(
S\right)  \right)  ^{\left\vert S\right\vert -1}\left(  X+x\left(  S\right)
\right)  \left(  Y-x\left(  S\right)  \right)  ^{N-\left\vert S\right\vert
-1}\\\text{(by (\ref{pf.lem.12.step.sum1}))}}}\\
&  \ \ \ \ \ \ \ \ \ \ +\underbrace{\sum_{\substack{S\subseteq V;\\S\neq
\varnothing\\t\notin S}}}_{\substack{=\sum_{\substack{S\subseteq V;\\t\notin
S;\\S\neq\varnothing}}=\sum_{\substack{S\subseteq V\setminus\left\{
t\right\}  ;\\S\neq\varnothing}}\\\text{(since the subsets }S\text{ of
}V\\\text{satisfying }t\notin S\text{ are precisely}\\\text{the subsets of
}V\setminus\left\{  t\right\}  \text{)}}}X\left(  X+x\left(  S\right)
\right)  ^{\left\vert S\right\vert -1}\underbrace{\left(  Y-x\left(  S\right)
\right)  ^{N-\left\vert S\right\vert }}_{\substack{=\left(  Y-x\left(
S\right)  \right)  \left(  Y-x\left(  S\right)  \right)  ^{N-\left\vert
S\right\vert -1}\\\text{(since }N-\left\vert S\right\vert \geq1\\\text{(by
(\ref{pf.lem.12.step.pos1})))}}}\\
&  =XY^{N-1}+\sum_{\substack{S\subseteq V\setminus\left\{  t\right\}
;\\S\neq\varnothing}}X\left(  X+x\left(  S\right)  \right)  ^{\left\vert
S\right\vert -1}\left(  X+x\left(  S\right)  \right)  \left(  Y-x\left(
S\right)  \right)  ^{N-\left\vert S\right\vert -1}\\
&  \ \ \ \ \ \ \ \ \ \ +\sum_{\substack{S\subseteq V\setminus\left\{
t\right\}  ;\\S\neq\varnothing}}X\left(  X+x\left(  S\right)  \right)
^{\left\vert S\right\vert -1}\left(  Y-x\left(  S\right)  \right)  \left(
Y-x\left(  S\right)  \right)  ^{N-\left\vert S\right\vert -1}.
\end{align*}
Subtracting $XY^{N-1}$ from both sides of this equality, we obtain%
\begin{align*}
&  \sum_{\substack{S\subseteq V;\\S\neq\varnothing}}X\left(  X+x\left(
S\right)  \right)  ^{\left\vert S\right\vert -1}\left(  Y-x\left(  S\right)
\right)  ^{N-\left\vert S\right\vert }-XY^{N-1}\\
&  =\sum_{\substack{S\subseteq V\setminus\left\{  t\right\}  ;\\S\neq
\varnothing}}X\left(  X+x\left(  S\right)  \right)  ^{\left\vert S\right\vert
-1}\left(  X+x\left(  S\right)  \right)  \left(  Y-x\left(  S\right)  \right)
^{N-\left\vert S\right\vert -1}\\
&  \ \ \ \ \ \ \ \ \ \ +\sum_{\substack{S\subseteq V\setminus\left\{
t\right\}  ;\\S\neq\varnothing}}X\left(  X+x\left(  S\right)  \right)
^{\left\vert S\right\vert -1}\left(  Y-x\left(  S\right)  \right)  \left(
Y-x\left(  S\right)  \right)  ^{N-\left\vert S\right\vert -1}\\
&  =\sum_{\substack{S\subseteq V\setminus\left\{  t\right\}  ;\\S\neq
\varnothing}}\left(  X\left(  X+x\left(  S\right)  \right)  ^{\left\vert
S\right\vert -1}\left(  X+x\left(  S\right)  \right)  \left(  Y-x\left(
S\right)  \right)  ^{N-\left\vert S\right\vert -1}\right. \\
&  \ \ \ \ \ \ \ \ \ \ \ \ \ \ \ \ \ \ \ \ \left.  +X\left(  X+x\left(
S\right)  \right)  ^{\left\vert S\right\vert -1}\left(  Y-x\left(  S\right)
\right)  \left(  Y-x\left(  S\right)  \right)  ^{N-\left\vert S\right\vert
-1}\right) \\
&  =\sum_{\substack{S\subseteq V\setminus\left\{  t\right\}  ;\\S\neq
\varnothing}}X\left(  X+x\left(  S\right)  \right)  ^{\left\vert S\right\vert
-1}\underbrace{\left(  \left(  X+x\left(  S\right)  \right)  +\left(
Y-x\left(  S\right)  \right)  \right)  }_{=X+Y}\left(  Y-x\left(  S\right)
\right)  ^{N-\left\vert S\right\vert -1}\\
&  \ \ \ \ \ \ \ \ \ \ \ \ \ \ \ \ \ \ \ \ \left(
\begin{array}
[c]{c}%
\text{since each subset }S\text{ of }V\setminus\left\{  t\right\}  \text{
satisfying }S\neq\varnothing\text{ satisfies}\\
X\left(  X+x\left(  S\right)  \right)  ^{\left\vert S\right\vert -1}\left(
X+x\left(  S\right)  \right)  \left(  Y-x\left(  S\right)  \right)
^{N-\left\vert S\right\vert -1}\\
+X\left(  X+x\left(  S\right)  \right)  ^{\left\vert S\right\vert -1}\left(
Y-x\left(  S\right)  \right)  \left(  Y-x\left(  S\right)  \right)
^{N-\left\vert S\right\vert -1}\\
=X\left(  X+x\left(  S\right)  \right)  ^{\left\vert S\right\vert -1}\left(
\left(  X+x\left(  S\right)  \right)  +\left(  Y-x\left(  S\right)  \right)
\right)  \left(  Y-x\left(  S\right)  \right)  ^{N-\left\vert S\right\vert -1}%
\end{array}
\right) \\
&  =\sum_{\substack{S\subseteq V\setminus\left\{  t\right\}  ;\\S\neq
\varnothing}}\underbrace{X\left(  X+x\left(  S\right)  \right)  ^{\left\vert
S\right\vert -1}\left(  X+Y\right)  }_{\substack{=\left(  X+Y\right)  X\left(
X+x\left(  S\right)  \right)  ^{\left\vert S\right\vert -1}\\\text{(since
}X+Y\text{ commutes with }X\left(  X+x\left(  S\right)  \right)  ^{\left\vert
S\right\vert -1}\\\text{(since }X+Y\text{ lies in the center of }%
\mathbb{L}\text{))}}}\ \ \underbrace{\left(  Y-x\left(  S\right)  \right)
^{N-\left\vert S\right\vert -1}}_{\substack{=\left(  Y-x\left(  S\right)
\right)  ^{N-1-\left\vert S\right\vert }\\\text{(since }N-\left\vert
S\right\vert -1=N-1-\left\vert S\right\vert \text{)}}}\\
&  =\sum_{\substack{S\subseteq V\setminus\left\{  t\right\}  ;\\S\neq
\varnothing}}\left(  X+Y\right)  X\left(  X+x\left(  S\right)  \right)
^{\left\vert S\right\vert -1}\left(  Y-x\left(  S\right)  \right)
^{N-1-\left\vert S\right\vert }\\
&  =\left(  X+Y\right)  \underbrace{\sum_{\substack{S\subseteq V\setminus
\left\{  t\right\}  ;\\S\neq\varnothing}}X\left(  X+x\left(  S\right)
\right)  ^{\left\vert S\right\vert -1}\left(  Y-x\left(  S\right)  \right)
^{N-1-\left\vert S\right\vert }}_{\substack{=\left(  X+Y\right)
^{N-1}-Y^{N-1}\\\text{(by (\ref{pf.lem.12.step.ih2b}))}}}\\
&  =\left(  X+Y\right)  \left(  \left(  X+Y\right)  ^{N-1}-Y^{N-1}\right) \\
&  =\underbrace{\left(  X+Y\right)  \left(  X+Y\right)  ^{N-1}}_{=\left(
X+Y\right)  ^{N}}-\underbrace{\left(  X+Y\right)  Y^{N-1}}_{=XY^{N-1}%
+YY^{N-1}}\\
&  =\left(  X+Y\right)  ^{N}-\left(  XY^{N-1}+YY^{N-1}\right)  =\left(
X+Y\right)  ^{N}-XY^{N-1}-YY^{N-1}.
\end{align*}
Adding $XY^{N-1}$ to both sides of this equality, we obtain%
\begin{align}
&  \sum_{\substack{S\subseteq V;\\S\neq\varnothing}}X\left(  X+x\left(
S\right)  \right)  ^{\left\vert S\right\vert -1}\left(  Y-x\left(  S\right)
\right)  ^{N-\left\vert S\right\vert }\nonumber\\
&  =\left(  X+Y\right)  ^{N}-XY^{N-1}-YY^{N-1}+XY^{N-1}=\left(  X+Y\right)
^{N}-\underbrace{YY^{N-1}}_{=Y^{N}}\nonumber\\
&  =\left(  X+Y\right)  ^{N}-Y^{N}. \label{pf.lem.12.step.part1}%
\end{align}

Now, forget that we fixed $t$. We thus have proven the equalities
(\ref{pf.lem.12.step.sum0}), (\ref{pf.lem.12.step.part1}) and
(\ref{pf.lem.12.step.ih1=1'o}) for each $t\in V$.

The set $V$ is nonempty (since $\left\vert V\right\vert =N>0$ (since $N$ is a
positive integer)). Hence, there exists some $q\in V$. Consider this $q$.
Applying (\ref{pf.lem.12.step.part1}) to $t=q$, we obtain%
\[
\sum_{\substack{S\subseteq V;\\S\neq\varnothing}}X\left(  X+x\left(  S\right)
\right)  ^{\left\vert S\right\vert -1}\left(  Y-x\left(  S\right)  \right)
^{N-\left\vert S\right\vert }=\left(  X+Y\right)  ^{N}-Y^{N}.
\]
Adding $Y^{N}$ to both sides of this equality, we obtain%
\begin{equation}
Y^{N}+\sum_{\substack{S\subseteq V;\\S\neq\varnothing}}X\left(  X+x\left(
S\right)  \right)  ^{\left\vert S\right\vert -1}\left(  Y-x\left(  S\right)
\right)  ^{N-\left\vert S\right\vert }=\left(  X+Y\right)  ^{N}.
\label{pf.lem.12.step.part1'}%
\end{equation}
Since $n=N$, we can rewrite this equality as follows:%
\[
Y^{n}+\sum_{\substack{S\subseteq V;\\S\neq\varnothing}}X\left(  X+x\left(
S\right)  \right)  ^{\left\vert S\right\vert -1}\left(  Y-x\left(  S\right)
\right)  ^{n-\left\vert S\right\vert }=\left(  X+Y\right)  ^{n}.
\]
In other words, the equality (\ref{eq.lem12.2}) holds.

On the other hand, every $t\in V$ satisfies%
\begin{align}
&  \sum_{\substack{S\subseteq V;\\S\neq\varnothing;\\t\in S}}\left(
X+x\left(  S\right)  \right)  ^{\left\vert S\right\vert -1}\left(  Y-x\left(
S\right)  \right)  ^{N-\left\vert S\right\vert }\nonumber\\
&  =\sum_{S\subseteq V\setminus\left\{  t\right\}  }\left(  X+x_{t}+x\left(
S\right)  \right)  ^{\left\vert S\right\vert }\left(  Y-x_{t}-x\left(
S\right)  \right)  ^{N-1-\left\vert S\right\vert }\ \ \ \ \ \ \ \ \ \ \left(
\text{by (\ref{pf.lem.12.step.sum0})}\right) \nonumber\\
&  =\sum_{\substack{i_{1},i_{2},\ldots,i_{k}\text{ are}\\\text{distinct}%
\\\text{elements of }V\setminus\left\{  t\right\}  }}\left(  X+Y\right)
^{N-1-k}x_{i_{1}}x_{i_{2}}\cdots x_{i_{k}}\ \ \ \ \ \ \ \ \ \ \left(  \text{by
(\ref{pf.lem.12.step.ih1=1'o})}\right)  . \label{pf.lem.12.step.innersum}%
\end{align}
Now, every subset $S$ of $V$ satisfies%
\begin{align}
x\left(  S\right)   &  =\underbrace{\sum_{s\in S}}_{\substack{=\sum
_{\substack{s\in V;\\s\in S}}\\\text{(since }S\subseteq V\text{)}}%
}x_{s}\ \ \ \ \ \ \ \ \ \ \left(  \text{by the definition of }x\left(
S\right)  \right) \nonumber\\
&  =\sum_{\substack{s\in V;\\s\in S}}x_{s}=\sum_{\substack{t\in V;\\t\in
S}}x_{t} \label{pf.lem.12.step.St}%
\end{align}
(here, we have substituted $t$ for $s$ in the sum).

Now,
\begin{align*}
&  \sum_{\substack{S\subseteq V;\\S\neq\varnothing}}\underbrace{x\left(
S\right)  }_{\substack{=\sum_{\substack{t\in V;\\t\in S}}x_{t}\\\text{(by
(\ref{pf.lem.12.step.St}))}}}\left(  X+x\left(  S\right)  \right)
^{\left\vert S\right\vert -1}\left(  Y-x\left(  S\right)  \right)
^{N-\left\vert S\right\vert }\\
&  =\sum_{\substack{S\subseteq V;\\S\neq\varnothing}}\left(  \sum
_{\substack{t\in V;\\t\in S}}x_{t}\right)  \left(  X+x\left(  S\right)
\right)  ^{\left\vert S\right\vert -1}\left(  Y-x\left(  S\right)  \right)
^{N-\left\vert S\right\vert }\\
&  =\underbrace{\sum_{\substack{S\subseteq V;\\S\neq\varnothing}%
}\ \ \sum_{\substack{t\in V;\\t\in S}}}_{=\sum_{t\in V}\ \ \sum
_{\substack{S\subseteq V;\\S\neq\varnothing;\\t\in S}}}x_{t}\left(  X+x\left(
S\right)  \right)  ^{\left\vert S\right\vert -1}\left(  Y-x\left(  S\right)
\right)  ^{N-\left\vert S\right\vert }\\
&  =\sum_{t\in V}\ \ \underbrace{\sum_{\substack{S\subseteq V;\\S\neq
\varnothing;\\t\in S}}x_{t}\left(  X+x\left(  S\right)  \right)  ^{\left\vert
S\right\vert -1}\left(  Y-x\left(  S\right)  \right)  ^{N-\left\vert
S\right\vert }}_{=x_{t}\sum_{\substack{S\subseteq V;\\S\neq\varnothing;\\t\in
S}}\left(  X+x\left(  S\right)  \right)  ^{\left\vert S\right\vert -1}\left(
Y-x\left(  S\right)  \right)  ^{N-\left\vert S\right\vert }}\\
&  =\sum_{t\in V}x_{t}\underbrace{\sum_{\substack{S\subseteq V;\\S\neq
\varnothing;\\t\in S}}\left(  X+x\left(  S\right)  \right)  ^{\left\vert
S\right\vert -1}\left(  Y-x\left(  S\right)  \right)  ^{N-\left\vert
S\right\vert }}_{\substack{=\sum_{\substack{i_{1},i_{2},\ldots,i_{k}\text{
are}\\\text{distinct}\\\text{elements of }V\setminus\left\{  t\right\}
}}\left(  X+Y\right)  ^{N-1-k}x_{i_{1}}x_{i_{2}}\cdots x_{i_{k}}\\\text{(by
(\ref{pf.lem.12.step.innersum}))}}}
\end{align*}%
\begin{align}
&  =\sum_{t\in V}\underbrace{x_{t}\sum_{\substack{i_{1},i_{2},\ldots
,i_{k}\text{ are}\\\text{distinct}\\\text{elements of }V\setminus\left\{
t\right\}  }}\left(  X+Y\right)  ^{N-1-k}x_{i_{1}}x_{i_{2}}\cdots x_{i_{k}}%
}_{=\sum_{\substack{i_{1},i_{2},\ldots,i_{k}\text{ are}\\\text{distinct}%
\\\text{elements of }V\setminus\left\{  t\right\}  }}x_{t}\left(  X+Y\right)
^{N-1-k}x_{i_{1}}x_{i_{2}}\cdots x_{i_{k}}}\nonumber\\
&  =\sum_{t\in V}\sum_{\substack{i_{1},i_{2},\ldots,i_{k}\text{ are}%
\\\text{distinct}\\\text{elements of }V\setminus\left\{  t\right\}
}}\underbrace{x_{t}\left(  X+Y\right)  ^{N-1-k}}_{\substack{=\left(
X+Y\right)  ^{N-1-k}x_{t}\\\text{(since }\left(  X+Y\right)  ^{N-1-k}\text{
commutes with }x_{t}\\\text{(since }\left(  X+Y\right)  ^{N-1-k}\text{ lies in
the center of }\mathbb{L}\\\text{(since }X+Y\text{ lies in the center of
}\mathbb{L}\text{,}\\\text{but the center of }\mathbb{L}\text{ is a subring of
}\mathbb{L}\text{)))}}}x_{i_{1}}x_{i_{2}}\cdots x_{i_{k}}\nonumber\\
&  =\sum_{t\in V}\underbrace{\sum_{\substack{i_{1},i_{2},\ldots,i_{k}\text{
are}\\\text{distinct}\\\text{elements of }V\setminus\left\{  t\right\}  }%
}}_{=\sum_{k\in\mathbb{N}}\sum_{\substack{i_{1},i_{2},\ldots,i_{k}\text{
are}\\\text{distinct}\\\text{elements of }V\setminus\left\{  t\right\}  }%
}}\left(  X+Y\right)  ^{N-1-k}x_{t}\left(  x_{i_{1}}x_{i_{2}}\cdots x_{i_{k}%
}\right) \nonumber\\
&  =\sum_{t\in V}\sum_{k\in\mathbb{N}}\sum_{\substack{i_{1},i_{2},\ldots
,i_{k}\text{ are}\\\text{distinct}\\\text{elements of }V\setminus\left\{
t\right\}  }}\left(  X+Y\right)  ^{N-1-k}x_{t}\left(  x_{i_{1}}x_{i_{2}}\cdots
x_{i_{k}}\right)  . \label{pf.lem.12.step.u1}%
\end{align}

Let us next observe that if $i_{1},i_{2},\ldots,i_{k}$ are distinct elements
of $V$ (for some $k\in\mathbb{N}$), then $N-k\in\mathbb{N}$%
\ \ \ \ \footnote{\textit{Proof.} Let $i_{1},i_{2},\ldots,i_{k}$ be distinct
elements of $V$. Then, $\left\{  i_{1},i_{2},\ldots,i_{k}\right\}  $ is a
subset of $V$. Hence, $\left\vert \left\{  i_{1},i_{2},\ldots,i_{k}\right\}
\right\vert \leq\left\vert V\right\vert =n=N$. But the elements $i_{1}%
,i_{2},\ldots,i_{k}$ are distinct; hence, $\left\vert \left\{  i_{1}%
,i_{2},\ldots,i_{k}\right\}  \right\vert =k$. Thus, $k=\left\vert \left\{
i_{1},i_{2},\ldots,i_{k}\right\}  \right\vert \leq N$. Hence, $N-k\geq0$.
Therefore, $N-k\in\mathbb{N}$. Qed.}, and therefore the term $\left(
X+Y\right)  ^{N-k}$ is a well-defined element of $\mathbb{L}$. Hence, the sum
$\sum_{\substack{i_{1},i_{2},\ldots,i_{k}\text{ are}\\\text{distinct}%
\\\text{elements of }V}}\left(  X+Y\right)  ^{N-k}x_{i_{1}}x_{i_{2}}\cdots
x_{i_{k}}$ is well-defined.

On the other hand, consider the sum $\sum_{\substack{i_{1},i_{2},\ldots
,i_{0}\text{ are}\\\text{distinct}\\\text{elements of }V}}\left(  X+Y\right)
^{N}$. This sum has only one addend (since there exists only one $0$-tuple
$\left(  i_{1},i_{2},\ldots,i_{0}\right)  $ of distinct elements of $V$,
namely the empty $0$-tuple $\left(  {}\right)  $), namely the addend for
$\left(  i_{1},i_{2},\ldots,i_{0}\right)  =\left(  {}\right)  $. Thus, this
sum simplifies as follows:%
\[
\sum_{\substack{i_{1},i_{2},\ldots,i_{0}\text{ are}\\\text{distinct}%
\\\text{elements of }V}}\left(  X+Y\right)  ^{N}=\left(  X+Y\right)  ^{N}.
\]
Hence,%
\[
\sum_{\substack{i_{1},i_{2},\ldots,i_{0}\text{ are}\\\text{distinct}%
\\\text{elements of }V}}\underbrace{\left(  X+Y\right)  ^{N-0}}_{=\left(
X+Y\right)  ^{N}}\underbrace{x_{i_{1}}x_{i_{2}}\cdots x_{i_{0}}}_{=\left(
\text{empty product}\right)  =1}=\sum_{\substack{i_{1},i_{2},\ldots
,i_{0}\text{ are}\\\text{distinct}\\\text{elements of }V}}\left(  X+Y\right)
^{N}=\left(  X+Y\right)  ^{N}.
\]

Now,%
\begin{align*}
&  \sum_{\substack{i_{1},i_{2},\ldots,i_{k}\text{ are}\\\text{distinct}%
\\\text{elements of }V}}\left(  X+Y\right)  ^{N-k}x_{i_{1}}x_{i_{2}}\cdots
x_{i_{k}}\\
&  =\underbrace{\sum_{\substack{i_{1},i_{2},\ldots,i_{k}\text{ are}%
\\\text{distinct}\\\text{elements of }V;\\k=0}}\left(  X+Y\right)
^{N-k}x_{i_{1}}x_{i_{2}}\cdots x_{i_{k}}}_{\substack{=\sum_{\substack{i_{1}%
,i_{2},\ldots,i_{0}\text{ are}\\\text{distinct}\\\text{elements of }V}}\left(
X+Y\right)  ^{N-0}x_{i_{1}}x_{i_{2}}\cdots x_{i_{0}}\\=\left(  X+Y\right)
^{N}}}+\sum_{\substack{i_{1},i_{2},\ldots,i_{k}\text{ are}\\\text{distinct}%
\\\text{elements of }V;\\k\neq0}}\left(  X+Y\right)  ^{N-k}x_{i_{1}}x_{i_{2}%
}\cdots x_{i_{k}}\\
&  \ \ \ \ \ \ \ \ \ \ \ \ \ \ \ \ \ \ \ \ \left(
\begin{array}
[c]{c}%
\text{since every }k\text{-tuple }\left(  i_{1},i_{2},\ldots,i_{k}\right)
\text{ satisfies}\\
\text{either }k=0\text{ or }k\neq0\text{ (but not both)}%
\end{array}
\right) \\
&  =\left(  X+Y\right)  ^{N}+\sum_{\substack{i_{1},i_{2},\ldots,i_{k}\text{
are}\\\text{distinct}\\\text{elements of }V;\\k\neq0}}\left(  X+Y\right)
^{N-k}x_{i_{1}}x_{i_{2}}\cdots x_{i_{k}}.
\end{align*}
Subtracting $\left(  X+Y\right)  ^{N}$ from both sides of this equality, we
obtain%
\begin{align}
&  \sum_{\substack{i_{1},i_{2},\ldots,i_{k}\text{ are}\\\text{distinct}%
\\\text{elements of }V}}\left(  X+Y\right)  ^{N-k}x_{i_{1}}x_{i_{2}}\cdots
x_{i_{k}}-\left(  X+Y\right)  ^{N}\nonumber\\
&  =\underbrace{\sum_{\substack{i_{1},i_{2},\ldots,i_{k}\text{ are}%
\\\text{distinct}\\\text{elements of }V;\\k\neq0}}}_{\substack{=\sum
_{\substack{i_{1},i_{2},\ldots,i_{k}\text{ are}\\\text{distinct}%
\\\text{elements of }V;\\k>0}}\\\text{(because for any }k\text{-tuple }\left(
i_{1},i_{2},\ldots,i_{k}\right)  \text{,}\\\text{the condition }\left(
k\neq0\right)  \text{ is equivalent to}\\\text{the condition }\left(
k>0\right)  \text{ (since }k\in\mathbb{N}\text{))}}}\left(  X+Y\right)
^{N-k}x_{i_{1}}x_{i_{2}}\cdots x_{i_{k}}\nonumber\\
&  =\sum_{\substack{i_{1},i_{2},\ldots,i_{k}\text{ are}\\\text{distinct}%
\\\text{elements of }V;\\k>0}}\left(  X+Y\right)  ^{N-k}x_{i_{1}}x_{i_{2}%
}\cdots x_{i_{k}}\nonumber\\
&  =\sum_{k>0}\sum_{\substack{i_{1},i_{2},\ldots,i_{k}\text{ are}%
\\\text{distinct}\\\text{elements of }V}}\left(  X+Y\right)  ^{N-k}x_{i_{1}%
}x_{i_{2}}\cdots x_{i_{k}}\nonumber\\
&  =\sum_{k\in\mathbb{N}}\sum_{\substack{i_{1},i_{2},\ldots,i_{k+1}\text{
are}\\\text{distinct}\\\text{elements of }V}}\left(  X+Y\right)  ^{N-\left(
k+1\right)  }\underbrace{x_{i_{1}}x_{i_{2}}\cdots x_{i_{k+1}}}%
_{\substack{=x_{i_{1}}\left(  x_{i_{2}}x_{i_{3}}\cdots x_{i_{k+1}}\right)
\\\text{(since }k\in\mathbb{N}\text{)}}}\nonumber\\
&  \ \ \ \ \ \ \ \ \ \ \ \ \ \ \ \ \ \ \ \ \left(  \text{here, we substituted
}k+1\text{ for }k\text{ in the outer sum}\right) \nonumber\\
&  =\sum_{k\in\mathbb{N}}\underbrace{\sum_{\substack{i_{1},i_{2}%
,\ldots,i_{k+1}\text{ are}\\\text{distinct}\\\text{elements of }V}}\left(
X+Y\right)  ^{N-\left(  k+1\right)  }x_{i_{1}}\left(  x_{i_{2}}x_{i_{3}}\cdots
x_{i_{k+1}}\right)  }_{\substack{=\sum_{\substack{t,i_{1},i_{2},\ldots
,i_{k}\text{ are}\\\text{distinct}\\\text{elements of }V}}\left(  X+Y\right)
^{N-\left(  k+1\right)  }x_{t}\left(  x_{i_{1}}x_{i_{2}}\cdots x_{i_{k}%
}\right)  \\\text{(here, we renamed the summation index }\left(  i_{1}%
,i_{2},\ldots,i_{k+1}\right)  \text{ as }\left(  t,i_{1},i_{2},\ldots
,i_{k}\right)  \text{)}}}\nonumber\\
&  =\sum_{k\in\mathbb{N}}\underbrace{\sum_{\substack{t,i_{1},i_{2}%
,\ldots,i_{k}\text{ are}\\\text{distinct}\\\text{elements of }V}}}%
_{=\sum_{\substack{t,i_{1},i_{2},\ldots,i_{k}\text{ are elements of
}V;\\\text{the elements }t,i_{1},i_{2},\ldots,i_{k}\text{ are distinct}}%
}}\left(  X+Y\right)  ^{N-\left(  k+1\right)  }x_{t}\left(  x_{i_{1}}x_{i_{2}%
}\cdots x_{i_{k}}\right) \nonumber\\
&  =\sum_{k\in\mathbb{N}}\sum_{\substack{t,i_{1},i_{2},\ldots,i_{k}\text{ are
elements of }V;\\\text{the elements }t,i_{1},i_{2},\ldots,i_{k}\text{ are
distinct}}}\left(  X+Y\right)  ^{N-\left(  k+1\right)  }x_{t}\left(  x_{i_{1}%
}x_{i_{2}}\cdots x_{i_{k}}\right)  . \label{pf.lem.12.step.u2}%
\end{align}

But let $k\in\mathbb{N}$ be arbitrary. If $\left(  t,i_{1},i_{2},\ldots
,i_{k}\right)  $ is any $\left(  k+1\right)  $-tuple of elements of $V$, then
we have the following chain of equivalences:%
\begin{align*}
&  \ \left(  \text{the elements }t,i_{1},i_{2},\ldots,i_{k}\text{ are
distinct}\right) \\
&  \Longleftrightarrow\ \left(  \text{the elements }i_{1},i_{2},\ldots
,i_{k}\text{ are distinct and differ from }t\right) \\
&  \Longleftrightarrow\ \left(  \text{the elements }i_{1},i_{2},\ldots
,i_{k}\text{ are distinct}\right)  \wedge\underbrace{\left(  \text{the
elements }i_{1},i_{2},\ldots,i_{k}\text{ differ from }t\right)  }%
_{\Longleftrightarrow\ \left(  \text{the elements }i_{1},i_{2},\ldots
,i_{k}\text{ belong to }V\setminus\left\{  t\right\}  \right)  }\\
&  \Longleftrightarrow\ \left(  \text{the elements }i_{1},i_{2},\ldots
,i_{k}\text{ are distinct}\right)  \wedge\left(  \text{the elements }%
i_{1},i_{2},\ldots,i_{k}\text{ belong to }V\setminus\left\{  t\right\}
\right) \\
&  \Longleftrightarrow\ \left(  \text{the elements }i_{1},i_{2},\ldots
,i_{k}\text{ are distinct elements of }V\setminus\left\{  t\right\}  \right)
.
\end{align*}
Hence, we have the following equality of summation signs:%
\begin{align}
\sum_{\substack{t,i_{1},i_{2},\ldots,i_{k}\text{ are elements of
}V;\\\text{the elements }t,i_{1},i_{2},\ldots,i_{k}\text{ are distinct}}}  &
=\sum_{\substack{t,i_{1},i_{2},\ldots,i_{k}\text{ are elements of
}V;\\\text{the elements }i_{1},i_{2},\ldots,i_{k}\text{ are}\\\text{distinct
elements of }V\setminus\left\{  t\right\}  }}\nonumber\\
&  =\sum_{t\in V}\underbrace{\sum_{\substack{i_{1},i_{2},\ldots,i_{k}\text{
are elements of }V;\\\text{the elements }i_{1},i_{2},\ldots,i_{k}\text{
are}\\\text{distinct elements of }V\setminus\left\{  t\right\}  }%
}}_{\substack{=\sum_{\substack{i_{1},i_{2},\ldots,i_{k}\text{ are distinct
elements of }V\setminus\left\{  t\right\}  ;\\\text{the elements }i_{1}%
,i_{2},\ldots,i_{k}\text{ are elements of }V}}\\=\sum_{\substack{i_{1}%
,i_{2},\ldots,i_{k}\text{ are}\\\text{distinct elements of }V\setminus\left\{
t\right\}  }}\\\text{(because if }i_{1},i_{2},\ldots,i_{k}\text{
are}\\\text{distinct elements of }V\setminus\left\{  t\right\}  \text{,
then}\\i_{1},i_{2},\ldots,i_{k}\text{ must automatically}\\\text{be elements
of }V\\\text{(since }i_{1},i_{2},\ldots,i_{k}\text{ are elements of
}V\setminus\left\{  t\right\}  \text{,}\\\text{but }V\setminus\left\{
t\right\}  \text{ is a subset of }V\text{))}}}\nonumber\\
&  =\sum_{t\in V}\sum_{\substack{i_{1},i_{2},\ldots,i_{k}\text{ are}%
\\\text{distinct elements of }V\setminus\left\{  t\right\}  }}.
\label{pf.lem.12.step.u4}%
\end{align}

Now, forget that we fixed $k$. We thus have proven the equality
(\ref{pf.lem.12.step.u4}) for each $k\in\mathbb{N}$. Now,
(\ref{pf.lem.12.step.u2}) becomes%
\begin{align*}
&  \sum_{\substack{i_{1},i_{2},\ldots,i_{k}\text{ are}\\\text{distinct}%
\\\text{elements of }V}}\left(  X+Y\right)  ^{N-k}x_{i_{1}}x_{i_{2}}\cdots
x_{i_{k}}-\left(  X+Y\right)  ^{N}\\
&  =\sum_{k\in\mathbb{N}}\underbrace{\sum_{\substack{t,i_{1},i_{2}%
,\ldots,i_{k}\text{ are elements of }V;\\\text{the elements }t,i_{1}%
,i_{2},\ldots,i_{k}\text{ are distinct}}}}_{\substack{=\sum_{t\in V}%
\sum_{\substack{i_{1},i_{2},\ldots,i_{k}\text{ are}\\\text{distinct elements
of }V\setminus\left\{  t\right\}  }}\\\text{(by (\ref{pf.lem.12.step.u4}))}%
}}\underbrace{\left(  X+Y\right)  ^{N-\left(  k+1\right)  }}%
_{\substack{=\left(  X+Y\right)  ^{N-1-k}\\\text{(since }N-\left(  k+1\right)
=N-1-k\text{)}}}x_{t}\left(  x_{i_{1}}x_{i_{2}}\cdots x_{i_{k}}\right) \\
&  =\underbrace{\sum_{k\in\mathbb{N}}\sum_{t\in V}}_{=\sum_{t\in V}\sum
_{k\in\mathbb{N}}}\sum_{\substack{i_{1},i_{2},\ldots,i_{k}\text{
are}\\\text{distinct elements of }V\setminus\left\{  t\right\}  }}\left(
X+Y\right)  ^{N-1-k}x_{t}\left(  x_{i_{1}}x_{i_{2}}\cdots x_{i_{k}}\right) \\
&  =\sum_{t\in V}\sum_{k\in\mathbb{N}}\sum_{\substack{i_{1},i_{2},\ldots
,i_{k}\text{ are}\\\text{distinct elements of }V\setminus\left\{  t\right\}
}}\left(  X+Y\right)  ^{N-1-k}x_{t}\left(  x_{i_{1}}x_{i_{2}}\cdots x_{i_{k}%
}\right) \\
&  =\sum_{t\in V}\sum_{k\in\mathbb{N}}\sum_{\substack{i_{1},i_{2},\ldots
,i_{k}\text{ are}\\\text{distinct}\\\text{elements of }V\setminus\left\{
t\right\}  }}\left(  X+Y\right)  ^{N-1-k}x_{t}\left(  x_{i_{1}}x_{i_{2}}\cdots
x_{i_{k}}\right)  .
\end{align*}
Comparing this with (\ref{pf.lem.12.step.u1}), we obtain%
\begin{align}
&  \sum_{\substack{S\subseteq V;\\S\neq\varnothing}}x\left(  S\right)  \left(
X+x\left(  S\right)  \right)  ^{\left\vert S\right\vert -1}\left(  Y-x\left(
S\right)  \right)  ^{N-\left\vert S\right\vert }\nonumber\\
&  =\sum_{\substack{i_{1},i_{2},\ldots,i_{k}\text{ are}\\\text{distinct}%
\\\text{elements of }V}}\left(  X+Y\right)  ^{N-k}x_{i_{1}}x_{i_{2}}\cdots
x_{i_{k}}-\left(  X+Y\right)  ^{N}. \label{pf.lem.12.step.v1}%
\end{align}
Adding this equality to (\ref{pf.lem.12.step.part1'}), we obtain%
\begin{align}
&  Y^{N}+\sum_{\substack{S\subseteq V;\\S\neq\varnothing}}X\left(  X+x\left(
S\right)  \right)  ^{\left\vert S\right\vert -1}\left(  Y-x\left(  S\right)
\right)  ^{N-\left\vert S\right\vert }\nonumber\\
&  \ \ \ \ \ \ \ \ \ \ +\sum_{\substack{S\subseteq V;\\S\neq\varnothing
}}x\left(  S\right)  \left(  X+x\left(  S\right)  \right)  ^{\left\vert
S\right\vert -1}\left(  Y-x\left(  S\right)  \right)  ^{N-\left\vert
S\right\vert }\nonumber\\
&  =\left(  X+Y\right)  ^{N}+\sum_{\substack{i_{1},i_{2},\ldots,i_{k}\text{
are}\\\text{distinct}\\\text{elements of }V}}\left(  X+Y\right)
^{N-k}x_{i_{1}}x_{i_{2}}\cdots x_{i_{k}}-\left(  X+Y\right)  ^{N}\nonumber\\
&  =\sum_{\substack{i_{1},i_{2},\ldots,i_{k}\text{ are}\\\text{distinct}%
\\\text{elements of }V}}\left(  X+Y\right)  ^{N-k}x_{i_{1}}x_{i_{2}}\cdots
x_{i_{k}}. \label{pf.lem.12.step.part2}%
\end{align}
Now,%
\begin{align*}
&  \sum_{S\subseteq V}\left(  X+x\left(  S\right)  \right)  ^{\left\vert
S\right\vert }\left(  Y-x\left(  S\right)  \right)  ^{N-\left\vert
S\right\vert }\\
&  =\underbrace{\left(  X+x\left(  \varnothing\right)  \right)  ^{\left\vert
\varnothing\right\vert }}_{\substack{=\left(  X+x\left(  \varnothing\right)
\right)  ^{0}\\\text{(since }\left\vert \varnothing\right\vert =0\text{)}%
}}\left(  Y-\underbrace{x\left(  \varnothing\right)  }%
_{\substack{=0\\\text{(by Lemma \ref{lem.x(S).triv} \textbf{(a)})}}}\right)
^{N-\left\vert \varnothing\right\vert }+\sum_{\substack{S\subseteq
V;\\S\neq\varnothing}}\left(  X+x\left(  S\right)  \right)  ^{\left\vert
S\right\vert }\left(  Y-x\left(  S\right)  \right)  ^{N-\left\vert
S\right\vert }\\
&  \ \ \ \ \ \ \ \ \ \ \ \ \ \ \ \ \ \ \ \ \left(
\begin{array}
[c]{c}%
\text{here, we have split off the addend for }S=\varnothing\text{ from the
sum}\\
\text{(since }\varnothing\text{ is a subset of }V\text{)}%
\end{array}
\right) \\
&  =\underbrace{\left(  X+x\left(  \varnothing\right)  \right)  ^{0}}%
_{=1}\underbrace{\left(  Y-0\right)  ^{N-\left\vert \varnothing\right\vert }%
}_{\substack{=Y^{N-0}\\\text{(since }Y-0=Y\text{ and }\left\vert
\varnothing\right\vert =0\text{)}}}+\sum_{\substack{S\subseteq V;\\S\neq
\varnothing}}\left(  X+x\left(  S\right)  \right)  ^{\left\vert S\right\vert
}\left(  Y-x\left(  S\right)  \right)  ^{N-\left\vert S\right\vert }\\
&  =\underbrace{Y^{N-0}}_{=Y^{N}}+\sum_{\substack{S\subseteq V;\\S\neq
\varnothing}}\ \ \underbrace{\left(  X+x\left(  S\right)  \right)
^{\left\vert S\right\vert }}_{\substack{=\left(  X+x\left(  S\right)  \right)
\left(  X+x\left(  S\right)  \right)  ^{\left\vert S\right\vert -1}%
\\\text{(since }\left\vert S\right\vert -1\geq0\\\text{(by
(\ref{pf.lem.12.step.S-1>=0})))}}}\left(  Y-x\left(  S\right)  \right)
^{N-\left\vert S\right\vert }\\
&  =Y^{N}+\sum_{\substack{S\subseteq V;\\S\neq\varnothing}%
}\ \ \underbrace{\left(  X+x\left(  S\right)  \right)  \left(  X+x\left(
S\right)  \right)  ^{\left\vert S\right\vert -1}\left(  Y-x\left(  S\right)
\right)  ^{N-\left\vert S\right\vert }}_{=X\left(  X+x\left(  S\right)
\right)  ^{\left\vert S\right\vert -1}\left(  Y-x\left(  S\right)  \right)
^{N-\left\vert S\right\vert }+x\left(  S\right)  \left(  X+x\left(  S\right)
\right)  ^{\left\vert S\right\vert -1}\left(  Y-x\left(  S\right)  \right)
^{N-\left\vert S\right\vert }}\\
&  =Y^{N}+\underbrace{\sum_{\substack{S\subseteq V;\\S\neq\varnothing}}\left(
X\left(  X+x\left(  S\right)  \right)  ^{\left\vert S\right\vert -1}\left(
Y-x\left(  S\right)  \right)  ^{N-\left\vert S\right\vert }+x\left(  S\right)
\left(  X+x\left(  S\right)  \right)  ^{\left\vert S\right\vert -1}\left(
Y-x\left(  S\right)  \right)  ^{N-\left\vert S\right\vert }\right)  }%
_{=\sum_{\substack{S\subseteq V;\\S\neq\varnothing}}X\left(  X+x\left(
S\right)  \right)  ^{\left\vert S\right\vert -1}\left(  Y-x\left(  S\right)
\right)  ^{N-\left\vert S\right\vert }+\sum_{\substack{S\subseteq
V;\\S\neq\varnothing}}x\left(  S\right)  \left(  X+x\left(  S\right)  \right)
^{\left\vert S\right\vert -1}\left(  Y-x\left(  S\right)  \right)
^{N-\left\vert S\right\vert }}\\
&  =Y^{N}+\sum_{\substack{S\subseteq V;\\S\neq\varnothing}}X\left(  X+x\left(
S\right)  \right)  ^{\left\vert S\right\vert -1}\left(  Y-x\left(  S\right)
\right)  ^{N-\left\vert S\right\vert }\\
&  \ \ \ \ \ \ \ \ \ \ +\sum_{\substack{S\subseteq V;\\S\neq\varnothing
}}x\left(  S\right)  \left(  X+x\left(  S\right)  \right)  ^{\left\vert
S\right\vert -1}\left(  Y-x\left(  S\right)  \right)  ^{N-\left\vert
S\right\vert }\\
&  =\sum_{\substack{i_{1},i_{2},\ldots,i_{k}\text{ are}\\\text{distinct}%
\\\text{elements of }V}}\left(  X+Y\right)  ^{N-k}x_{i_{1}}x_{i_{2}}\cdots
x_{i_{k}}\ \ \ \ \ \ \ \ \ \ \left(  \text{by (\ref{pf.lem.12.step.part2}%
)}\right)  .
\end{align*}
Since $n=N$, we can rewrite this equality as follows:%
\begin{equation}
\sum_{S\subseteq V}\left(  X+x\left(  S\right)  \right)  ^{\left\vert
S\right\vert }\left(  Y-x\left(  S\right)  \right)  ^{n-\left\vert
S\right\vert }=\sum_{\substack{i_{1},i_{2},\ldots,i_{k}\text{ are}%
\\\text{distinct}\\\text{elements of }V}}\left(  X+Y\right)  ^{n-k}x_{i_{1}%
}x_{i_{2}}\cdots x_{i_{k}}.\nonumber
\end{equation}
In other words, the equality (\ref{eq.lem12.1}) holds.

We have thus shown that the equalities (\ref{eq.lem12.1}) and
(\ref{eq.lem12.2}) hold.

Now, forget that we fixed $V$, $n$, $x_{s}$, $X$ and $Y$ and assumed that
$n=N$. We thus have proven that if $V$, $n$, $x_{s}$, $X$ and $Y$ are as in
Lemma \ref{lem.12}, and if we have $n=N$, then the equalities
(\ref{eq.lem12.1}) and (\ref{eq.lem12.2}) hold. In other words, Lemma
\ref{lem.12} holds when $n=N$. This completes the induction step. Thus, Lemma
\ref{lem.12} is proven by induction.
\end{proof}

We can now prove Theorem \ref{thm.1} and Theorem \ref{thm.2} in their original forms:

\begin{proof}
[Proof of Theorem \ref{thm.1}.]From (\ref{eq.lem12.1}), we obtain%
\begin{align*}
&  \sum_{\substack{i_{1},i_{2},\ldots,i_{k}\text{ are}\\\text{distinct}%
\\\text{elements of }V}}\left(  X+Y\right)  ^{n-k}x_{i_{1}}x_{i_{2}}\cdots
x_{i_{k}}\\
&  =\sum_{S\subseteq V}\left(  X+\underbrace{x\left(  S\right)  }%
_{\substack{=\sum_{s\in S}x_{s}\\\text{(by the definition of }x\left(
S\right)  \text{)}}}\right)  ^{\left\vert S\right\vert }\left(
Y-\underbrace{x\left(  S\right)  }_{\substack{=\sum_{s\in S}x_{s}\\\text{(by
the definition of }x\left(  S\right)  \text{)}}}\right)  ^{n-\left\vert
S\right\vert }\\
&  =\sum_{S\subseteq V}\left(  X+\sum_{s\in S}x_{s}\right)  ^{\left\vert
S\right\vert }\left(  Y-\sum_{s\in S}x_{s}\right)  ^{n-\left\vert S\right\vert
}.
\end{align*}
In other words,%
\[
\sum_{S\subseteq V}\left(  X+\sum_{s\in S}x_{s}\right)  ^{\left\vert
S\right\vert }\left(  Y-\sum_{s\in S}x_{s}\right)  ^{n-\left\vert S\right\vert
}=\sum_{\substack{i_{1},i_{2},\ldots,i_{k}\text{ are}\\\text{distinct}%
\\\text{elements of }V}}\left(  X+Y\right)  ^{n-k}x_{i_{1}}x_{i_{2}}\cdots
x_{i_{k}}.
\]
This proves Theorem \ref{thm.1}.
\end{proof}

\begin{proof}
[Proof of Theorem \ref{thm.2}.]From (\ref{eq.lem12.2}), we obtain%
\begin{align*}
&  \left(  X+Y\right)  ^{n}\\
&  =Y^{n}+\sum_{\substack{S\subseteq V;\\S\neq\varnothing}}X\left(
X+\underbrace{x\left(  S\right)  }_{\substack{=\sum_{s\in S}x_{s}\\\text{(by
the definition of }x\left(  S\right)  \text{)}}}\right)  ^{\left\vert
S\right\vert -1}\left(  Y-\underbrace{x\left(  S\right)  }_{\substack{=\sum
_{s\in S}x_{s}\\\text{(by the definition of }x\left(  S\right)  \text{)}%
}}\right)  ^{n-\left\vert S\right\vert }\\
&  =Y^{n}+\sum_{\substack{S\subseteq V;\\S\neq\varnothing}}X\left(
X+\sum_{s\in S}x_{s}\right)  ^{\left\vert S\right\vert -1}\left(  Y-\sum_{s\in
S}x_{s}\right)  ^{n-\left\vert S\right\vert }.
\end{align*}
Comparing this with%
\begin{align*}
&  \sum_{S\subseteq V}X\left(  X+\sum_{s\in S}x_{s}\right)  ^{\left\vert
S\right\vert -1}\left(  Y-\sum_{s\in S}x_{s}\right)  ^{n-\left\vert
S\right\vert }\\
&  =\underbrace{X\left(  X+\sum_{s\in\varnothing}x_{s}\right)  ^{\left\vert
\varnothing\right\vert -1}}_{\substack{=1\\\text{(according to our convention
for}\\\text{interpreting }X\left(  X+\sum_{s\in S}x_{s}\right)  ^{\left\vert
S\right\vert -1}\\\text{when }S=\varnothing\text{)}}}\left(
Y-\underbrace{\sum_{s\in\varnothing}x_{s}}_{=\left(  \text{empty sum}\right)
=0}\right)  ^{n-\left\vert \varnothing\right\vert }\\
&  \ \ \ \ \ \ \ \ \ \ +\sum_{\substack{S\subseteq V;\\S\neq\varnothing
}}X\left(  X+\sum_{s\in S}x_{s}\right)  ^{\left\vert S\right\vert -1}\left(
Y-\sum_{s\in S}x_{s}\right)  ^{n-\left\vert S\right\vert }\\
&  \ \ \ \ \ \ \ \ \ \ \ \ \ \ \ \ \ \ \ \ \left(
\begin{array}
[c]{c}%
\text{here, we have split off the addend for }S=\varnothing\text{ from the
sum}\\
\text{(since }\varnothing\text{ is a subset of }V\text{)}%
\end{array}
\right) \\
&  =\underbrace{\left(  Y-0\right)  ^{n-\left\vert \varnothing\right\vert }%
}_{\substack{=Y^{n-0}\\\text{(since }Y-0=Y\\\text{and }\left\vert
\varnothing\right\vert =0\text{)}}}+\sum_{\substack{S\subseteq V;\\S\neq
\varnothing}}X\left(  X+\sum_{s\in S}x_{s}\right)  ^{\left\vert S\right\vert
-1}\left(  Y-\sum_{s\in S}x_{s}\right)  ^{n-\left\vert S\right\vert }\\
&  =\underbrace{Y^{n-0}}_{=Y^{n}}+\sum_{\substack{S\subseteq V;\\S\neq
\varnothing}}X\left(  X+\sum_{s\in S}x_{s}\right)  ^{\left\vert S\right\vert
-1}\left(  Y-\sum_{s\in S}x_{s}\right)  ^{n-\left\vert S\right\vert }\\
&  =Y^{n}+\sum_{\substack{S\subseteq V;\\S\neq\varnothing}}X\left(
X+\sum_{s\in S}x_{s}\right)  ^{\left\vert S\right\vert -1}\left(  Y-\sum_{s\in
S}x_{s}\right)  ^{n-\left\vert S\right\vert },
\end{align*}
we obtain%
\[
\sum_{S\subseteq V}X\left(  X+\sum_{s\in S}x_{s}\right)  ^{\left\vert
S\right\vert -1}\left(  Y-\sum_{s\in S}x_{s}\right)  ^{n-\left\vert
S\right\vert }=\left(  X+Y\right)  ^{n}.
\]
This proves Theorem \ref{thm.2}.
\end{proof}
\end{verlong}

\subsection{Proofs of Theorems \ref{thm.4} and \ref{thm.5}}

\begin{vershort}
\begin{proof}
[Proof of Theorem \ref{thm.4}.]We have\footnote{The reader is invited to check
that the below manipulations work even for the $S=V$ addend (despite, or
rather because of, our special rule for interpreting $\left(  Y-\sum_{s\in
S}x_{s}\right)  ^{n-\left\vert S\right\vert -1}\left(  Y-\sum_{s\in V}%
x_{s}\right)  $ in this case).}
\begin{align*}
&  \sum_{S\subseteq V}X\left(  X+\sum_{s\in S}x_{s}\right)  ^{\left\vert
S\right\vert -1}\left(  Y-\sum_{s\in S}x_{s}\right)  ^{n-\left\vert
S\right\vert -1}\underbrace{\left(  Y-\sum_{s\in V}x_{s}\right)  }_{=\left(
Y-\sum_{s\in S}x_{s}\right)  -\sum_{t\in V\setminus S}x_{t}}\\
&  =\sum_{S\subseteq V}X\left(  X+\sum_{s\in S}x_{s}\right)  ^{\left\vert
S\right\vert -1}\left(  Y-\sum_{s\in S}x_{s}\right)  ^{n-\left\vert
S\right\vert -1}\left(  \left(  Y-\sum_{s\in S}x_{s}\right)  -\sum_{t\in
V\setminus S}x_{t}\right) \\
&  =\sum_{S\subseteq V}X\left(  X+\sum_{s\in S}x_{s}\right)  ^{\left\vert
S\right\vert -1}\underbrace{\left(  Y-\sum_{s\in S}x_{s}\right)
^{n-\left\vert S\right\vert -1}\left(  Y-\sum_{s\in S}x_{s}\right)
}_{=\left(  Y-\sum_{s\in S}x_{s}\right)  ^{n-\left\vert S\right\vert }}\\
&  \ \ \ \ \ \ \ \ \ \ -\underbrace{\sum_{S\subseteq V}X\left(  X+\sum_{s\in
S}x_{s}\right)  ^{\left\vert S\right\vert -1}\left(  Y-\sum_{s\in S}%
x_{s}\right)  ^{n-\left\vert S\right\vert -1}\sum_{t\in V\setminus S}x_{t}%
}_{=\sum_{S\subseteq V}\sum_{t\in V\setminus S}X\left(  X+\sum_{s\in S}%
x_{s}\right)  ^{\left\vert S\right\vert -1}\left(  Y-\sum_{s\in S}%
x_{s}\right)  ^{n-\left\vert S\right\vert -1}x_{t}}\\
&  =\underbrace{\sum_{S\subseteq V}X\left(  X+\sum_{s\in S}x_{s}\right)
^{\left\vert S\right\vert -1}\left(  Y-\sum_{s\in S}x_{s}\right)
^{n-\left\vert S\right\vert }}_{\substack{=\left(  X+Y\right)  ^{n}\\\text{(by
Theorem \ref{thm.2})}}}\\
&  \ \ \ \ \ \ \ \ \ \ -\underbrace{\sum_{S\subseteq V}\sum_{t\in V\setminus
S}}_{=\sum_{t\in V}\sum_{S\subseteq V\setminus\left\{  t\right\}  }}X\left(
X+\sum_{s\in S}x_{s}\right)  ^{\left\vert S\right\vert -1}\underbrace{\left(
Y-\sum_{s\in S}x_{s}\right)  ^{n-\left\vert S\right\vert -1}}_{=\left(
Y-\sum_{s\in S}x_{s}\right)  ^{n-1-\left\vert S\right\vert }}x_{t}%
\end{align*}%
\begin{align*}
&  =\left(  X+Y\right)  ^{n}-\sum_{t\in V}\underbrace{\sum_{S\subseteq
V\setminus\left\{  t\right\}  }X\left(  X+\sum_{s\in S}x_{s}\right)
^{\left\vert S\right\vert -1}\left(  Y-\sum_{s\in S}x_{s}\right)
^{n-1-\left\vert S\right\vert }}_{\substack{=\left(  X+Y\right)
^{n-1}\\\text{(by Theorem \ref{thm.2}, applied to }V\setminus\left\{
t\right\}  \text{ and }n-1\text{ instead of }V\text{ and }n\text{)}}}x_{t}\\
&  =\left(  X+Y\right)  ^{n}-\underbrace{\sum_{t\in V}\left(  X+Y\right)
^{n-1}x_{t}}_{=\left(  X+Y\right)  ^{n-1}\sum_{t\in V}x_{t}}=\left(
X+Y\right)  ^{n}-\underbrace{\left(  X+Y\right)  ^{n-1}\sum_{t\in V}x_{t}%
}_{\substack{=\left(  \sum_{t\in V}x_{t}\right)  \left(  X+Y\right)
^{n-1}\\\text{(since }X+Y\text{ lies in the center of }\mathbb{L}\text{)}}}\\
&  =\left(  X+Y\right)  ^{n}-\underbrace{\left(  \sum_{t\in V}x_{t}\right)
}_{=\sum_{s\in V}x_{s}}\left(  X+Y\right)  ^{n-1}=\left(  X+Y\right)
^{n}-\left(  \sum_{s\in V}x_{s}\right)  \left(  X+Y\right)  ^{n-1}\\
&  =\left(  X+Y-\sum_{s\in V}x_{s}\right)  \left(  X+Y\right)  ^{n-1}.
\end{align*}
This proves Theorem \ref{thm.4}.
\end{proof}
\end{vershort}

\begin{verlong}
Now, we are going to prepare for the proof of Theorem \ref{thm.4}. Let us
first restate this theorem (or, more precisely, its case when $V$ is nonempty)
in a more convenient form:

\begin{lemma}
\label{lem.4}Let $V$ be a finite nonempty set. Let $n=\left\vert V\right\vert
$. For each $s\in V$, let $x_{s}$ be an element of $\mathbb{L}$. Let $X$ and
$Y$ be two elements of $\mathbb{L}$ such that $X+Y$ lies in the center of
$\mathbb{L}$. Then,%
\begin{align*}
&  Y^{n-1}\left(  Y-x\left(  V\right)  \right)  +X\left(  X+x\left(  V\right)
\right)  ^{n-1}\\
&  \ \ \ \ \ \ \ \ \ \ +\sum_{\substack{S\subseteq V;\\S\neq\varnothing
;\ S\neq V}}X\left(  X+x\left(  S\right)  \right)  ^{\left\vert S\right\vert
-1}\left(  Y-x\left(  S\right)  \right)  ^{n-\left\vert S\right\vert
-1}\left(  Y-x\left(  V\right)  \right) \\
&  =\left(  X+Y-x\left(  V\right)  \right)  \left(  X+Y\right)  ^{n-1}.
\end{align*}
Here, we are again using the notation from Definition \ref{def.x(S)}.
\end{lemma}

Notice that the claim of Lemma \ref{lem.4} (unlike the claim of Theorem
\ref{thm.4}) does not require any convention about how to interpret the term
$X\left(  X+\sum_{s\in S}x_{s}\right)  ^{\left\vert S\right\vert -1}$ when
$S=\varnothing$ (and any other such conventions), because all expressions
appearing in Lemma \ref{lem.4} are well-defined a-priori.

\begin{proof}
[Proof of Lemma \ref{lem.4}.]Let us first notice that all terms appearing in
Lemma \ref{lem.4} are well-defined\footnote{\textit{Proof.} This follows from
the following four observations:
\par
\begin{itemize}
\item The terms $Y^{n-1}$, $X\left(  X+x\left(  V\right)  \right)  ^{n-1}$ and
$\left(  X+Y\right)  ^{n-1}$ are well-defined. (\textit{Proof:} We know that
the set $V$ is nonempty. Hence, $\left\vert V\right\vert >0$. Thus,
$\left\vert V\right\vert \geq1$ (since $\left\vert V\right\vert $ is an
integer). Therefore, $n=\left\vert V\right\vert \geq1$, so that $n-1\geq0$ and
thus $n-1\in\mathbb{N}$. Hence, the terms $Y^{n-1}$, $X\left(  X+x\left(
V\right)  \right)  ^{n-1}$ and $\left(  X+Y\right)  ^{n-1}$ are well-defined.
Qed.)
\par
\item For every subset $S$ of $V$, the term $\left\vert S\right\vert $ is
well-defined. (\textit{Proof:} Let $S$ be a subset of $V$. Then, $S$ is a
subset of the finite set $V$, and therefore itself is finite. Hence, the term
$\left\vert S\right\vert $ is well-defined. Qed.)
\par
\item For every subset $S$ of $V$ satisfying $S\neq\varnothing$ and $S\neq V$,
the term $\left(  X+x\left(  S\right)  \right)  ^{\left\vert S\right\vert -1}$
is well-defined. (\textit{Proof:} Let $S$ be a subset of $V$ satisfying
$S\neq\varnothing$ and $S\neq V$. Then, $\left\vert S\right\vert >0$ (since
$S\neq\varnothing$) and thus $\left\vert S\right\vert \geq1$ (since
$\left\vert S\right\vert $ is an integer). In other words, $\left\vert
S\right\vert -1\in\mathbb{N}$. Hence, the term $\left(  X+x\left(  S\right)
\right)  ^{\left\vert S\right\vert -1}$ is well-defined. Qed.)
\par
\item For every subset $S$ of $V$ satisfying $S\neq\varnothing$ and $S\neq V$,
the term $\left(  Y-x\left(  S\right)  \right)  ^{n-\left\vert S\right\vert
-1}$ is well-defined. (\textit{Proof:} Let $S$ be a subset of $V$ satisfying
$S\neq\varnothing$ and $S\neq V$. Then, $S$ is a proper subset of $V$ (since
$S$ is a subset of $V$ satisfying $S\neq V$). It is well-known that if $B$ is
a finite set, and if $A$ is a proper subset of $B$, then $\left\vert
A\right\vert <\left\vert B\right\vert $. Applying this to $A=S$ and $B=V$, we
obtain $\left\vert S\right\vert <\left\vert V\right\vert =n$, so that
$n-\left\vert S\right\vert >0$. Since $n-\left\vert S\right\vert $ is an
integer, this entails that $n-\left\vert S\right\vert \geq1$. Hence,
$n-\left\vert S\right\vert -1\geq0$. In other words, $n-\left\vert
S\right\vert -1\in\mathbb{N}$. Hence, the term $\left(  Y-x\left(  S\right)
\right)  ^{n-\left\vert S\right\vert -1}$ is well-defined. Qed.)
\end{itemize}
}. Also, note that $n=\left\vert V\right\vert >0$ (since $V$ is nonempty), so
that $n\geq1$ and thus $n-1\in\mathbb{N}$. From $n=\left\vert V\right\vert $,
we also obtain $n-\left\vert V\right\vert =0$.

The definition of $x\left(  V\right)  $ yields%
\begin{equation}
x\left(  V\right)  =\sum_{s\in V}x_{s}=\sum_{t\in V}x_{t}
\label{pf.lem.4.x(V)}%
\end{equation}
(here, we have renamed the summation index $s$ as $t$).

For every subset $S$ of $V$, we have
\begin{equation}
Y-x\left(  V\right)  =\left(  Y-x\left(  S\right)  \right)  -\sum
_{\substack{t\in V;\\t\notin S}}x_{t} \label{pf.lem.4.1}%
\end{equation}
\footnote{\textit{Proof of (\ref{pf.lem.4.1}):} Let $S$ be a subset of $V$.
The definition of $x\left(  S\right)  $ yields $x\left(  S\right)  =\sum_{s\in
S}x_{s}=\sum_{t\in S}x_{t}$ (here, we have renamed the summation index $s$ as
$t$). But (\ref{pf.lem.4.x(V)}) yields%
\begin{align*}
x\left(  V\right)   &  =\sum_{t\in V}x_{t}=\underbrace{\sum_{\substack{t\in
V;\\t\in S}}}_{\substack{=\sum_{t\in S}\\\text{(since }S\text{ is a subset of
}V\text{)}}}x_{t}+\sum_{\substack{t\in V;\\t\notin S}}x_{t}\\
&  \ \ \ \ \ \ \ \ \ \ \ \ \ \ \ \ \ \ \ \ \left(  \text{since each }t\in
V\text{ satisfies either }t\in S\text{ or }t\notin S\text{ (but not
both)}\right) \\
&  =\underbrace{\sum_{t\in S}x_{t}}_{=x\left(  S\right)  }+\sum
_{\substack{t\in V;\\t\notin S}}x_{t}=x\left(  S\right)  +\sum_{\substack{t\in
V;\\t\notin S}}x_{t}.
\end{align*}
Thus,%
\[
Y-\underbrace{x\left(  V\right)  }_{=x\left(  S\right)  +\sum_{\substack{t\in
V;\\t\notin S}}x_{t}}=Y-\left(  x\left(  S\right)  +\sum_{\substack{t\in
V;\\t\notin S}}x_{t}\right)  =\left(  Y-x\left(  S\right)  \right)
-\sum_{\substack{t\in V;\\t\notin S}}x_{t}.
\]
This proves (\ref{pf.lem.4.1}).}.

Fix any $t\in V$. Every subset $S$ of $V$ satisfying $t\notin S$ must
automatically satisfy $S\neq V$\ \ \ \ \footnote{\textit{Proof.} Let $S$ be a
subset of $V$ satisfying $t\notin S$. If we had $S=V$, then we would have
$t\in V=S$, which would contradict $t\notin S$. Hence, we cannot have $S=V$.
Thus, we have $S\neq V$. Qed.}. Hence, for any subset $S$ of $V$, we have the
following logical equivalence:%
\[
\left(  t\notin S\text{ and }S\neq\varnothing\text{ and }S\neq V\right)
\ \Longleftrightarrow\ \left(  t\notin S\text{ and }S\neq\varnothing\right)
.
\]
Thus, we have the following equality of summation signs:%
\begin{equation}
\sum_{\substack{S\subseteq V;\\t\notin S;\\S\neq\varnothing;\ S\neq V}%
}=\sum_{\substack{S\subseteq V;\\t\notin S;\\S\neq\varnothing}}=\sum
_{\substack{S\subseteq V\setminus\left\{  t\right\}  ;\\S\neq\varnothing}}
\label{pf.lem.4.sumeq1}%
\end{equation}
(since the subsets $S$ of $V$ satisfying $t\notin S$ are precisely the subsets
of $V\setminus\left\{  t\right\}  $).

The set $V\setminus\left\{  t\right\}  $ is a subset of the finite set $V$,
and thus is itself finite. Moreover, from $t\in V$, we obtain $\left\vert
V\setminus\left\{  t\right\}  \right\vert =\underbrace{\left\vert V\right\vert
}_{=n}-\,1=n-1$. Hence, we can apply the equality (\ref{eq.lem12.2}) to
$V\setminus\left\{  t\right\}  $ and $n-1$ instead of $V$ and $n$. We thus
obtain%
\[
Y^{n-1}+\sum_{\substack{S\subseteq V\setminus\left\{  t\right\}
;\\S\neq\varnothing}}X\left(  X+x\left(  S\right)  \right)  ^{\left\vert
S\right\vert -1}\left(  Y-x\left(  S\right)  \right)  ^{n-1-\left\vert
S\right\vert }=\left(  X+Y\right)  ^{n-1}.
\]
Subtracting $Y^{n-1}$ from both sides of this equality, we obtain%
\begin{equation}
\sum_{\substack{S\subseteq V\setminus\left\{  t\right\}  ;\\S\neq\varnothing
}}X\left(  X+x\left(  S\right)  \right)  ^{\left\vert S\right\vert -1}\left(
Y-x\left(  S\right)  \right)  ^{n-1-\left\vert S\right\vert }=\left(
X+Y\right)  ^{n-1}-Y^{n-1}. \label{pf.lem.4.t-eq0}%
\end{equation}

Now,%
\begin{align}
&  \underbrace{\sum_{\substack{S\subseteq V;\\S\neq\varnothing;\ S\neq
V;\\t\notin S}}}_{\substack{=\sum_{\substack{S\subseteq V;\\t\notin
S;\\S\neq\varnothing;\ S\neq V}}=\sum_{\substack{S\subseteq V\setminus\left\{
t\right\}  ;\\S\neq\varnothing}}\\\text{(by (\ref{pf.lem.4.sumeq1}))}%
}}X\left(  X+x\left(  S\right)  \right)  ^{\left\vert S\right\vert
-1}\underbrace{\left(  Y-x\left(  S\right)  \right)  ^{n-\left\vert
S\right\vert -1}}_{\substack{=\left(  Y-x\left(  S\right)  \right)
^{n-1-\left\vert S\right\vert }\\\text{(since }n-\left\vert S\right\vert
-1=n-1-\left\vert S\right\vert \text{)}}}x_{t}\nonumber\\
&  =\sum_{\substack{S\subseteq V\setminus\left\{  t\right\}  ;\\S\neq
\varnothing}}X\left(  X+x\left(  S\right)  \right)  ^{\left\vert S\right\vert
-1}\left(  Y-x\left(  S\right)  \right)  ^{n-1-\left\vert S\right\vert }%
x_{t}\nonumber\\
&  =\underbrace{\left(  \sum_{\substack{S\subseteq V\setminus\left\{
t\right\}  ;\\S\neq\varnothing}}X\left(  X+x\left(  S\right)  \right)
^{\left\vert S\right\vert -1}\left(  Y-x\left(  S\right)  \right)
^{n-1-\left\vert S\right\vert }\right)  }_{\substack{=\left(  X+Y\right)
^{n-1}-Y^{n-1}\\\text{(by (\ref{pf.lem.4.t-eq0}))}}}x_{t}\nonumber\\
&  =\left(  \left(  X+Y\right)  ^{n-1}-Y^{n-1}\right)  x_{t}.
\label{pf.lem.4.t-eq1}%
\end{align}

Now, let us forget that we fixed $t$. We thus have proven the equality
(\ref{pf.lem.4.t-eq1}) for each $t\in V$.

Subtracting $Y^{n}$ from both sides of the equality (\ref{eq.lem12.2}), we
obtain%
\begin{equation}
\sum_{\substack{S\subseteq V;\\S\neq\varnothing}}X\left(  X+x\left(  S\right)
\right)  ^{\left\vert S\right\vert -1}\left(  Y-x\left(  S\right)  \right)
^{n-\left\vert S\right\vert }=\left(  X+Y\right)  ^{n}-Y^{n}.
\label{pf.lem.4.4}%
\end{equation}
But the set $V$ is nonempty. In other words, $V\neq\varnothing$. Hence, $V$ is
a subset of $V$ satisfying $V\neq\varnothing$. In other words, $V$ is a subset
$S$ of $V$ satisfying $S\neq\varnothing$. Hence, the sum $\sum
_{\substack{S\subseteq V;\\S\neq\varnothing}}X\left(  X+x\left(  S\right)
\right)  ^{\left\vert S\right\vert -1}\left(  Y-x\left(  S\right)  \right)
^{n-\left\vert S\right\vert }$ has an addend for $S=V$. If we split off this
addend from this sum, then we obtain%
\begin{align*}
&  \sum_{\substack{S\subseteq V;\\S\neq\varnothing}}X\left(  X+x\left(
S\right)  \right)  ^{\left\vert S\right\vert -1}\left(  Y-x\left(  S\right)
\right)  ^{n-\left\vert S\right\vert }\\
&  =X\underbrace{\left(  X+x\left(  V\right)  \right)  ^{\left\vert
V\right\vert -1}}_{\substack{=\left(  X+x\left(  V\right)  \right)
^{n-1}\\\text{(since }\left\vert V\right\vert =n\text{)}}}\underbrace{\left(
Y-x\left(  V\right)  \right)  ^{n-\left\vert V\right\vert }}%
_{\substack{=\left(  Y-x\left(  V\right)  \right)  ^{0}\\\text{(since
}n-\left\vert V\right\vert =0\text{)}}}\\
&  \ \ \ \ \ \ \ \ \ \ +\sum_{\substack{S\subseteq V;\\S\neq\varnothing
;\ S\neq V}}X\left(  X+x\left(  S\right)  \right)  ^{\left\vert S\right\vert
-1}\left(  Y-x\left(  S\right)  \right)  ^{n-\left\vert S\right\vert }\\
&  =X\left(  X+x\left(  V\right)  \right)  ^{n-1}\underbrace{\left(
Y-x\left(  V\right)  \right)  ^{0}}_{=1}+\sum_{\substack{S\subseteq
V;\\S\neq\varnothing;\ S\neq V}}X\left(  X+x\left(  S\right)  \right)
^{\left\vert S\right\vert -1}\left(  Y-x\left(  S\right)  \right)
^{n-\left\vert S\right\vert }\\
&  =X\left(  X+x\left(  V\right)  \right)  ^{n-1}+\sum_{\substack{S\subseteq
V;\\S\neq\varnothing;\ S\neq V}}X\left(  X+x\left(  S\right)  \right)
^{\left\vert S\right\vert -1}\left(  Y-x\left(  S\right)  \right)
^{n-\left\vert S\right\vert }.
\end{align*}
Subtracting $X\left(  X+x\left(  V\right)  \right)  ^{n-1}$ from this
equality, we obtain%
\begin{align*}
&  \sum_{\substack{S\subseteq V;\\S\neq\varnothing}}X\left(  X+x\left(
S\right)  \right)  ^{\left\vert S\right\vert -1}\left(  Y-x\left(  S\right)
\right)  ^{n-\left\vert S\right\vert }-X\left(  X+x\left(  V\right)  \right)
^{n-1}\\
&  =\sum_{\substack{S\subseteq V;\\S\neq\varnothing;\ S\neq V}}X\left(
X+x\left(  S\right)  \right)  ^{\left\vert S\right\vert -1}\left(  Y-x\left(
S\right)  \right)  ^{n-\left\vert S\right\vert }.
\end{align*}
Hence,%
\begin{align}
&  \sum_{\substack{S\subseteq V;\\S\neq\varnothing;\ S\neq V}}X\left(
X+x\left(  S\right)  \right)  ^{\left\vert S\right\vert -1}\left(  Y-x\left(
S\right)  \right)  ^{n-\left\vert S\right\vert }\nonumber\\
&  =\underbrace{\sum_{\substack{S\subseteq V;\\S\neq\varnothing}}X\left(
X+x\left(  S\right)  \right)  ^{\left\vert S\right\vert -1}\left(  Y-x\left(
S\right)  \right)  ^{n-\left\vert S\right\vert }}_{\substack{=\left(
X+Y\right)  ^{n}-Y^{n}\\\text{(by (\ref{pf.lem.4.4}))}}}-\,X\left(  X+x\left(
V\right)  \right)  ^{n-1}\nonumber\\
&  =\left(  X+Y\right)  ^{n}-Y^{n}-X\left(  X+x\left(  V\right)  \right)
^{n-1}. \label{pf.lem.4.6}%
\end{align}

Now,%
\begin{align*}
&  \sum_{\substack{S\subseteq V;\\S\neq\varnothing;\ S\neq V}}X\left(
X+x\left(  S\right)  \right)  ^{\left\vert S\right\vert -1}\left(  Y-x\left(
S\right)  \right)  ^{n-\left\vert S\right\vert -1}\underbrace{\left(
Y-x\left(  V\right)  \right)  }_{\substack{=\left(  Y-x\left(  S\right)
\right)  -\sum_{\substack{t\in V;\\t\notin S}}x_{t}\\\text{(by
(\ref{pf.lem.4.1}))}}}\\
&  =\sum_{\substack{S\subseteq V;\\S\neq\varnothing;\ S\neq V}%
}\underbrace{X\left(  X+x\left(  S\right)  \right)  ^{\left\vert S\right\vert
-1}\left(  Y-x\left(  S\right)  \right)  ^{n-\left\vert S\right\vert
-1}\left(  \left(  Y-x\left(  S\right)  \right)  -\sum_{\substack{t\in
V;\\t\notin S}}x_{t}\right)  }_{\substack{=X\left(  X+x\left(  S\right)
\right)  ^{\left\vert S\right\vert -1}\left(  Y-x\left(  S\right)  \right)
^{n-\left\vert S\right\vert -1}\left(  Y-x\left(  S\right)  \right)
\\-X\left(  X+x\left(  S\right)  \right)  ^{\left\vert S\right\vert -1}\left(
Y-x\left(  S\right)  \right)  ^{n-\left\vert S\right\vert -1}\sum
_{\substack{t\in V;\\t\notin S}}x_{t}}}\\
&  =\sum_{\substack{S\subseteq V;\\S\neq\varnothing;\ S\neq V}}\left(
X\left(  X+x\left(  S\right)  \right)  ^{\left\vert S\right\vert -1}\left(
Y-x\left(  S\right)  \right)  ^{n-\left\vert S\right\vert -1}\left(
Y-x\left(  S\right)  \right)  \right. \\
&  \ \ \ \ \ \ \ \ \ \ \ \ \ \ \ \ \ \ \ \ \left.  -X\left(  X+x\left(
S\right)  \right)  ^{\left\vert S\right\vert -1}\left(  Y-x\left(  S\right)
\right)  ^{n-\left\vert S\right\vert -1}\sum_{\substack{t\in V;\\t\notin
S}}x_{t}\right) \\
&  =\sum_{\substack{S\subseteq V;\\S\neq\varnothing;\ S\neq V}}X\left(
X+x\left(  S\right)  \right)  ^{\left\vert S\right\vert -1}\underbrace{\left(
Y-x\left(  S\right)  \right)  ^{n-\left\vert S\right\vert -1}\left(
Y-x\left(  S\right)  \right)  }_{=\left(  Y-x\left(  S\right)  \right)
^{n-\left\vert S\right\vert }}\\
&  \ \ \ \ \ \ \ \ \ \ -\sum_{\substack{S\subseteq V;\\S\neq\varnothing
;\ S\neq V}}\underbrace{X\left(  X+x\left(  S\right)  \right)  ^{\left\vert
S\right\vert -1}\left(  Y-x\left(  S\right)  \right)  ^{n-\left\vert
S\right\vert -1}\sum_{\substack{t\in V;\\t\notin S}}x_{t}}_{=\sum
_{\substack{t\in V;\\t\notin S}}X\left(  X+x\left(  S\right)  \right)
^{\left\vert S\right\vert -1}\left(  Y-x\left(  S\right)  \right)
^{n-\left\vert S\right\vert -1}x_{t}}%
\end{align*}%
\begin{align*}
&  =\underbrace{\sum_{\substack{S\subseteq V;\\S\neq\varnothing;\ S\neq
V}}X\left(  X+x\left(  S\right)  \right)  ^{\left\vert S\right\vert -1}\left(
Y-x\left(  S\right)  \right)  ^{n-\left\vert S\right\vert }}%
_{\substack{=\left(  X+Y\right)  ^{n}-Y^{n}-X\left(  X+x\left(  V\right)
\right)  ^{n-1}\\\text{(by (\ref{pf.lem.4.6}))}}}\\
&  \ \ \ \ \ \ \ \ \ \ -\underbrace{\sum_{\substack{S\subseteq V;\\S\neq
\varnothing;\ S\neq V}}\ \ \sum_{\substack{t\in V;\\t\notin S}}}_{=\sum_{t\in
V}\ \ \sum_{\substack{S\subseteq V;\\S\neq\varnothing;\ S\neq V;\\t\notin S}%
}}X\left(  X+x\left(  S\right)  \right)  ^{\left\vert S\right\vert -1}\left(
Y-x\left(  S\right)  \right)  ^{n-\left\vert S\right\vert -1}x_{t}\\
&  =\left(  X+Y\right)  ^{n}-Y^{n}-X\left(  X+x\left(  V\right)  \right)
^{n-1}\\
&  \ \ \ \ \ \ \ \ \ \ -\sum_{t\in V}\ \ \underbrace{\sum
_{\substack{S\subseteq V;\\S\neq\varnothing;\ S\neq V;\\t\notin S}}X\left(
X+x\left(  S\right)  \right)  ^{\left\vert S\right\vert -1}\left(  Y-x\left(
S\right)  \right)  ^{n-\left\vert S\right\vert -1}x_{t}}_{\substack{=\left(
\left(  X+Y\right)  ^{n-1}-Y^{n-1}\right)  x_{t}\\\text{(by
(\ref{pf.lem.4.t-eq1}))}}}\\
&  =\left(  X+Y\right)  ^{n}-Y^{n}-X\left(  X+x\left(  V\right)  \right)
^{n-1}-\underbrace{\sum_{t\in V}\left(  \left(  X+Y\right)  ^{n-1}%
-Y^{n-1}\right)  x_{t}}_{=\left(  \left(  X+Y\right)  ^{n-1}-Y^{n-1}\right)
\sum_{t\in V}x_{t}}\\
&  =\left(  X+Y\right)  ^{n}-Y^{n}-X\left(  X+x\left(  V\right)  \right)
^{n-1}-\left(  \left(  X+Y\right)  ^{n-1}-Y^{n-1}\right)  \underbrace{\sum
_{t\in V}x_{t}}_{\substack{=x\left(  V\right)  \\\text{(by
(\ref{pf.lem.4.x(V)}))}}}
\end{align*}%
\begin{align*}
&  =\underbrace{\left(  X+Y\right)  ^{n}}_{=\left(  X+Y\right)  ^{n-1}\left(
X+Y\right)  }-\underbrace{Y^{n}}_{=Y^{n-1}Y}-\,X\left(  X+x\left(  V\right)
\right)  ^{n-1}-\underbrace{\left(  \left(  X+Y\right)  ^{n-1}-Y^{n-1}\right)
x\left(  V\right)  }_{=\left(  X+Y\right)  ^{n-1}x\left(  V\right)
-Y^{n-1}x\left(  V\right)  }\\
&  =\left(  X+Y\right)  ^{n-1}\left(  X+Y\right)  -Y^{n-1}Y-X\left(
X+x\left(  V\right)  \right)  ^{n-1}-\left(  \left(  X+Y\right)
^{n-1}x\left(  V\right)  -Y^{n-1}x\left(  V\right)  \right) \\
&  =\left(  X+Y\right)  ^{n-1}\left(  X+Y\right)  -Y^{n-1}Y-X\left(
X+x\left(  V\right)  \right)  ^{n-1}-\left(  X+Y\right)  ^{n-1}x\left(
V\right)  +Y^{n-1}x\left(  V\right) \\
&  =\underbrace{\left(  \left(  X+Y\right)  ^{n-1}\left(  X+Y\right)  -\left(
X+Y\right)  ^{n-1}x\left(  V\right)  \right)  }_{=\left(  X+Y\right)
^{n-1}\left(  \left(  X+Y\right)  -x\left(  V\right)  \right)  }%
-\underbrace{\left(  Y^{n-1}Y-Y^{n-1}x\left(  V\right)  \right)  }%
_{=Y^{n-1}\left(  Y-x\left(  V\right)  \right)  }-\,X\left(  X+x\left(
V\right)  \right)  ^{n-1}\\
&  =\underbrace{\left(  X+Y\right)  ^{n-1}\left(  \left(  X+Y\right)
-x\left(  V\right)  \right)  }_{\substack{=\left(  \left(  X+Y\right)
-x\left(  V\right)  \right)  \left(  X+Y\right)  ^{n-1}\\\text{(since }\left(
X+Y\right)  ^{n-1}\text{ commutes with }\left(  X+Y\right)  -x\left(
V\right)  \\\text{(since }\left(  X+Y\right)  ^{n-1}\text{ lies in the center
of }\mathbb{L}\\\text{(since }X+Y\text{ lies in the center of }\mathbb{L}%
\text{,}\\\text{but the center of }\mathbb{L}\text{ is a subring of
}\mathbb{L}\text{)))}}}-\,Y^{n-1}\left(  Y-x\left(  V\right)  \right)
-X\left(  X+x\left(  V\right)  \right)  ^{n-1}\\
&  =\left(  \left(  X+Y\right)  -x\left(  V\right)  \right)  \left(
X+Y\right)  ^{n-1}-Y^{n-1}\left(  Y-x\left(  V\right)  \right)  -X\left(
X+x\left(  V\right)  \right)  ^{n-1}.
\end{align*}
Solving this equality for $\left(  \left(  X+Y\right)  -x\left(  V\right)
\right)  \left(  X+Y\right)  ^{n-1}$, we obtain%
\begin{align*}
&  \left(  \left(  X+Y\right)  -x\left(  V\right)  \right)  \left(
X+Y\right)  ^{n-1}\\
&  =\sum_{\substack{S\subseteq V;\\S\neq\varnothing;\ S\neq V}}X\left(
X+x\left(  S\right)  \right)  ^{\left\vert S\right\vert -1}\left(  Y-x\left(
S\right)  \right)  ^{n-\left\vert S\right\vert -1}\left(  Y-x\left(  V\right)
\right) \\
&  \ \ \ \ \ \ \ \ \ \ +Y^{n-1}\left(  Y-x\left(  V\right)  \right)  +X\left(
X+x\left(  V\right)  \right)  ^{n-1}\\
&  =Y^{n-1}\left(  Y-x\left(  V\right)  \right)  +X\left(  X+x\left(
V\right)  \right)  ^{n-1}\\
&  \ \ \ \ \ \ \ \ \ \ +\sum_{\substack{S\subseteq V;\\S\neq\varnothing
;\ S\neq V}}X\left(  X+x\left(  S\right)  \right)  ^{\left\vert S\right\vert
-1}\left(  Y-x\left(  S\right)  \right)  ^{n-\left\vert S\right\vert
-1}\left(  Y-x\left(  V\right)  \right)  .
\end{align*}
Hence,%
\begin{align*}
&  Y^{n-1}\left(  Y-x\left(  V\right)  \right)  +X\left(  X+x\left(  V\right)
\right)  ^{n-1}\\
&  \ \ \ \ \ \ \ \ \ \ +\sum_{\substack{S\subseteq V;\\S\neq\varnothing
;\ S\neq V}}X\left(  X+x\left(  S\right)  \right)  ^{\left\vert S\right\vert
-1}\left(  Y-x\left(  S\right)  \right)  ^{n-\left\vert S\right\vert
-1}\left(  Y-x\left(  V\right)  \right) \\
&  =\underbrace{\left(  \left(  X+Y\right)  -x\left(  V\right)  \right)
}_{=X+Y-x\left(  V\right)  }\left(  X+Y\right)  ^{n-1}\\
&  =\left(  X+Y-x\left(  V\right)  \right)  \left(  X+Y\right)  ^{n-1}.
\end{align*}
This proves Lemma \ref{lem.4}.
\end{proof}

\begin{proof}
[Proof of Theorem \ref{thm.4}.]Theorem \ref{thm.4} holds in the case when
$V=\varnothing$\ \ \ \ \footnote{\textit{Proof.} Assume that $V=\varnothing$.
We must prove that Theorem \ref{thm.4} holds.
\par
We have $V=\varnothing$. Hence, $\left\vert V\right\vert =\left\vert
\varnothing\right\vert =0$. Thus, $n=\left\vert V\right\vert =0$.
\par
But $V$ is the empty set (since $V=\varnothing$). Hence, the only subset $S$
of $V$ is the empty set $\varnothing$. Thus, the sum $\sum_{S\subseteq
V}X\left(  X+\sum_{s\in S}x_{s}\right)  ^{\left\vert S\right\vert -1}\left(
Y-\sum_{s\in S}x_{s}\right)  ^{n-\left\vert S\right\vert -1}\left(
Y-\sum_{s\in V}x_{s}\right)  $ has only one addend: namely, the addend for
$S=\varnothing$. Therefore, this sum simplifies as follows:%
\begin{align*}
&  \sum_{S\subseteq V}X\left(  X+\sum_{s\in S}x_{s}\right)  ^{\left\vert
S\right\vert -1}\left(  Y-\sum_{s\in S}x_{s}\right)  ^{n-\left\vert
S\right\vert -1}\left(  Y-\sum_{s\in V}x_{s}\right) \\
&  =\underbrace{X\left(  X+\sum_{s\in\varnothing}x_{s}\right)  ^{\left\vert
\varnothing\right\vert -1}}_{\substack{=1\\\text{(according to our convention
for}\\\text{interpreting }X\left(  X+\sum_{s\in S}x_{s}\right)  ^{\left\vert
S\right\vert -1}\\\text{when }S=\varnothing\text{)}}}\underbrace{\left(
Y-\sum_{s\in\varnothing}x_{s}\right)  ^{n-\left\vert \varnothing\right\vert
-1}\left(  Y-\sum_{s\in V}x_{s}\right)  }_{\substack{=1\\\text{(according to
our convention for}\\\text{interpreting }\left(  Y-\sum_{s\in S}x_{s}\right)
^{n-\left\vert S\right\vert -1}\left(  Y-\sum_{s\in V}x_{s}\right)
\\\text{when }\left\vert S\right\vert =n\text{)}}}\\
&  =1.
\end{align*}
Comparing this with%
\[
\left(  X+Y-\sum_{s\in V}x_{s}\right)  \left(  X+Y\right)  ^{n-1}%
=1\ \ \ \ \ \ \ \ \ \ \left(
\begin{array}
[c]{c}%
\text{according to our convention for}\\
\text{interpreting }\left(  X+Y-\sum_{s\in V}x_{s}\right)  \left(  X+Y\right)
^{n-1}\\
\text{when }n=0
\end{array}
\right)  ,
\]
we obtain%
\begin{align*}
&  \sum_{S\subseteq V}X\left(  X+\sum_{s\in S}x_{s}\right)  ^{\left\vert
S\right\vert -1}\left(  Y-\sum_{s\in S}x_{s}\right)  ^{n-\left\vert
S\right\vert -1}\left(  Y-\sum_{s\in V}x_{s}\right) \\
&  =\left(  X+Y-\sum_{s\in V}x_{s}\right)  \left(  X+Y\right)  ^{n-1}.
\end{align*}
In other words, Theorem \ref{thm.4} holds. Qed.}. Hence, for the rest of this
proof, we can WLOG assume that we don't have $V=\varnothing$. Assume this.

We have $V\neq\varnothing$ (since we don't have $V=\varnothing$). Hence, the
set $V$ is nonempty. Thus, Lemma \ref{lem.4} yields%
\begin{align}
&  Y^{n-1}\left(  Y-x\left(  V\right)  \right)  +X\left(  X+x\left(  V\right)
\right)  ^{n-1}\nonumber\\
&  \ \ \ \ \ \ \ \ \ \ +\sum_{\substack{S\subseteq V;\\S\neq\varnothing
;\ S\neq V}}X\left(  X+x\left(  S\right)  \right)  ^{\left\vert S\right\vert
-1}\left(  Y-x\left(  S\right)  \right)  ^{n-\left\vert S\right\vert
-1}\left(  Y-x\left(  V\right)  \right) \nonumber\\
&  =\left(  X+Y-x\left(  V\right)  \right)  \left(  X+Y\right)  ^{n-1}.
\label{pf.thm.4.1}%
\end{align}

We have $n-\underbrace{\left\vert \varnothing\right\vert }_{=0}-\,1=n-0-1=n-1$.

Every subset $S$ of $V$ satisfies%
\begin{equation}
x\left(  S\right)  =\sum_{s\in S}x_{s} \label{pf.thm.4.2}%
\end{equation}
(by the definition of $x\left(  S\right)  $). Applying this to $S=V$, we
obtain%
\begin{equation}
x\left(  V\right)  =\sum_{s\in V}x_{s}. \label{pf.thm.4.3}%
\end{equation}

But the set $V$ is a subset $S$ of $V$ satisfying $S\neq\varnothing$ (since
$V$ is a subset of $V$ and since $V\neq\varnothing$). Hence, the sum
$\sum_{\substack{S\subseteq V;\\S\neq\varnothing}}X\left(  X+\sum_{s\in
S}x_{s}\right)  ^{\left\vert S\right\vert -1}\left(  Y-\sum_{s\in S}%
x_{s}\right)  ^{n-\left\vert S\right\vert -1}\left(  Y-\sum_{s\in V}%
x_{s}\right)  $ has an addend for $S=V$. If we split off this addend from the
sum, then we obtain%
\begin{align}
&  \sum_{\substack{S\subseteq V;\\S\neq\varnothing}}X\left(  X+\sum_{s\in
S}x_{s}\right)  ^{\left\vert S\right\vert -1}\left(  Y-\sum_{s\in S}%
x_{s}\right)  ^{n-\left\vert S\right\vert -1}\left(  Y-\sum_{s\in V}%
x_{s}\right) \nonumber\\
&  =X\underbrace{\left(  X+\sum_{s\in V}x_{s}\right)  ^{\left\vert
V\right\vert -1}}_{\substack{=\left(  X+\sum_{s\in V}x_{s}\right)
^{n-1}\\\text{(since }\left\vert V\right\vert =n\text{)}}}\underbrace{\left(
Y-\sum_{s\in V}x_{s}\right)  ^{n-\left\vert V\right\vert -1}\left(
Y-\sum_{s\in V}x_{s}\right)  }_{\substack{=1\\\text{(according to our
convention for}\\\text{interpreting }\left(  Y-\sum_{s\in S}x_{s}\right)
^{n-\left\vert S\right\vert -1}\left(  Y-\sum_{s\in V}x_{s}\right)
\\\text{when }\left\vert S\right\vert =n\text{)}}}\nonumber\\
&  \ \ \ \ \ \ \ \ \ \ +\sum_{\substack{S\subseteq V;\\S\neq\varnothing
;\ S\neq V}}X\left(  X+\underbrace{\sum_{s\in S}x_{s}}_{\substack{=x\left(
S\right)  \\\text{(by (\ref{pf.thm.4.2}))}}}\right)  ^{\left\vert S\right\vert
-1}\left(  Y-\underbrace{\sum_{s\in S}x_{s}}_{\substack{=x\left(  S\right)
\\\text{(by (\ref{pf.thm.4.2}))}}}\right)  ^{n-\left\vert S\right\vert
-1}\left(  Y-\underbrace{\sum_{s\in V}x_{s}}_{\substack{=x\left(  V\right)
\\\text{(by (\ref{pf.thm.4.3}))}}}\right) \nonumber\\
&  =X\left(  X+\underbrace{\sum_{s\in V}x_{s}}_{\substack{=x\left(  V\right)
\\\text{(by (\ref{pf.thm.4.3}))}}}\right)  ^{n-1}+\sum_{\substack{S\subseteq
V;\\S\neq\varnothing;\ S\neq V}}X\left(  X+x\left(  S\right)  \right)
^{\left\vert S\right\vert -1}\left(  Y-x\left(  S\right)  \right)
^{n-\left\vert S\right\vert -1}\left(  Y-x\left(  V\right)  \right)
\nonumber\\
&  =X\left(  X+x\left(  V\right)  \right)  ^{n-1}+\sum_{\substack{S\subseteq
V;\\S\neq\varnothing;\ S\neq V}}X\left(  X+x\left(  S\right)  \right)
^{\left\vert S\right\vert -1}\left(  Y-x\left(  S\right)  \right)
^{n-\left\vert S\right\vert -1}\left(  Y-x\left(  V\right)  \right)  .
\label{pf.thm.4.8}%
\end{align}
Now,%
\begin{align*}
&  \sum_{S\subseteq V}X\left(  X+\sum_{s\in S}x_{s}\right)  ^{\left\vert
S\right\vert -1}\left(  Y-\sum_{s\in S}x_{s}\right)  ^{n-\left\vert
S\right\vert -1}\left(  Y-\sum_{s\in V}x_{s}\right) \\
&  =\underbrace{X\left(  X+\sum_{s\in\varnothing}x_{s}\right)  ^{\left\vert
\varnothing\right\vert -1}}_{\substack{=1\\\text{(according to our convention
for}\\\text{interpreting }X\left(  X+\sum_{s\in S}x_{s}\right)  ^{\left\vert
S\right\vert -1}\\\text{when }S=\varnothing\text{)}}}\left(
Y-\underbrace{\sum_{s\in\varnothing}x_{s}}_{=\left(  \text{empty sum}\right)
=0}\right)  ^{n-\left\vert \varnothing\right\vert -1}\left(  Y-\sum_{s\in
V}x_{s}\right) \\
&  \ \ \ \ \ \ \ \ \ \ +\underbrace{\sum_{\substack{S\subseteq V;\\S\neq
\varnothing}}X\left(  X+\sum_{s\in S}x_{s}\right)  ^{\left\vert S\right\vert
-1}\left(  Y-\sum_{s\in S}x_{s}\right)  ^{n-\left\vert S\right\vert -1}\left(
Y-\sum_{s\in V}x_{s}\right)  }_{\substack{=X\left(  X+x\left(  V\right)
\right)  ^{n-1}+\sum_{\substack{S\subseteq V;\\S\neq\varnothing;\ S\neq
V}}X\left(  X+x\left(  S\right)  \right)  ^{\left\vert S\right\vert -1}\left(
Y-x\left(  S\right)  \right)  ^{n-\left\vert S\right\vert -1}\left(
Y-x\left(  V\right)  \right)  \\\text{(by (\ref{pf.thm.4.8}))}}}\\
&  \ \ \ \ \ \ \ \ \ \ \ \ \ \ \ \ \ \ \ \ \left(
\begin{array}
[c]{c}%
\text{here, we have split off the addend for }S=\varnothing\text{ from the
sum}\\
\text{(since }\varnothing\text{ is a subset of }V\text{)}%
\end{array}
\right) \\
&  =\underbrace{\left(  Y-0\right)  ^{n-\left\vert \varnothing\right\vert -1}%
}_{\substack{=Y^{n-1}\\\text{(since }Y-0=Y\text{ and }n-\left\vert
\varnothing\right\vert -1=n-1\text{)}}}\left(  Y-\underbrace{\sum_{s\in
V}x_{s}}_{\substack{=x\left(  V\right)  \\\text{(by (\ref{pf.thm.4.3}))}%
}}\right) \\
&  \ \ \ \ \ \ \ \ \ \ +X\left(  X+x\left(  V\right)  \right)  ^{n-1}%
+\sum_{\substack{S\subseteq V;\\S\neq\varnothing;\ S\neq V}}X\left(
X+x\left(  S\right)  \right)  ^{\left\vert S\right\vert -1}\left(  Y-x\left(
S\right)  \right)  ^{n-\left\vert S\right\vert -1}\left(  Y-x\left(  V\right)
\right) \\
&  =Y^{n-1}\left(  Y-x\left(  V\right)  \right)  +X\left(  X+x\left(
V\right)  \right)  ^{n-1}\\
&  \ \ \ \ \ \ \ \ \ \ +\sum_{\substack{S\subseteq V;\\S\neq\varnothing
;\ S\neq V}}X\left(  X+x\left(  S\right)  \right)  ^{\left\vert S\right\vert
-1}\left(  Y-x\left(  S\right)  \right)  ^{n-\left\vert S\right\vert
-1}\left(  Y-x\left(  V\right)  \right) \\
&  =\left(  X+Y-\underbrace{x\left(  V\right)  }_{\substack{=\sum_{s\in
V}x_{s}\\\text{(by (\ref{pf.thm.4.3}))}}}\right)  \left(  X+Y\right)
^{n-1}\ \ \ \ \ \ \ \ \ \ \left(  \text{by (\ref{pf.thm.4.1})}\right) \\
&  =\left(  X+Y-\sum_{s\in V}x_{s}\right)  \left(  X+Y\right)  ^{n-1}.
\end{align*}
This proves Theorem \ref{thm.4}.
\end{proof}
\end{verlong}

\begin{vershort}
\begin{proof}
[Proof of Theorem \ref{thm.5}.]Apply Theorem \ref{thm.4} to $Y+\sum_{s\in
V}x_{s}$ instead of $Y$.
\end{proof}
\end{vershort}

\begin{verlong}
\begin{proof}
[Proof of Theorem \ref{thm.5}.]We know that $X+Y+\sum_{s\in V}x_{s}$ lies in
the center of $\mathbb{L}$. In other words, $X+\left(  Y+\sum_{s\in V}%
x_{s}\right)  $ lies in the center of $\mathbb{L}$ (since $X+\left(
Y+\sum_{s\in V}x_{s}\right)  =X+Y+\sum_{s\in V}x_{s}$). Thus, Theorem
\ref{thm.4} (applied to $Y+\sum_{s\in V}x_{s}$ instead of $Y$) yields%
\begin{align}
&  \sum_{S\subseteq V}X\left(  X+\sum_{s\in S}x_{s}\right)  ^{\left\vert
S\right\vert -1}\left(  Y+\sum_{s\in V}x_{s}-\sum_{s\in S}x_{s}\right)
^{n-\left\vert S\right\vert -1}\left(  Y+\sum_{s\in V}x_{s}-\sum_{s\in V}%
x_{s}\right) \nonumber\\
&  =\left(  X+\underbrace{Y+\sum_{s\in V}x_{s}-\sum_{s\in V}x_{s}}%
_{=Y}\right)  \left(  X+Y+\sum_{s\in V}x_{s}\right)  ^{n-1}\nonumber\\
&  =\left(  X+Y\right)  \left(  X+Y+\sum_{s\in V}x_{s}\right)  ^{n-1}.
\label{pf.thm.5.1}%
\end{align}

But every subset $S$ of $V$ satisfies%
\begin{equation}
\sum_{s\in V}x_{s}-\sum_{s\in S}x_{s}=\sum_{s\in V\setminus S}x_{s}
\label{pf.thm.5.3}%
\end{equation}
\footnote{\textit{Proof of (\ref{pf.thm.5.3}):} Let $S$ be a subset of $V$.
Then, each element $s\in V$ satisfies either $s\in S$ or $s\notin S$ (but not
both). Hence,%
\[
\sum_{s\in V}x_{s}=\underbrace{\sum_{\substack{s\in V;\\s\in S}}}%
_{\substack{=\sum_{s\in S}\\\text{(since }S\text{ is a subset of }V\text{)}%
}}x_{s}+\underbrace{\sum_{\substack{s\in V;\\s\notin S}}}_{=\sum_{s\in
V\setminus S}}x_{s}=\sum_{s\in S}x_{s}+\sum_{s\in V\setminus S}x_{s}.
\]
Hence,%
\[
\underbrace{\sum_{s\in V}x_{s}}_{=\sum_{s\in S}x_{s}+\sum_{s\in V\setminus
S}x_{s}}-\sum_{s\in S}x_{s}=\sum_{s\in S}x_{s}+\sum_{s\in V\setminus S}%
x_{s}-\sum_{s\in S}x_{s}=\sum_{s\in V\setminus S}x_{s}.
\]
This proves (\ref{pf.thm.5.3}).}. Now,%
\begin{align*}
&  \sum_{S\subseteq V}X\left(  X+\sum_{s\in S}x_{s}\right)  ^{\left\vert
S\right\vert -1}\left(  Y+\underbrace{\sum_{s\in V}x_{s}-\sum_{s\in S}x_{s}%
}_{\substack{=\sum_{s\in V\setminus S}x_{s}\\\text{(by (\ref{pf.thm.5.3}))}%
}}\right)  ^{n-\left\vert S\right\vert -1}\underbrace{\left(  Y+\sum_{s\in
V}x_{s}-\sum_{s\in V}x_{s}\right)  }_{=Y}\\
&  =\sum_{S\subseteq V}X\left(  X+\sum_{s\in S}x_{s}\right)  ^{\left\vert
S\right\vert -1}\left(  Y+\sum_{s\in V\setminus S}x_{s}\right)  ^{n-\left\vert
S\right\vert -1}Y.
\end{align*}
Comparing this with (\ref{pf.thm.5.1}), we obtain%
\begin{align*}
&  \sum_{S\subseteq V}X\left(  X+\sum_{s\in S}x_{s}\right)  ^{\left\vert
S\right\vert -1}\left(  Y+\sum_{s\in V\setminus S}x_{s}\right)  ^{n-\left\vert
S\right\vert -1}Y\\
&  =\left(  X+Y\right)  \left(  X+Y+\sum_{s\in V}x_{s}\right)  ^{n-1}.
\end{align*}
This proves Theorem \ref{thm.5}.
\end{proof}
\end{verlong}

\section{Applications}

\subsection{Polarization identities}

Let us show how a rather classical identity in noncommutative rings follows as
a particular case from Theorem \ref{thm.1}. Namely, we shall prove the
following \textit{polarization identity}:

\begin{corollary}
\label{cor.polar1}Let $V$ be a finite set. Let $n=\left\vert V\right\vert $.
For each $s\in V$, let $x_{s}$ be an element of $\mathbb{L}$. Let
$X\in\mathbb{L}$. Then,%
\[
\sum_{S\subseteq V}\left(  -1\right)  ^{n-\left\vert S\right\vert }\left(
X+\sum_{s\in S}x_{s}\right)  ^{n}=\sum_{\substack{\left(  i_{1},i_{2}%
,\ldots,i_{n}\right)  \text{ is a list}\\\text{of all elements of
}V\\\text{(with no repetitions)}}}x_{i_{1}}x_{i_{2}}\cdots x_{i_{n}}.
\]

\end{corollary}

\begin{vershort}
\begin{proof}
[Proof of Corollary \ref{cor.polar1}.]Apply Theorem \ref{thm.1} to $Y=-X$, and
notice how all addends on the right hand side having $k<n$ vanish (since
$\left(  X+\left(  -X\right)  \right)  ^{n-k}=0$ for $k<n$), whereas the
remaining addends are precisely the addends of the sum \newline$\sum
_{\substack{\left(  i_{1},i_{2},\ldots,i_{n}\right)  \text{ is a
list}\\\text{of all elements of }V\\\text{(with no repetitions)}}}x_{i_{1}%
}x_{i_{2}}\cdots x_{i_{n}}$.
\end{proof}
\end{vershort}

\begin{verlong}
\begin{proof}
[Proof of Corollary \ref{cor.polar1}.]If $\left(  i_{1},i_{2},\ldots
,i_{k}\right)  $ is a $k$-tuple of distinct elements of $V$ for some
$k\in\mathbb{N}$, then the statement $\left(  n\leq k\right)  $ is equivalent
to the statement $\left(  k=n\right)  $\ \ \ \ \footnote{\textit{Proof.} Let
$\left(  i_{1},i_{2},\ldots,i_{k}\right)  $ be a $k$-tuple of distinct
elements of $V$ for some $k\in\mathbb{N}$. We must prove that the statement
$\left(  n\leq k\right)  $ is equivalent to the statement $\left(  k=n\right)
$.
\par
The elements $i_{1},i_{2},\ldots,i_{k}$ are distinct (since $\left(
i_{1},i_{2},\ldots,i_{k}\right)  $ is a $k$-tuple of distinct elements of
$V$). Hence, $\left\vert \left\{  i_{1},i_{2},\ldots,i_{k}\right\}
\right\vert =k$.
\par
But $\left(  i_{1},i_{2},\ldots,i_{k}\right)  $ is a $k$-tuple of elements of
$V$. Hence, $i_{1},i_{2},\ldots,i_{k}$ are elements of $V$. In other words,
$\left\{  i_{1},i_{2},\ldots,i_{k}\right\}  \subseteq V$. Hence, $\left\vert
\left\{  i_{1},i_{2},\ldots,i_{k}\right\}  \right\vert \leq\left\vert
V\right\vert =n$. Thus, $n\geq\left\vert \left\{  i_{1},i_{2},\ldots
,i_{k}\right\}  \right\vert =k$. Hence, if $n\leq k$, then $n=k$ (because
combining $n\leq k$ with $n\geq k$ yields $n=k$). In other words, the
implication $\left(  n\leq k\right)  \Longrightarrow\left(  n=k\right)  $
holds. Conversely, the implication $\left(  n=k\right)  \Longrightarrow\left(
n\leq k\right)  $ holds (obviously). Combining this implication with the
implication $\left(  n\leq k\right)  \Longrightarrow\left(  n=k\right)  $, we
obtain the equivalence $\left(  n\leq k\right)  \Longleftrightarrow\left(
n=k\right)  $. Hence, we obtain the following chain of equivalences: $\left(
n\leq k\right)  \Longleftrightarrow\left(  n=k\right)  \Longleftrightarrow
\left(  k=n\right)  $. Thus, the statement $\left(  n\leq k\right)  $ is
equivalent to the statement $\left(  k=n\right)  $. Qed.}. Hence, we have the
following equality of summation signs:%
\begin{equation}
\sum_{\substack{i_{1},i_{2},\ldots,i_{k}\text{ are}\\\text{distinct}%
\\\text{elements of }V;\\n\leq k}}=\sum_{\substack{i_{1},i_{2},\ldots
,i_{k}\text{ are}\\\text{distinct}\\\text{elements of }V;\\k=n}}.
\label{pf.cor.polar1.sumeq1}%
\end{equation}

Furthermore, let us define a set $\mathfrak{A}$ by%
\[
\mathfrak{A}=\left\{  \left(  i_{1},i_{2},\ldots,i_{n}\right)  \ \mid\ \left(
i_{1},i_{2},\ldots,i_{n}\right)  \text{ is a list of all elements of }V\text{
(with no repetitions)}\right\}  .
\]
Let us also define a set $\mathfrak{B}$ by%
\[
\mathfrak{B}=\left\{  \left(  i_{1},i_{2},\ldots,i_{n}\right)  \ \mid\ \left(
i_{1},i_{2},\ldots,i_{n}\right)  \text{ is a list of distinct elements of
}V\right\}  .
\]
Then, $\mathfrak{A}\subseteq\mathfrak{B}$\ \ \ \ \footnote{\textit{Proof.} Let
$j\in\mathfrak{A}$. Thus,
\begin{align*}
j  &  \in\mathfrak{A}=\left\{  \left(  i_{1},i_{2},\ldots,i_{n}\right)
\ \mid\ \left(  i_{1},i_{2},\ldots,i_{n}\right)  \text{ is a list of all
elements of }V\text{ (with no repetitions)}\right\} \\
&  =\left\{  \left(  j_{1},j_{2},\ldots,j_{n}\right)  \ \mid\ \left(
j_{1},j_{2},\ldots,j_{n}\right)  \text{ is a list of all elements of }V\text{
(with no repetitions)}\right\}
\end{align*}
(here, we have renamed the index $\left(  i_{1},i_{2},\ldots,i_{n}\right)  $
as $\left(  j_{1},j_{2},\ldots,j_{n}\right)  $). In other words, $j$ can be
written in the form $j=\left(  j_{1},j_{2},\ldots,j_{n}\right)  $, where
$\left(  j_{1},j_{2},\ldots,j_{n}\right)  $ is a list of all elements of $V$
(with no repetitions). Consider this $\left(  j_{1},j_{2},\ldots,j_{n}\right)
$.
\par
The list $\left(  j_{1},j_{2},\ldots,j_{n}\right)  $ is a list of elements of
$V$, and its entries are distinct (since $\left(  j_{1},j_{2},\ldots
,j_{n}\right)  $ is a list with no repetitions). In other words, $\left(
j_{1},j_{2},\ldots,j_{n}\right)  $ is a list of distinct elements of $V$.
Thus, $\left(  j_{1},j_{2},\ldots,j_{n}\right)  $ has the form $\left(
j_{1},j_{2},\ldots,j_{n}\right)  =\left(  i_{1},i_{2},\ldots,i_{n}\right)  $,
where $\left(  i_{1},i_{2},\ldots,i_{n}\right)  $ is a list of distinct
elements of $V$ (namely, $\left(  i_{1},i_{2},\ldots,i_{n}\right)  =\left(
j_{1},j_{2},\ldots,j_{n}\right)  $). In other words,%
\[
\left(  j_{1},j_{2},\ldots,j_{n}\right)  \in\left\{  \left(  i_{1}%
,i_{2},\ldots,i_{n}\right)  \ \mid\ \left(  i_{1},i_{2},\ldots,i_{n}\right)
\text{ is a list of distinct elements of }V\right\}  .
\]
In light of $j=\left(  j_{1},j_{2},\ldots,j_{n}\right)  $ and $\mathfrak{B}%
=\left\{  \left(  i_{1},i_{2},\ldots,i_{n}\right)  \ \mid\ \left(  i_{1}%
,i_{2},\ldots,i_{n}\right)  \text{ is a list of distinct elements of
}V\right\}  $, this rewrites as $j\in\mathfrak{B}$.
\par
Now, forget that we fixed $j$. We thus have shown that $j\in\mathfrak{B}$ for
each $j\in\mathfrak{A}$. In other words, $\mathfrak{A}\subseteq\mathfrak{B}$.}
and $\mathfrak{B}\subseteq\mathfrak{A}$\ \ \ \ \footnote{\textit{Proof.} Let
$j\in\mathfrak{B}$. Thus,
\begin{align*}
j  &  \in\mathfrak{B}=\left\{  \left(  i_{1},i_{2},\ldots,i_{n}\right)
\ \mid\ \left(  i_{1},i_{2},\ldots,i_{n}\right)  \text{ is a list of distinct
elements of }V\right\} \\
&  =\left\{  \left(  j_{1},j_{2},\ldots,j_{n}\right)  \ \mid\ \left(
j_{1},j_{2},\ldots,j_{n}\right)  \text{ is a list of distinct elements of
}V\right\}
\end{align*}
(here, we have renamed the index $\left(  i_{1},i_{2},\ldots,i_{n}\right)  $
as $\left(  j_{1},j_{2},\ldots,j_{n}\right)  $). In other words, $j$ can be
written in the form $j=\left(  j_{1},j_{2},\ldots,j_{n}\right)  $, where
$\left(  j_{1},j_{2},\ldots,j_{n}\right)  $ is a list of distinct elements of
$V$. Consider this $\left(  j_{1},j_{2},\ldots,j_{n}\right)  $.
\par
The $n$ elements $j_{1},j_{2},\ldots,j_{n}$ are distinct (since $\left(
j_{1},j_{2},\ldots,j_{n}\right)  $ is a list of distinct elements of $V$).
Thus, $\left\vert \left\{  j_{1},j_{2},\ldots,j_{n}\right\}  \right\vert =n$.
But $j_{1},j_{2},\ldots,j_{n}$ are elements of $V$ (since $\left(  j_{1}%
,j_{2},\ldots,j_{n}\right)  $ is a list of elements of $V$). Hence, $\left\{
j_{1},j_{2},\ldots,j_{n}\right\}  $ is a subset of $V$.
\par
It is well-known that if $B$ is a finite set, and if $A$ is a subset of $B$
satisfying $\left\vert A\right\vert =\left\vert B\right\vert $, then $A=B$.
Applying this to $B=V$ and $A=\left\{  j_{1},j_{2},\ldots,j_{n}\right\}  $, we
obtain $\left\{  j_{1},j_{2},\ldots,j_{n}\right\}  =V$ (since $\left\vert
\left\{  j_{1},j_{2},\ldots,j_{n}\right\}  \right\vert =n=\left\vert
V\right\vert $). Thus, the list $\left(  j_{1},j_{2},\ldots,j_{n}\right)  $ is
a list of all elements of $V$. Furthermore, this list $\left(  j_{1}%
,j_{2},\ldots,j_{n}\right)  $ has no repetitions (since the $n$ elements
$j_{1},j_{2},\ldots,j_{n}$ are distinct). Thus, the list $\left(  j_{1}%
,j_{2},\ldots,j_{n}\right)  $ is a list of all elements of $V$ (with no
repetitions). Thus, $\left(  j_{1},j_{2},\ldots,j_{n}\right)  $ has the form
$\left(  j_{1},j_{2},\ldots,j_{n}\right)  =\left(  i_{1},i_{2},\ldots
,i_{n}\right)  $, where $\left(  i_{1},i_{2},\ldots,i_{n}\right)  $ is a list
of all elements of $V$ (with no repetitions) (namely, $\left(  i_{1}%
,i_{2},\ldots,i_{n}\right)  =\left(  j_{1},j_{2},\ldots,j_{n}\right)  $). In
other words,%
\[
\left(  j_{1},j_{2},\ldots,j_{n}\right)  \in\left\{  \left(  i_{1}%
,i_{2},\ldots,i_{n}\right)  \ \mid\ \left(  i_{1},i_{2},\ldots,i_{n}\right)
\text{ is a list of all elements of }V\text{ (with no repetitions)}\right\}
.
\]
In light of $j=\left(  j_{1},j_{2},\ldots,j_{n}\right)  $ and \newline%
$\mathfrak{A}=\left\{  \left(  i_{1},i_{2},\ldots,i_{n}\right)  \ \mid
\ \left(  i_{1},i_{2},\ldots,i_{n}\right)  \text{ is a list of all elements of
}V\text{ (with no repetitions)}\right\}  $, this rewrites as $j\in
\mathfrak{A}$.
\par
Now, forget that we fixed $j$. We thus have shown that $j\in\mathfrak{A}$ for
each $j\in\mathfrak{B}$. In other words, $\mathfrak{B}\subseteq\mathfrak{A}$%
.}. Combining these two inclusions, we obtain $\mathfrak{A}=\mathfrak{B}$.

But we recall that%
\[
\mathfrak{A}=\left\{  \left(  i_{1},i_{2},\ldots,i_{n}\right)  \ \mid\ \left(
i_{1},i_{2},\ldots,i_{n}\right)  \text{ is a list of all elements of }V\text{
(with no repetitions)}\right\}  .
\]
Hence, we have the following equality of summation signs:%
\begin{align}
&  \sum_{\substack{\left(  i_{1},i_{2},\ldots,i_{n}\right)  \text{ is a
list}\\\text{of all elements of }V\\\text{(with no repetitions)}}}\nonumber\\
&  =\sum_{\substack{\left(  i_{1},i_{2},\ldots,i_{n}\right)  \in\mathfrak{A}%
}}\nonumber\\
&  =\sum_{\left(  i_{1},i_{2},\ldots,i_{n}\right)  \in\mathfrak{B}%
}\ \ \ \ \ \ \ \ \ \ \left(  \text{since }\mathfrak{A}=\mathfrak{B}\right)
\nonumber\\
&  =\sum_{\substack{\left(  i_{1},i_{2},\ldots,i_{n}\right)  \text{ is a
list}\\\text{of distinct elements of }V}}\nonumber\\
&  \ \ \ \ \ \ \ \ \ \ \left(  \text{since }\mathfrak{B}=\left\{  \left(
i_{1},i_{2},\ldots,i_{n}\right)  \ \mid\ \left(  i_{1},i_{2},\ldots
,i_{n}\right)  \text{ is a list of distinct elements of }V\right\}  \right)
\nonumber\\
&  =\sum_{\substack{i_{1},i_{2},\ldots,i_{n}\text{ are}\\\text{distinct}%
\\\text{elements of }V}}. \label{pf.cor.polar1.sumeq2}%
\end{align}

We have $X+\left(  -X\right)  =0$. Thus, the element $X+\left(  -X\right)  $
belongs to the center of $\mathbb{L}$ (since the element $0$ belongs to the
center of $\mathbb{L}$). Hence, Theorem \ref{thm.1} (applied to $Y=-X$) yields%
\begin{align*}
&  \sum_{S\subseteq V}\left(  X+\sum_{s\in S}x_{s}\right)  ^{\left\vert
S\right\vert }\left(  -X-\sum_{s\in S}x_{s}\right)  ^{n-\left\vert
S\right\vert }\\
&  =\sum_{\substack{i_{1},i_{2},\ldots,i_{k}\text{ are}\\\text{distinct}%
\\\text{elements of }V}}\left(  \underbrace{X+\left(  -X\right)  }%
_{=0}\right)  ^{n-k}x_{i_{1}}x_{i_{2}}\cdots x_{i_{k}}=\sum_{\substack{i_{1}%
,i_{2},\ldots,i_{k}\text{ are}\\\text{distinct}\\\text{elements of }V}%
}0^{n-k}x_{i_{1}}x_{i_{2}}\cdots x_{i_{k}}\\
&  =\underbrace{\sum_{\substack{i_{1},i_{2},\ldots,i_{k}\text{ are}%
\\\text{distinct}\\\text{elements of }V;\\n\leq k}}}_{\substack{=\sum
_{\substack{i_{1},i_{2},\ldots,i_{k}\text{ are}\\\text{distinct}%
\\\text{elements of }V;\\k=n}}\\\text{(by (\ref{pf.cor.polar1.sumeq1}))}%
}}0^{n-k}x_{i_{1}}x_{i_{2}}\cdots x_{i_{k}}+\sum_{\substack{i_{1},i_{2}%
,\ldots,i_{k}\text{ are}\\\text{distinct}\\\text{elements of }V;\\n>k}%
}\underbrace{0^{n-k}}_{\substack{=0\\\text{(since }n-k>0\\\text{(since
}n>k\text{))}}}x_{i_{1}}x_{i_{2}}\cdots x_{i_{k}}\\
&  \ \ \ \ \ \ \ \ \ \ \ \ \ \ \ \ \ \ \ \ \left(
\begin{array}
[c]{c}%
\text{since each }k\text{-tuple }\left(  i_{1},i_{2},\ldots,i_{k}\right)
\text{ of distinct elements of }V\\
\text{satisfies either }n\leq k\text{ or }n>k\text{ (but not both)}%
\end{array}
\right) \\
&  =\sum_{\substack{i_{1},i_{2},\ldots,i_{k}\text{ are}\\\text{distinct}%
\\\text{elements of }V;\\k=n}}0^{n-k}x_{i_{1}}x_{i_{2}}\cdots x_{i_{k}%
}+\underbrace{\sum_{\substack{i_{1},i_{2},\ldots,i_{k}\text{ are}%
\\\text{distinct}\\\text{elements of }V;\\n>k}}0x_{i_{1}}x_{i_{2}}\cdots
x_{i_{k}}}_{=0}\\
&  =\sum_{\substack{i_{1},i_{2},\ldots,i_{k}\text{ are}\\\text{distinct}%
\\\text{elements of }V;\\k=n}}0^{n-k}x_{i_{1}}x_{i_{2}}\cdots x_{i_{k}%
}=\underbrace{\sum_{\substack{i_{1},i_{2},\ldots,i_{n}\text{ are}%
\\\text{distinct}\\\text{elements of }V}}}_{\substack{=\sum_{\substack{\left(
i_{1},i_{2},\ldots,i_{n}\right)  \text{ is a list}\\\text{of all elements of
}V\\\text{(with no repetitions)}}}\\\text{(by (\ref{pf.cor.polar1.sumeq2}))}%
}}\underbrace{0^{n-n}}_{=0^{0}=1}x_{i_{1}}x_{i_{2}}\cdots x_{i_{n}}\\
&  =\sum_{\substack{\left(  i_{1},i_{2},\ldots,i_{n}\right)  \text{ is a
list}\\\text{of all elements of }V\\\text{(with no repetitions)}}}x_{i_{1}%
}x_{i_{2}}\cdots x_{i_{n}}.
\end{align*}
Comparing this with%
\begin{align*}
&  \sum_{S\subseteq V}\left(  X+\sum_{s\in S}x_{s}\right)  ^{\left\vert
S\right\vert }\left(  \underbrace{-X-\sum_{s\in S}x_{s}}_{=-\left(
X+\sum_{s\in S}x_{s}\right)  }\right)  ^{n-\left\vert S\right\vert }\\
&  =\sum_{S\subseteq V}\left(  X+\sum_{s\in S}x_{s}\right)  ^{\left\vert
S\right\vert }\underbrace{\left(  -\left(  X+\sum_{s\in S}x_{s}\right)
\right)  ^{n-\left\vert S\right\vert }}_{=\left(  -1\right)  ^{n-\left\vert
S\right\vert }\left(  X+\sum_{s\in S}x_{s}\right)  ^{n-\left\vert S\right\vert
}}\\
&  =\sum_{S\subseteq V}\underbrace{\left(  X+\sum_{s\in S}x_{s}\right)
^{\left\vert S\right\vert }\left(  -1\right)  ^{n-\left\vert S\right\vert }%
}_{=\left(  -1\right)  ^{n-\left\vert S\right\vert }\left(  X+\sum_{s\in
S}x_{s}\right)  ^{\left\vert S\right\vert }}\left(  X+\sum_{s\in S}%
x_{s}\right)  ^{n-\left\vert S\right\vert }\\
&  =\sum_{S\subseteq V}\left(  -1\right)  ^{n-\left\vert S\right\vert
}\underbrace{\left(  X+\sum_{s\in S}x_{s}\right)  ^{\left\vert S\right\vert
}\left(  X+\sum_{s\in S}x_{s}\right)  ^{n-\left\vert S\right\vert }%
}_{\substack{=\left(  X+\sum_{s\in S}x_{s}\right)  ^{\left\vert S\right\vert
+\left(  n-\left\vert S\right\vert \right)  }=\left(  X+\sum_{s\in S}%
x_{s}\right)  ^{n}\\\text{(since }\left\vert S\right\vert +\left(
n-\left\vert S\right\vert \right)  =n\text{)}}}\\
&  =\sum_{S\subseteq V}\left(  -1\right)  ^{n-\left\vert S\right\vert }\left(
X+\sum_{s\in S}x_{s}\right)  ^{n},
\end{align*}
we obtain%
\[
\sum_{S\subseteq V}\left(  -1\right)  ^{n-\left\vert S\right\vert }\left(
X+\sum_{s\in S}x_{s}\right)  ^{n}=\sum_{\substack{\left(  i_{1},i_{2}%
,\ldots,i_{n}\right)  \text{ is a list}\\\text{of all elements of
}V\\\text{(with no repetitions)}}}x_{i_{1}}x_{i_{2}}\cdots x_{i_{n}}.
\]
This proves Corollary \ref{cor.polar1}.
\end{proof}
\end{verlong}

Corollary \ref{cor.polar1} has a companion result:

\begin{corollary}
\label{cor.polar2}Let $V$ be a finite set. Let $n=\left\vert V\right\vert $.
For each $s\in V$, let $x_{s}$ be an element of $\mathbb{L}$. Let
$X\in\mathbb{L}$. Let $m\in\mathbb{N}$ be such that $m<n$. Then,%
\[
\sum_{S\subseteq V}\left(  -1\right)  ^{n-\left\vert S\right\vert }\left(
X+\sum_{s\in S}x_{s}\right)  ^{m}=0.
\]

\end{corollary}

\begin{proof}
[Proof of Corollary \ref{cor.polar2}.]For any subset $W$ of $V$, we define an
element $s\left(  W\right)  \in\mathbb{L}$ by%
\begin{equation}
s\left(  W\right)  =\sum_{\substack{\left(  i_{1},i_{2},\ldots,i_{k}\right)
\text{ is a list}\\\text{of all elements of }W\\\text{(with no repetitions)}%
}}x_{i_{1}}x_{i_{2}}\cdots x_{i_{k}}, \label{pf.cor.polar2.s(W)=}%
\end{equation}
where $k=\left\vert W\right\vert $.

If $W$ is any subset of $V$, and if $Y$ is any element of $\mathbb{L}$, then
we define an element $r\left(  Y,W\right)  \in\mathbb{L}$ by%
\begin{equation}
r\left(  Y,W\right)  =\sum_{S\subseteq W}\left(  -1\right)  ^{\left\vert
W\right\vert -\left\vert S\right\vert }\left(  Y+\sum_{s\in S}x_{s}\right)
^{m}. \label{pf.cor.polar2.r(Y,W)=}%
\end{equation}

Now, from Corollary \ref{cor.polar1}, we can easily deduce the following claim:

\begin{statement}
\textit{Claim 1:} Let $W$ be any subset of $V$ satisfying $\left\vert
W\right\vert =m$. Let $Y\in\mathbb{L}$. Then, $r\left(  Y,W\right)  =s\left(
W\right)  $.
\end{statement}

[\textit{Proof of Claim 1:} We have $m=\left\vert W\right\vert $. Hence,
Corollary \ref{cor.polar1} (applied to $W$, $Y$ and $m$ instead of $V$, $X$
and $n$) yields%
\[
\sum_{S\subseteq W}\left(  -1\right)  ^{m-\left\vert S\right\vert }\left(
Y+\sum_{s\in S}x_{s}\right)  ^{m}=\sum_{\substack{\left(  i_{1},i_{2}%
,\ldots,i_{m}\right)  \text{ is a list}\\\text{of all elements of
}W\\\text{(with no repetitions)}}}x_{i_{1}}x_{i_{2}}\cdots x_{i_{m}}=s\left(
W\right)
\]
(by (\ref{pf.cor.polar2.s(W)=}), applied to $k=m$). Thus,
(\ref{pf.cor.polar2.r(Y,W)=}) becomes%
\begin{align*}
r\left(  Y,W\right)   &  =\sum_{S\subseteq W}\underbrace{\left(  -1\right)
^{\left\vert W\right\vert -\left\vert S\right\vert }}_{\substack{=\left(
-1\right)  ^{m-\left\vert S\right\vert }\\\text{(since }\left\vert
W\right\vert =m\text{)}}}\left(  Y+\sum_{s\in S}x_{s}\right)  ^{m}\\
&  =\sum_{S\subseteq W}\left(  -1\right)  ^{m-\left\vert S\right\vert }\left(
Y+\sum_{s\in S}x_{s}\right)  ^{m}=s\left(  W\right)  .
\end{align*}
This proves Claim 1.]

A more interesting claim is the following:

\begin{statement}
\textit{Claim 2:} Let $W$ be a subset of $V$. Let $t\in W$. Let $Y\in
\mathbb{L}$. Then,%
\[
r\left(  Y,W\right)  =r\left(  Y+x_{t},W\setminus\left\{  t\right\}  \right)
-r\left(  Y,W\setminus\left\{  t\right\}  \right)  .
\]

\end{statement}

\begin{vershort}
[\textit{Proof of Claim 2:} The definition of $r\left(  Y,W\setminus\left\{
t\right\}  \right)  $ yields%
\begin{align}
r\left(  Y,W\setminus\left\{  t\right\}  \right)   &  =\underbrace{\sum
_{S\subseteq W\setminus\left\{  t\right\}  }}_{\substack{=\sum
_{\substack{S\subseteq W;\\t\notin S}}}}\underbrace{\left(  -1\right)
^{\left\vert W\setminus\left\{  t\right\}  \right\vert -\left\vert
S\right\vert }}_{\substack{=\left(  -1\right)  ^{\left\vert W\right\vert
-1-\left\vert S\right\vert }\\\text{(since }\left\vert W\setminus\left\{
t\right\}  \right\vert =\left\vert W\right\vert -1\\\text{(since }t\in
W\text{))}}}\left(  Y+\sum_{s\in S}x_{s}\right)  ^{m}\nonumber\\
&  =\sum_{\substack{S\subseteq W;\\t\notin S}}\left(  -1\right)  ^{\left\vert
W\right\vert -1-\left\vert S\right\vert }\left(  Y+\sum_{s\in S}x_{s}\right)
^{m}. \label{pf.cor.polar2.c2.pf.short.1}%
\end{align}
The same argument (applied to $Y+x_{t}$ instead of $Y$) yields%
\begin{align}
r\left(  Y+x_{t},W\setminus\left\{  t\right\}  \right)   &  =\sum
_{\substack{S\subseteq W;\\t\notin S}}\underbrace{\left(  -1\right)
^{\left\vert W\right\vert -1-\left\vert S\right\vert }}_{\substack{=\left(
-1\right)  ^{\left\vert W\right\vert -\left\vert S\cup\left\{  t\right\}
\right\vert }\\\text{(since }\left\vert S\right\vert =\left\vert S\cup\left\{
t\right\}  \right\vert -1\\\text{(because }t\notin S\text{))}}}\left(
Y+\underbrace{x_{t}+\sum_{s\in S}x_{s}}_{\substack{=\sum_{s\in S\cup\left\{
t\right\}  }x_{s}\\\text{(since }t\notin S\text{)}}}\right)  ^{m}\nonumber\\
&  =\sum_{\substack{S\subseteq W;\\t\notin S}}\left(  -1\right)  ^{\left\vert
W\right\vert -\left\vert S\cup\left\{  t\right\}  \right\vert }\left(
Y+\sum_{s\in S\cup\left\{  t\right\}  }x_{s}\right)  ^{m}.
\label{pf.cor.polar2.c2.pf.short.5}%
\end{align}

But the definition of $r\left(  Y,W\right)  $ yields%
\begin{align*}
&  r\left(  Y,W\right) \\
&  =\sum_{S\subseteq W}\left(  -1\right)  ^{\left\vert W\right\vert
-\left\vert S\right\vert }\left(  Y+\sum_{s\in S}x_{s}\right)  ^{m}\\
&  =\underbrace{\sum_{\substack{S\subseteq W;\\t\in S}}\left(  -1\right)
^{\left\vert W\right\vert -\left\vert S\right\vert }\left(  Y+\sum_{s\in
S}x_{s}\right)  ^{m}}_{\substack{=\sum_{\substack{S\subseteq W;\\t\notin
S}}\left(  -1\right)  ^{\left\vert W\right\vert -\left\vert S\cup\left\{
t\right\}  \right\vert }\left(  Y+\sum_{s\in S\cup\left\{  t\right\}  }%
x_{s}\right)  ^{m}\\\text{(here, we have substituted }S\cup\left\{  t\right\}
\text{ for }S\text{ in the sum,}\\\text{since the map }\left\{  S\subseteq
W\ \mid\ t\notin S\right\}  \rightarrow\left\{  S\subseteq W\ \mid\ t\in
S\right\}  ,\ S\mapsto S\cup\left\{  t\right\}  \\\text{is a bijection)}%
}}+\sum_{\substack{S\subseteq W;\\t\notin S}}\underbrace{\left(  -1\right)
^{\left\vert W\right\vert -\left\vert S\right\vert }}_{\substack{=-\left(
-1\right)  ^{\left\vert W\right\vert -\left\vert S\right\vert -1}\\=-\left(
-1\right)  ^{\left\vert W\right\vert -1-\left\vert S\right\vert }}}\left(
Y+\sum_{s\in S}x_{s}\right)  ^{m}\\
&  \ \ \ \ \ \ \ \ \ \ \ \ \ \ \ \ \ \ \ \ \left(
\begin{array}
[c]{c}%
\text{since each subset }S\text{ of }W\text{ satisfies either }t\in S\text{ or
}t\notin S\\
\text{(but not both)}%
\end{array}
\right) \\
&  =\underbrace{\sum_{\substack{S\subseteq W;\\t\notin S}}\left(  -1\right)
^{\left\vert W\right\vert -\left\vert S\cup\left\{  t\right\}  \right\vert
}\left(  Y+\sum_{s\in S\cup\left\{  t\right\}  }x_{s}\right)  ^{m}%
}_{\substack{=r\left(  Y+x_{t},W\setminus\left\{  t\right\}  \right)
\\\text{(by (\ref{pf.cor.polar2.c2.pf.short.5}))}}}-\underbrace{\sum
_{\substack{S\subseteq W;\\t\notin S}}\left(  -1\right)  ^{\left\vert
W\right\vert -1-\left\vert S\right\vert }\left(  Y+\sum_{s\in S}x_{s}\right)
^{m}}_{\substack{=r\left(  Y,W\setminus\left\{  t\right\}  \right)
\\\text{(by (\ref{pf.cor.polar2.c2.pf.short.1}))}}}\\
&  =r\left(  Y+x_{t},W\setminus\left\{  t\right\}  \right)  -r\left(
Y,W\setminus\left\{  t\right\}  \right)  .
\end{align*}
This proves Claim 2.]
\end{vershort}

\begin{verlong}
[\textit{Proof of Claim 2:} Let us recall the following (well-known and
simple) fact (which holds for any set $W$ and any element $t\in W$): The map%
\begin{align}
\left\{  S\subseteq W\ \mid\ t\notin S\right\}   &  \rightarrow\left\{
S\subseteq W\ \mid\ t\in S\right\}  ,\nonumber\\
S  &  \mapsto S\cup\left\{  t\right\}  \label{pf.cor.polar2.c2.pf.bij}%
\end{align}
is a bijection\footnote{Its inverse is the map
\begin{align*}
\left\{  S\subseteq W\ \mid\ t\in S\right\}   &  \rightarrow\left\{
S\subseteq W\ \mid\ t\notin S\right\}  ,\\
T  &  \mapsto T\setminus\left\{  t\right\}  .
\end{align*}
}.

The definition of $r\left(  Y,W\setminus\left\{  t\right\}  \right)  $ yields%
\begin{align}
r\left(  Y,W\setminus\left\{  t\right\}  \right)   &  =\underbrace{\sum
_{S\subseteq W\setminus\left\{  t\right\}  }}_{\substack{=\sum
_{\substack{S\subseteq W;\\t\notin S}}\\\text{(since the subsets of
}W\setminus\left\{  t\right\}  \\\text{are precisely the subsets }S\text{ of
}W\\\text{satisfying }t\notin S\text{)}}}\underbrace{\left(  -1\right)
^{\left\vert W\setminus\left\{  t\right\}  \right\vert -\left\vert
S\right\vert }}_{\substack{=\left(  -1\right)  ^{\left\vert W\right\vert
-1-\left\vert S\right\vert }\\\text{(since }\left\vert W\setminus\left\{
t\right\}  \right\vert =\left\vert W\right\vert -1\\\text{(since }t\in
W\text{))}}}\left(  Y+\sum_{s\in S}x_{s}\right)  ^{m}\nonumber\\
&  =\sum_{\substack{S\subseteq W;\\t\notin S}}\left(  -1\right)  ^{\left\vert
W\right\vert -1-\left\vert S\right\vert }\left(  Y+\sum_{s\in S}x_{s}\right)
^{m}. \label{pf.cor.polar2.c2.pf.1}%
\end{align}
The same argument (applied to $Y+x_{t}$ instead of $Y$) yields%
\begin{equation}
r\left(  Y+x_{t},W\setminus\left\{  t\right\}  \right)  =\sum
_{\substack{S\subseteq W;\\t\notin S}}\left(  -1\right)  ^{\left\vert
W\right\vert -1-\left\vert S\right\vert }\left(  Y+x_{t}+\sum_{s\in S}%
x_{s}\right)  ^{m}. \label{pf.cor.polar2.c2.pf.2}%
\end{equation}

But every subset $S$ of $W$ satisfying $t\notin S$ satisfies%
\begin{equation}
\left\vert W\right\vert -1-\left\vert S\right\vert =\left\vert W\right\vert
-\left\vert S\cup\left\{  t\right\}  \right\vert \label{pf.cor.polar2.c2.pf.3}%
\end{equation}
\footnote{\textit{Proof of (\ref{pf.cor.polar2.c2.pf.3}):} Let $S$ be a subset
of $W$ satisfying $t\notin S$. From $t\notin S$, we obtain $\left\vert
S\cup\left\{  t\right\}  \right\vert =\left\vert S\right\vert +1$. Hence,
$\left\vert W\right\vert -\underbrace{\left\vert S\cup\left\{  t\right\}
\right\vert }_{=\left\vert S\right\vert +1}=\left\vert W\right\vert -\left(
\left\vert S\right\vert +1\right)  =\left\vert W\right\vert -1-\left\vert
S\right\vert $. This proves (\ref{pf.cor.polar2.c2.pf.3}).} and
\begin{equation}
x_{t}+\sum_{s\in S}x_{s}=\sum_{s\in S\cup\left\{  t\right\}  }x_{s}
\label{pf.cor.polar2.c2.pf.4}%
\end{equation}
\footnote{\textit{Proof of (\ref{pf.cor.polar2.c2.pf.4}):} Let $S$ be a subset
of $W$ satisfying $t\notin S$. Then, $\left(  S\cup\left\{  t\right\}
\right)  \setminus\left\{  t\right\}  =S$ (since $t\notin S$) and
$t\in\left\{  t\right\}  \subseteq S\cup\left\{  t\right\}  $. Now,
\begin{align*}
\sum_{s\in S\cup\left\{  t\right\}  }x_{s}  &  =x_{t}+\underbrace{\sum
_{\substack{s\in S\cup\left\{  t\right\}  ;\\s\neq t}}}_{\substack{=\sum
_{s\in\left(  S\cup\left\{  t\right\}  \right)  \setminus\left\{  t\right\}
}=\sum_{s\in S}\\\text{(since }\left(  S\cup\left\{  t\right\}  \right)
\setminus\left\{  t\right\}  =S\text{)}}}x_{s}\\
&  \ \ \ \ \ \ \ \ \ \ \left(
\begin{array}
[c]{c}%
\text{here, we have split off the addend for }s=t\\
\text{from the sum (since }t\in S\cup\left\{  t\right\}  \text{)}%
\end{array}
\right) \\
&  =x_{t}+\sum_{s\in S}x_{s}.
\end{align*}
This proves (\ref{pf.cor.polar2.c2.pf.4}).}. Hence,
(\ref{pf.cor.polar2.c2.pf.2}) becomes%
\begin{align}
r\left(  Y+x_{t},W\setminus\left\{  t\right\}  \right)   &  =\sum
_{\substack{S\subseteq W;\\t\notin S}}\underbrace{\left(  -1\right)
^{\left\vert W\right\vert -1-\left\vert S\right\vert }}_{\substack{=\left(
-1\right)  ^{\left\vert W\right\vert -\left\vert S\cup\left\{  t\right\}
\right\vert }\\\text{(since }\left\vert W\right\vert -1-\left\vert
S\right\vert =\left\vert W\right\vert -\left\vert S\cup\left\{  t\right\}
\right\vert \\\text{(by (\ref{pf.cor.polar2.c2.pf.3})))}}}\left(
Y+\underbrace{x_{t}+\sum_{s\in S}x_{s}}_{\substack{=\sum_{s\in S\cup\left\{
t\right\}  }x_{s}\\\text{(by (\ref{pf.cor.polar2.c2.pf.4}))}}}\right)
^{m}\nonumber\\
&  =\sum_{\substack{S\subseteq W;\\t\notin S}}\left(  -1\right)  ^{\left\vert
W\right\vert -\left\vert S\cup\left\{  t\right\}  \right\vert }\left(
Y+\sum_{s\in S\cup\left\{  t\right\}  }x_{s}\right)  ^{m}.
\label{pf.cor.polar2.c2.pf.5}%
\end{align}

But the definition of $r\left(  Y,W\right)  $ yields%
\begin{align*}
&  r\left(  Y,W\right) \\
&  =\sum_{S\subseteq W}\left(  -1\right)  ^{\left\vert W\right\vert
-\left\vert S\right\vert }\left(  Y+\sum_{s\in S}x_{s}\right)  ^{m}\\
&  =\underbrace{\sum_{\substack{S\subseteq W;\\t\in S}}\left(  -1\right)
^{\left\vert W\right\vert -\left\vert S\right\vert }\left(  Y+\sum_{s\in
S}x_{s}\right)  ^{m}}_{\substack{=\sum_{\substack{S\subseteq W;\\t\notin
S}}\left(  -1\right)  ^{\left\vert W\right\vert -\left\vert S\cup\left\{
t\right\}  \right\vert }\left(  Y+\sum_{s\in S\cup\left\{  t\right\}  }%
x_{s}\right)  ^{m}\\\text{(here, we have substituted }S\cup\left\{  t\right\}
\text{ for }S\text{ in the sum,}\\\text{since the map
(\ref{pf.cor.polar2.c2.pf.bij}) is a bijection)}}}+\sum_{\substack{S\subseteq
W;\\t\notin S}}\underbrace{\left(  -1\right)  ^{\left\vert W\right\vert
-\left\vert S\right\vert }}_{\substack{=-\left(  -1\right)  ^{\left\vert
W\right\vert -\left\vert S\right\vert -1}\\=-\left(  -1\right)  ^{\left\vert
W\right\vert -1-\left\vert S\right\vert }\\\text{(since }\left\vert
W\right\vert -\left\vert S\right\vert -1=\left\vert W\right\vert -1-\left\vert
S\right\vert \text{)}}}\left(  Y+\sum_{s\in S}x_{s}\right)  ^{m}\\
&  \ \ \ \ \ \ \ \ \ \ \ \ \ \ \ \ \ \ \ \ \left(
\begin{array}
[c]{c}%
\text{since each subset }S\text{ of }W\text{ satisfies either }t\in S\text{ or
}t\notin S\\
\text{(but not both)}%
\end{array}
\right) \\
&  =\sum_{\substack{S\subseteq W;\\t\notin S}}\left(  -1\right)  ^{\left\vert
W\right\vert -\left\vert S\cup\left\{  t\right\}  \right\vert }\left(
Y+\sum_{s\in S\cup\left\{  t\right\}  }x_{s}\right)  ^{m}+\underbrace{\sum
_{\substack{S\subseteq W;\\t\notin S}}\left(  -\left(  -1\right)  ^{\left\vert
W\right\vert -1-\left\vert S\right\vert }\right)  \left(  Y+\sum_{s\in S}%
x_{s}\right)  ^{m}}_{=-\sum_{\substack{S\subseteq W;\\t\notin S}}\left(
-1\right)  ^{\left\vert W\right\vert -1-\left\vert S\right\vert }\left(
Y+\sum_{s\in S}x_{s}\right)  ^{m}}\\
&  =\underbrace{\sum_{\substack{S\subseteq W;\\t\notin S}}\left(  -1\right)
^{\left\vert W\right\vert -\left\vert S\cup\left\{  t\right\}  \right\vert
}\left(  Y+\sum_{s\in S\cup\left\{  t\right\}  }x_{s}\right)  ^{m}%
}_{\substack{=r\left(  Y+x_{t},W\setminus\left\{  t\right\}  \right)
\\\text{(by (\ref{pf.cor.polar2.c2.pf.5}))}}}+\left(  -\underbrace{\sum
_{\substack{S\subseteq W;\\t\notin S}}\left(  -1\right)  ^{\left\vert
W\right\vert -1-\left\vert S\right\vert }\left(  Y+\sum_{s\in S}x_{s}\right)
^{m}}_{\substack{=r\left(  Y,W\setminus\left\{  t\right\}  \right)
\\\text{(by (\ref{pf.cor.polar2.c2.pf.1}))}}}\right) \\
&  =r\left(  Y+x_{t},W\setminus\left\{  t\right\}  \right)  +\left(  -r\left(
Y,W\setminus\left\{  t\right\}  \right)  \right)  =r\left(  Y+x_{t}%
,W\setminus\left\{  t\right\}  \right)  -r\left(  Y,W\setminus\left\{
t\right\}  \right)  .
\end{align*}
This proves Claim 2.]
\end{verlong}

Now, the following is easy to show by induction:

\begin{statement}
\textit{Claim 3:} Let $W$ be a subset of $V$ satisfying $\left\vert
W\right\vert >m$. Let $Y\in\mathbb{L}$. Then, $r\left(  Y,W\right)  =0$.
\end{statement}

[\textit{Proof of Claim 3:} We shall prove Claim 3 by strong induction over
$\left\vert W\right\vert $:

\textit{Induction step:} Let $k\in\mathbb{N}$. Assume that Claim 3 is proven
in the case when $\left\vert W\right\vert <k$. We must show that Claim 3 holds
in the case when $\left\vert W\right\vert =k$.

We have assumed that Claim 3 is proven in the case when $\left\vert
W\right\vert <k$. In other words,%
\begin{equation}
\left(
\begin{array}
[c]{c}%
\text{if }W\text{ is any subset of }V\text{ satisfying }\left\vert
W\right\vert >m\text{ and }\left\vert W\right\vert <k\text{,}\\
\text{and if }Y\in\mathbb{L}\text{, then }r\left(  Y,W\right)  =0
\end{array}
\right)  . \label{pf.cor.polar2.c3.pf.indass}%
\end{equation}

Now, let $W$ be any subset of $V$ satisfying $\left\vert W\right\vert >m$ and
$\left\vert W\right\vert =k$. Let $Y\in\mathbb{L}$. We shall show that
$r\left(  Y,W\right)  =0$.

\begin{vershort}
We have $\left\vert W\right\vert >m\geq0$. Hence, there exists some $t\in W$.
Consider this $t$.
\end{vershort}

\begin{verlong}
We have $\left\vert W\right\vert >m\geq0$ (since $m\in\mathbb{N}$). Hence, the
set $W$ is nonempty. In other words, there exists some $t\in W$. Consider this
$t$. Clearly, $W\setminus\left\{  t\right\}  $ is a subset of $V$ (since
$W\setminus\left\{  t\right\}  \subseteq W\subseteq V$).
\end{verlong}

\begin{vershort}
From $t\in W$, we obtain $\left\vert W\setminus\left\{  t\right\}  \right\vert
=\underbrace{\left\vert W\right\vert }_{>m}-1>m-1$, so that $\left\vert
W\setminus\left\{  t\right\}  \right\vert \geq m$. Thus, we are in one of the
following two cases:
\end{vershort}

\begin{verlong}
From $\left\vert W\right\vert >m$, we obtain $\left\vert W\right\vert \geq
m+1$ (since $\left\vert W\right\vert $ and $m$ are integers).

From $t\in W$, we obtain $\left\vert W\setminus\left\{  t\right\}  \right\vert
=\underbrace{\left\vert W\right\vert }_{\geq m+1}-\,1\geq\left(  m+1\right)
-1=m$. Thus, we are in one of the following two cases:
\end{verlong}

\textit{Case 1:} We have $\left\vert W\setminus\left\{  t\right\}  \right\vert
=m$.

\textit{Case 2:} We have $\left\vert W\setminus\left\{  t\right\}  \right\vert
>m$.

Let us first consider Case 1. In this case, we have $\left\vert W\setminus
\left\{  t\right\}  \right\vert =m$. Hence, Claim 1 (applied to $W\setminus
\left\{  t\right\}  $ instead of $W$) yields $r\left(  Y,W\setminus\left\{
t\right\}  \right)  =s\left(  W\setminus\left\{  t\right\}  \right)  $. Also,
Claim 1 (applied to $W\setminus\left\{  t\right\}  $ and $Y+x_{t}$ instead of
$W$ and $Y$) yields $r\left(  Y+x_{t},W\setminus\left\{  t\right\}  \right)
=s\left(  W\setminus\left\{  t\right\}  \right)  $. Now, Claim 2 yields%
\[
r\left(  Y,W\right)  =\underbrace{r\left(  Y+x_{t},W\setminus\left\{
t\right\}  \right)  }_{=s\left(  W\setminus\left\{  t\right\}  \right)
}-\underbrace{r\left(  Y,W\setminus\left\{  t\right\}  \right)  }_{=s\left(
W\setminus\left\{  t\right\}  \right)  }=s\left(  W\setminus\left\{
t\right\}  \right)  -s\left(  W\setminus\left\{  t\right\}  \right)  =0.
\]
Thus, $r\left(  Y,W\right)  =0$ is proven in Case 1.

Let us now consider Case 2. In this case, we have $\left\vert W\setminus
\left\{  t\right\}  \right\vert >m$. Also, $\left\vert W\setminus\left\{
t\right\}  \right\vert =\left\vert W\right\vert -1<\left\vert W\right\vert
=k$. Hence, (\ref{pf.cor.polar2.c3.pf.indass}) (applied to $W\setminus\left\{
t\right\}  $ instead of $W$) yields $r\left(  Y,W\setminus\left\{  t\right\}
\right)  =0$. Also, (\ref{pf.cor.polar2.c3.pf.indass}) (applied to
$W\setminus\left\{  t\right\}  $ and $Y+x_{t}$ instead of $W$ and $Y$) yields
$r\left(  Y+x_{t},W\setminus\left\{  t\right\}  \right)  =0$. Now, Claim 2
yields%
\[
r\left(  Y,W\right)  =\underbrace{r\left(  Y+x_{t},W\setminus\left\{
t\right\}  \right)  }_{=0}-\underbrace{r\left(  Y,W\setminus\left\{
t\right\}  \right)  }_{=0}=0-0=0.
\]
Thus, $r\left(  Y,W\right)  =0$ is proven in Case 2.

\begin{vershort}
We have now proven $r\left(  Y,W\right)  =0$ in each of the two Cases 1 and 2.
Hence, $r\left(  Y,W\right)  =0$ always holds.
\end{vershort}

\begin{verlong}
We have now proven $r\left(  Y,W\right)  =0$ in each of the two Cases 1 and 2.
Since these two Cases cover all possibilities, we thus conclude that $r\left(
Y,W\right)  =0$ always holds.
\end{verlong}

Now, let us forget that we fixed $W$ and $Y$. We thus have proven that if $W$
is any subset of $V$ satisfying $\left\vert W\right\vert >m$ and $\left\vert
W\right\vert =k$, and if $Y\in\mathbb{L}$, then $r\left(  Y,W\right)  =0$. In
other words, Claim 3 holds in the case when $\left\vert W\right\vert =k$. This
completes the induction step. Thus, Claim 3 is proven by strong induction.]

Now, recall that $V$ is a subset of $V$ satisfying $\left\vert V\right\vert
=n>m$ (since $m<n$). Hence, Claim 3 (applied to $W=V$ and $Y=X$) yields
$r\left(  X,V\right)  =0$. But the definition of $r\left(  X,V\right)  $
yields%
\[
r\left(  X,V\right)  =\sum_{S\subseteq V}\underbrace{\left(  -1\right)
^{\left\vert V\right\vert -\left\vert S\right\vert }}_{\substack{=\left(
-1\right)  ^{n-\left\vert S\right\vert }\\\text{(since }\left\vert
V\right\vert =n\text{)}}}\left(  X+\sum_{s\in S}x_{s}\right)  ^{m}%
=\sum_{S\subseteq V}\left(  -1\right)  ^{n-\left\vert S\right\vert }\left(
X+\sum_{s\in S}x_{s}\right)  ^{m}.
\]
Hence,%
\[
\sum_{S\subseteq V}\left(  -1\right)  ^{n-\left\vert S\right\vert }\left(
X+\sum_{s\in S}x_{s}\right)  ^{m}=r\left(  X,V\right)  =0.
\]
This proves Corollary \ref{cor.polar2}.
\end{proof}

\section{Questions}

The above results are not the first generalizations of the classical
Abel--Hurwitz identities; there are various others. In particular,
generalizations appear in \cite{Strehl92}, \cite[Theorem 1.2]{Carrel14},
\cite{Johns96}, \cite{Kalai79}, \cite{Pitman02}, \cite[\S 1.6]{Riorda68} (see
the end of \cite{qedmo09} for some of these) and \cite{KelPos}. We have not
tried to lift these generalizations into our noncommutative setting, but we
suspect that this is possible.

\end{document}